\documentclass[11pt,a4paper,leqno]{amsart}

\usepackage[latin1]{inputenc}
\usepackage[T1]{fontenc}
\usepackage{amsfonts}
\usepackage{amsmath}
\usepackage{amssymb}
\usepackage{eurosym}
\usepackage{mathrsfs}
\usepackage{palatino}
\usepackage{color}
\usepackage{esint}
\usepackage{url}
\usepackage{verbatim}

\usepackage{enumerate}

\usepackage[pagebackref,hypertexnames=false, colorlinks, citecolor=blue, linkcolor=blue, urlcolor=red]{hyperref}

\newcommand{\R}{\mathbb{R}}
\newcommand{\C}{\mathbb{C}}

\newcommand{\N}{\mathbb{N}}

\newcommand{\Z}{\mathbb{Z}}
\newcommand{\E}{\mathbb{E}}

\newcommand{\calM}{\mathcal{M}}
\newcommand{\calS}{\mathcal{S}}

\newcommand{\calI}{\mathcal{I}}

\newcommand{\calU}{\mathcal{U}}

\newcommand{\bbP}{\mathbb{P}}

\numberwithin{equation}{section}

\newcommand{\ud}[0]{\,\mathrm{d}}

\newcommand{\dist}[0]{\operatorname{dist}}

\newcommand{\abs}[1]{|#1|}
\newcommand{\Babs}[1]{\Big|#1\Big|}
\newcommand{\norm}[2]{|#1|_{#2}}

\newcommand{\Norm}[2]{\|#1\|_{#2}}

\newcommand{\pair}[2]{\langle #1,#2 \rangle}

\newcommand{\ave}[1]{\langle #1\rangle}
\newcommand{\bave}[1]{\big\langle #1\big\rangle}

\newcommand{\BMO}[0]{\operatorname{BMO}}
\newcommand{\supp}[0]{\operatorname{spt}}

\newcommand{\sign}[0]{\operatorname{sgn}}

\newcommand{\eps}[0]{\varepsilon}

\newcommand{\ch}[0]{\operatorname{ch}}

\newcommand{\calD}[0]{\mathcal{D}}

\newcommand{\wt}[1]{{\widetilde{#1}}}
\newcommand{\wh}[1]{{\widehat{#1}}}
\newcommand{\sub}[0]{\operatorname{sub}}

\swapnumbers
\theoremstyle{plain}
\newtheorem{thm}[equation]{Theorem}
\newtheorem{lem}[equation]{Lemma}
\newtheorem{prop}[equation]{Proposition}
\newtheorem{cor}[equation]{Corollary}

\theoremstyle{definition}
\newtheorem{defn}[equation]{Definition}

\theoremstyle{remark}
\newtheorem{rem}[equation]{Remark}

\title{Multiresolution analysis and Zygmund dilations}

\author{Tuomas Hyt\"onen}
\author{Kangwei Li}
\author{Henri Martikainen}
\author{Emil Vuorinen}

\address[T.H. \& E.V.]{Department of Mathematics and Statistics, University of Helsinki, P.O.B. 68, FI-00014 University of Helsinki, Finland}
\email{tuomas.hytonen@helsinki.fi}
\email{emil.vuorinen@helsinki.fi}

\address[K.L.]{Center for Applied Mathematics, Tianjin University, Weijin Road 92, 300072 Tianjin, China}
\email{kli@tju.edu.cn}

\address[H.M.]{Department of Mathematics and Statistics, Washington University in St. Louis, 1 Brookings Drive, St. Louis, MO 63130, USA}
\email{henri@wustl.edu}

\makeatletter
\@namedef{subjclassname@2020}{%
  \textup{2020} Mathematics Subject Classification}
\makeatother


\begin{document}

\allowdisplaybreaks

\begin{abstract}
Zygmund dilations are a group of dilations lying in between the standard product theory
and the one-parameter setting -- in $\R^3 = \R \times \R \times \R$
they are the dilations $(x_1, x_2, x_3) \mapsto (\delta_1 x_1, \delta_2 x_2, \delta_1 \delta_2 x_3)$. The dyadic multiresolution analysis and the related
dyadic-probabilistic methods
have been very impactful in the modern product singular integral theory. However, multiresolution analysis has not been understood in the Zygmund dilation setting
or in other modified product space settings.
In this paper we develop this missing dyadic multiresolution analysis of Zygmund type, and justify its usefulness
by bounding, on weighted spaces, a general class of singular integrals that are invariant under Zygmund dilations. We provide
novel examples of Zygmund $A_p$ weights and Zygmund kernels showcasing the optimality of our kernel assumptions
for weighted estimates.
\end{abstract}

\maketitle
\tableofcontents

\section{Introduction}

\subsection{General background: singular integrals adapted to different dilation structures}

Singular integral operators (SIOs) on $\R^d$ take the general form 
\begin{equation}\label{eq:CZO}
  Tf(x) = \int_{\R^d} K(x,y)f(y)\ud y,
\end{equation}
where $K(x,y)$ is called the {\em kernel} of $T$.

\subsubsection{Classical theory}
Classical theory of Calder\'on--Zygmund \cite{CZ} addresses kernels that are invariant under the basic transformations of the ambient space: translations and positive dilations.
By translation invariance we mean that $K(x+h,y+h)=K(x,y)$ for all $x,y,h\in\R^d$, and thus in fact $K(x,y)=K(x-y)$ is a {\em convolution kernel}.
Invariance under positive dilations refers to the condition that $\delta^d K(\delta x,\delta y)=K(x,y)$ for all $x,y\in\R^d$ and $\delta>0$; in combination with translation invariance this implies that $K(x,y)=\Omega(x-y)\abs{x-y}^{-d}$ for some $\Omega(u)$ that depends only on the angular part $u/\abs{u}$ of its argument.
Under appropriate regularity and cancellation hypotheses on $\Omega$, \cite{CZ} obtains the $L^p$ boundedness of $f\mapsto Tf$ given by \eqref{eq:CZO}. Related results in a ``dual'' representation of the same operators as {\em Fourier multipliers} are due to Mihlin \cite{Mih}, and a certain unification of the two points of view is provided by H\"ormander \cite{Hor}.

It is by now well-known that neither translation nor dilation invariance is essential for the theory on the level of individual operators. Dilation invariance was already abolished in \cite{Hor,Mih}, and translation invariance via the general theory first developed by Coifman--Weiss \cite{CW} and Coifman--Meyer \cite{CM}, and to some extent completed by David--Journ\'e \cite{DJ} whose ``$T(1)$ theorem'' provides a general $L^2$ boundedness criterion of SIOs without either of the two invariances. Nevertheless, much of the theory (rather naturally) deals with classes of operators, such that both translates and dilates of an operator, while not necessarily equal to the original operator, are nevertheless operators of the same class.
Thus, any assumptions imposed on the kernel will obey these invariances, and, in particular, the unique pointwise upper bound on $K(x,y)$ with this property takes the form
\begin{equation}\label{eq:CZK1param}
  \abs{K(x,y)}\lesssim\abs{x-y}^{-d};
\end{equation}
additional regularity and cancellation hypotheses with similar invariances (as in \cite{CZ}, and often modelled after \cite{CZ}) are needed to obtain any interesting conclusions.

Attempting anything even close to a survey of the extensive Calder\'on--Zygmund theory initiated by \cite{CZ} would take us much too far afield, but two further aspects are directly relevant to our present investigation: the (classical) description of the {\em weighted norm inequalities} satisfied by these operators \cite{CF} and by the closely related maximal operator \cite{Muc}, as well as the (more recent) {\em dyadic representation theorem} \cite{Hy}, which allows one to decompose any standard SIO into simpler constituents amenable to further analysis.

\subsubsection{Product space theory}
There is also a well-developed theory of more general $m$-parameter dilations
\begin{equation}\label{eq:prod-dil}
  (x_1, \ldots, x_m) \mapsto (\delta_1 x_1, \ldots, \delta_m x_m), \qquad \delta_1, \ldots, \delta_m > 0,
\end{equation}
when $\R^d$ is viewed as the {\em product space}
\begin{displaymath}
  \R^d = \prod_{i=1}^m \R^{d_i}, \qquad d = d_1 + \cdots + d_m.
\end{displaymath}
The upper bound on $K(x,y)$ that is invariant under these dilations takes the form
\begin{equation}\label{eq:CZKmparam}
  \abs{K(x,y)}\lesssim\prod_{i=1}^m\abs{x_i-y_i}^{-d_i},
\end{equation}
which is a much weaker assumption than the classical \eqref{eq:CZK1param}. The theory of Fourier multipliers in this setting by Marcinkiewicz \cite{Mar} actually predates all other results mentioned in this Introduction. A theory of convolution kernels $K(x-y)$ of this type, and weighted norm inequalities for both these SIOs and the related {\em strong maximal operator}, has been developed in \cite{FS}, extending the classical theory of \cite{CZ,CF,Muc} to the $m$-parameter case. Moreover, $m$-parameter extensions of the $T(1)$ theorem \cite{DJ} and the dyadic representation theorem \cite{Hy} have been obtained in \cite{Jo2} and \cite{Ma1}, respectively. Thus, at least in terms of these aspects of the theory, the $m$-parameter results are essentially on the same level with their classical counterparts.

\subsubsection{Entangled dilation systems}
Yet a further extension of the basic set-up consists of ``entangled'' systems of dilations
\begin{equation}\label{eq:NWdil}
  (x_1, \ldots, x_m) \mapsto (\delta_1^{\lambda_{11}}\cdots\delta_k^{\lambda_{1k}} x_1, \ldots, \delta_1^{\lambda_{m1}}\cdots\delta_k^{\lambda_{mk}} x_m), \qquad \delta_1, \ldots, \delta_k > 0,
\end{equation}
described by a fixed {\em dilation matrix}
\begin{equation*}
  \Lambda=\begin{pmatrix} {\lambda_{11}} & \ldots & {\lambda_{1k}} \\ \vdots & \ddots & \vdots \\ {\lambda_{m1}} & \ldots & {\lambda_{mk}} \end{pmatrix}.
\end{equation*}

A pioneering contribution in this direction is due to Nagel--Wainger \cite{NW}, who obtain the $L^2$ boundedness of a class of operators with kernels invariant under both translations and dilations of type \eqref{eq:NWdil}, thus extending the scope of the classical Calder\'on--Zygmund results \cite{CZ} to the general dilations \eqref{eq:NWdil}. Later, Ricci--Stein \cite{RS} deal with $L^p$ boundedness and more general kernels, still translation-invariant, but relaxing the invariance under \eqref{eq:NWdil} on the level of individual kernels to invariance of the assumptions imposed on them---again, in parallel with the development of the classical theory. In terms of limits of this theory, it seems that neither a weighted theory nor a $T(1)$ criterion for non-convolution operators is available on this level of generality; moreover, the kernels covered by the theory are either required to have strong invariance properties as in \cite{NW}, or they are somewhat indirectly described in terms of a suitable series representation, as in \cite{RS}.

\subsection{Zygmund dilations; existing results in this setting}
Further advances have been achieved in the case of specific individual dilation groups, among which the group of {\em Zygmund dilations}
\begin{equation}\label{eq:Zdil}
  (x_1, x_2, x_3) \mapsto (\delta_1 x_1, \delta_2 x_2, \delta_1 \delta_2 x_3), \qquad \delta_1, \delta_2 > 0,
\end{equation}
provides the simplest nontrivial example of interest.  This is the setting that we will concentrate on in this paper. A concrete example of a kernel invariant under \eqref{eq:Zdil} is the {\em Nagel--Wainger kernel}
\begin{equation*}
  K(x_1, x_2, x_3) = \frac{\sign(x_1x_2)}{x_1^2 x_2^2 + x_3^2};
\end{equation*}
the $L^2$ boundedness of the corresponding SIO is a special case of the results of \cite{NW}.

A particular interest in the Zygmund dilations \eqref{eq:Zdil} stems from the fact that, in the usual representation of the Heisenberg group on $\R^3$, the group law
\begin{equation*}
  (x_1,x_2,x_3)\odot(y_1,y_2,y_3)=\Big(x_1+y_1,x_2+y_2,x_3+y_3+\alpha(x_1y_2-y_1x_2)\Big)
\end{equation*}
is compatible with these dilations. (There are different conventions about the choice of the parameter $\alpha\in\R\setminus\{0\}$ above, but this choice does not affect the mentioned compatibility.) See M\"uller--Ricci--Stein \cite{MRS} for results and further motivation on Fourier multipliers of Marcinkiewicz type in this setting.

However, the starting point of the line of investigation that we continue in this paper is the work of R.\ Fefferman and Pipher \cite{FP}. They studied a class of convolution-type SIOs compatible with \eqref{eq:Zdil}, connecting both convolution kernel and Fourier multiplier points of view, and, importantly, they identified the weighted norm inequalities relevant for this class of operators, and the related maximal operator compatible with \eqref{eq:Zdil}. The condition on the weights and the maximal operator are given by the familiar formula
\begin{equation*}
\begin{split}
    [w]_{A_{p,Z}} &:=\sup_{I\in\mathcal{R}_Z}\Big(\fint_I w(x)\ud x\Big)\Big(\fint_I w^{-1/(p-1)}(x) \ud x\Big)^{p-1}<\infty, \\
    M_Zf(x) &:=\sup_{I\in\mathcal{R}_Z} 1_I(x)\fint_I\abs{f(y)}\ud y,
\end{split} 
\end{equation*}
where the supremums are taken over the collection
\begin{equation*}
  \mathcal{R}_Z:=\{I_1\times I_2\times I_3: \text{each } I_i\subset\R\text{ is an interval}, \ell(I_1)\ell(I_2)=\ell(I_3)\}
\end{equation*}
of {\em Zygmund rectangles}, which is invariant under \eqref{eq:Zdil}. A significant point is that the class of weights $A_{p,Z}$ is strictly larger than the class of {\em strong $A_p$ weights}, which is relevant in the product-space theory. Thus, the strategy of \cite{RS}, of reducing operators with an entangled dilation structure \eqref{eq:NWdil} into ones with a simpler product dilation structure \eqref{eq:prod-dil}, is out of question, if one seeks for the maximal generality of weighted norm inequalities for SIOs compatible with \eqref{eq:Zdil}.

A certain limitation of the results of \cite{FP}, when compared to the classical and the product space theory of SIOs, is that \cite{FP} assumes an unspecified ``large enough'' degree of smoothness of the kernels (and Fourier multipliers) that they consider. This limitation was recently addressed by Han et al. \cite{HLLT} who introduced a new set of \eqref{eq:Zdil}-invariant conditions on convolution kernels, involving only minimal H\"older-continuity type regularity assumptions analogous to those in the classical theory, and proved that the corresponding class of SIOs is well-defined and bounded on (the unweighted) $L^p$. (See also \cite{HLLTW} for related end-point estimates, although these go to a somewhat different direction than our presents goals.) In particular, the kernel size bound of \cite{HLLT} takes the form
\begin{equation}\label{eq:CZZK}
   \abs{K(x,y)}\lesssim\frac{D_\theta(x-y)}{\abs{x_1-y_1}\abs{x_2-y_2}\abs{x_3-y_3}},
\end{equation}
where the denominator is the same as that in the product theory \eqref{eq:CZKmparam} (with $d=m=3$ and $d_1=d_2=d_3=1$), but there is an additional decay
\begin{equation*}
  D_\theta(z):=\Bigg(\frac{\abs{z_1 z_2}}{\abs{z_3}}+\frac{\abs{z_3}}{\abs{z_1 z_2}}\Bigg)^{-\theta},
\end{equation*}
for some $\theta\in(0,1]$, in terms of the deviation of $z\in\R^3$ from the ``Zygmund manifold'' $\abs{z_1z_2}=\abs{z_3}$.

In view of the $A_{p,Z}$-weighted theory for smooth enough kernels due to \cite{FP}, and the unweighted theory under minimal kernel assumptions by \cite{HLLT}, it might seem natural to expect an $A_{p,Z}$-weighted theory under the same minimal kernel assumptions in the Zygmund dilation setting. Such theory has recently been considered in \cite{DLOPW2}; however, there seems to be some issues as the weighted estimates, as we will see, are not, in fact, true with minimal kernel assumptions.

\subsection{Contributions of this paper}

Perhaps surprisingly, as we will see in this paper, there is a delicate threshold in the kernel parameter $\theta$ in \eqref{eq:CZZK} concerning their compatibility with $A_{p,Z}$-weighted norm inequalities: For kernels with $\theta=1$ (possibly with a logarithmic correction necessary to cover some examples), we achieve a full weighted theory with the same class of weights $A_{p,Z}$ as in \cite{FP}. For kernels with $\theta\in(0,1)$, we still have weighted norm inequalities in the smaller class of strong $A_p$ weights (consistent, in particular, with the unweighted results due to \cite{HLLT}), but we also provide explicit counterexamples to the boundedness in the full class of $A_{p,Z}$ weights whenever $\theta<1$, see Proposition \ref{prop:mainCounterex}.

An additional interest in our counterexample is the fact that it depends on new explicit examples of $A_{p,Z}$ weights. The previous lack of such examples in the literature has been a limitation of the existing theory.
The following class of examples is a Zygmund dilation analogue of the well-known power weights in the classical theory:

\begin{lem}\label{lem:zygexa}
  Let $p\in(1,\infty)$, and let $w:\R^3\to[0,\infty)$ be given by
  \begin{equation*}
    w(x):=\begin{cases}\abs{x_1x_2}^{\alpha-\gamma}\abs{x_3}^\gamma, & \text{if}\quad\abs{x_3}\leq\abs{x_1x_2}, \\
    \abs{x_1x_2}^{\alpha-\delta}\abs{x_3}^\delta, & \text{if}\quad\abs{x_3}\geq\abs{x_1x_2}.\end{cases}
  \end{equation*}
  Then $w\in A_{p, Z}$ if and only if
  \begin{equation*}
      \gamma\in(-1,p-1),\quad\alpha\in(\gamma-1,\gamma+p-1),\quad\delta\in(\alpha-(p-1),\alpha+1). 
  \end{equation*}
  \end{lem}

\begin{proof}
We give the proof of this after Lemma \ref{lem:zrec} in the body of the paper.
\end{proof}

Besides describing rather precise limits for the weighted theory of the class of singular convolution operators introduced by \cite{HLLT}, we actually go a step further by dropping the translation-invariance assumption on the operators present in all previous related contributions \cite{DLOPW2,HLLT,HLLTW}. In other words, for the first time in the Zygmund dilation setting, we deal with kernels of the general form $K(x,y)$ instead of $K(x-y)$, and we obtain a (special) $T(1)$-type theorem in this framework. The assumptions on the kernel consists of a natural generalisation of the conditions of \cite{HLLT}, and indeed we will check (see Proposition \ref{E:eqlem6}) that the operators considered in \cite{DLOPW2,HLLT,HLLTW} are a special case of those that we deal with. (For kernel size and smoothness, this will be fairly obvious, but the connections between the cancellation assumptions in the two settings are not as straightforward.)

An advantage of replacing convolution kernels $K(x-y)$ by the general form $K(x,y)$ is that the latter make sense also on ambient domains without any translational structure. Thus, although we still formulate the results of the present paper on the Euclidean domain $\R^3$ only, our set-up and methods should be amenable to a further extension to much more general ambient domains, in analogy with the well-known extension of the classical Calder\'on--Zygmund theory to {\em spaces of homogeneous type} \cite{CW}.

Achieving this level of generality depends on a methodological change of paradigm compared to the previous contributions \cite{DLOPW2,FP,HLLT,HLLTW}  in the Zygmund dilation setting. In short, our approach consists of developing a {\em multiresolution analysis}, and more specifically a {\em dyadic representation theorem}, in this framework. These techniques have been highly influential in recent advances on SIOs and related questions of  geometric measure theory (see e.g. \cite{HPW, Hy, Ma1, NTV, To:book}).
However, multiresolution analysis, dyadic or otherwise, was not previously understood in the Zygmund dilation setting
or in other modified product space settings.
In fact, in \cite{DLOPW2} the authors also write that (emphasis added):
\emph{"[In] this specific multi-parameter setting with Zygmund dilations, most [classical methods] do not seem to apply. [\ldots ] the multiresolution analysis is not understood in this setting and it is unknown whether there is a suitable wavelet basis."}
We develop these missing multiresolution methods and also bring the related dyadic-probabilistic apparatus
to the Zygmund setting. The methods are flexible and will lend themselves to further generalizations.

Postponing the detailed definitions until later, we now give a conceptual formulation of our main results as follows. For detailed statements, we refer the reader to Theorem \ref{E:thm1} for the representation formula \eqref{eq:repr-intro}, to Corollary \ref{E:cor1} for the positive directions in both boundedness claims, and to Proposition \ref{prop:mainCounterex} for the negative statement in \eqref{it:theta<1}.

\begin{thm}\label{thm:intro}
  Let $T$ be a singular integral of Zygmund type as defined in Section \ref{E:sec1}, with parameter $\theta\in(0,1]$ in the kernel bound \eqref{eq:CZZK} and the related smoothness and cancellation assumptions in Section \ref{E:sec1}.
  Then $T$ is an average, over random Zygmund lattices,
  of the dyadic model operators $Q_{k}$ of Definition \ref{E:defn1}, more specifically, we have 
  for suitable constants $c_{k, \theta}$ that
\begin{equation}\label{eq:repr-intro}
    \ave{Tf,g}
    = C\E \sum_{k \in \N^3} c_{k, \theta} \langle Q_{k} f, g \rangle.
\end{equation}
  \begin{enumerate}[\rm(a)]
    \item\label{it:theta<1} If $\theta\in(0,1)$, the operator $T$ is bounded in unweighted $L^p$ spaces (and more generally on $L^p(w)$ with strong $A_p$ weights),
    but $L^p(w)$ estimates with $w\in A_{p,Z}$ may fail.
    \item\label{it:theta=1}  If $\theta = 1$, possibly with a suitable logarithmic growth factor,
    then for every $p\in(1,\infty)$ and every weight $w\in A_{p,Z}$, the operator $T$ extends boundedly to $L^p(w)$.
  \end{enumerate}
\end{thm}

See also our paper \cite{HLMV:CZX}, where we isolate the conditions on the lower-dimensional kernels obtained by fixing the variables $x_1,y_1$ in the Zygmund setting
and study the related operators on $\R^2$. In particular, we study whether these lower-dimensional kernels
should be considered as one-parameter or bi-parameter, or something strictly in between them.

\subsection{About the proof}
As has been done before, we view Zygmund dilations as certain kind of modified bi-parameter theory on $\R^3 = \R \times \R^2$.
There are really two different classes of Zygmund SIOs, one is modified bi-parameter theory w.r.t. to the parameters
$\{1\}$ and $\{2,3\}$ and the second class w.r.t to the parameters $\{2\}$ and $\{1,3\}$. Already in \cite{FP}
general Ricci--Stein operators \cite{RS} were split into two different multipliers satisfying the above. 

We let $\calD = \prod_{m=1}^3 \calD^m$,
where $\calD^1$, $\calD^2$ and $\calD^3$ are dyadic grids on $\R$, and
define the Zygmund rectangles $I = I^1 \times I^2 \times I^3 \in \calD_Z \subset \calD$
by insisting that $\ell(I^1) \ell(I^2) = \ell(I^3)$.
Given $I = \prod_{m=1}^3 I^m$ we define the Zygmund martingale difference operator
\begin{equation*}
  \Delta_{I,Z} f := \Delta_{I^1} \Delta_{I^2 \times I^3} f,
\end{equation*}
where $\Delta_{I^2 \times I^3} \ne \Delta_{I^2} \Delta_{I^3}$ is the \textbf{one-parameter} martingale difference on the rectangle $I^2 \times I^3$.
For a dyadic $\lambda > 0$ we define the dilated lattices
$$
  \calD_{\lambda}^{2,3} = \{I^2 \times I^3 \in \calD^{2,3} := \calD^2 \times \calD^3 \colon \ell(I^3) = \lambda \ell(I^2)\}
$$
and then expand
\begin{equation*}
  \begin{split}
    f = \sum_{I^1 \in \calD^1} \Delta_{I^1} f = \sum_{I^1 \in \calD^1} \sum_{I^2 \times I^3 \in \calD^{2,3}_{\ell(I^1)}} \Delta_{I^1} \Delta_{I^2 \times I^3}  f
    = \sum_{I \in \calD_Z} \Delta_{I, Z} f.
  \end{split}
\end{equation*}
Similar definitions can be made that are useful for the Zygmund SIOs that
have structure in the parameters $\{2\}$ and $\{1,3\}$. This fundamental and simple decomposition, however,
does not, on its own, appear to be sufficiently sophisticated to study bilinear forms $\pair{Tf}{g}$ induced by operators $T$
invariant under Zygmund dilations.

Given an SIO $T$ of Zygmund type, we could at this point write the decomposition
$$
\langle Tf, g \rangle = 
\sum_{I, J \in \calD_Z} \pair{T \Delta_{I, Z}f}{\Delta_{J, Z} g}.
$$
However, if we wish that the Zygmund scaling continues to hold for the rectangles related to functions
associated with partial adjoints of $T$,
the only way that this can be the case is to have $$
\ell(I^u)=\ell(J^u), \qquad u=1,2,3.
$$
The above simple Zygmund resolution of $T$ does not give us this property, and requires a substantial 
modification, see \eqref{eq:dec1} through \eqref{eq:dec5}. 

With the starting resolution established, there are still significant levels of freedom on how to organize the martingale structure.
In particular, given $I, J \in \calD_Z$ with $\ell(I^u)=\ell(J^u)$ for $u=1,2,3$ we could seek for a common parent $K \supset I \cup J$ that is itself still
a Zygmund rectangle, or we could allow $K$ to be a general tri-parameter rectangle. Perhaps surprisingly, it is critical
to do the latter. This is somewhat counterintuitive, as e.g. weights involve only Zygmund rectangles.
Eventually, we model Zygmund SIOs through the particular model operators appearing in Definition \ref{E:defn1}.
Proving weighted estimates for them with good enough decay on the complexity is a sophisticated
problem leading to the interplay with the kernel constant $\theta$ associated with the decay in terms
of the Zygmund ratio. Indeed, something rather delicate has to be going on as we know that
weighted estimates do not hold with $\theta < 1$.

The weighted boundedness of the model operators, and thus that of $T$, is eventually proved by establishing completely new maximal and square function
estimates. We need not only estimates for the maximal function $M_{\calD_{ \sub}}$
associated with the sub-Zygmund rectangles $$
\calD_{\sub}:=\{I\in\calD: \ell(I^1)\ell(I^2)\geq\ell(I^3)\}$$
but also more refined estimates, obtained by a suitable use of Stein-Weiss interpolation, for the family of maximal functions $M_{\calD_{\lambda}}$ associated with the super-Zygmund rectangles
$$
\calD_{\lambda} := \{I\in\calD: \lambda \ell(I^1)\ell(I^2) = \ell(I^3)\}, \qquad \lambda = 2^k, \, k \ge 0.
$$
These super-Zygmund rectangles have to be handled via the subcollections $\calD_{\lambda}$, and
the key part is to improve upon the trivial growth estimate
$$
\| M_{\calD_{\lambda}} f \|_{L^p(w)} \lesssim \lambda \| f \|_{L^p(w)},
$$
where $w$ is a Zygmund $A_p$ weight. We have to get $\lambda^{1-\eta}$ with a constant $\eta$
allowed to depend on the weight -- this is just barely enough for the weighted estimates with $\theta = 1$.

Similarly, the most natural Zygmund square function
$$
S_{Z} f := \Big( \sum_{I \in \calD_Z} |\Delta_{I,Z} f|^2 \Big)^{1/2}
$$
does not cover the delicate Zygmund structure present in our model operators. To combat this we introduce
new variants tailored to help us stay in the Zygmund realm and avoid venturing out to the general tri-parameter realm.

\subsection*{Acknowledgements}
T. Hyt\"onen and E. Vuorinen were supported by the Academy of Finland through project numbers 314829 (both) and 346314 (T.H.),
and by the University of Helsinki internal grants for the Centre of Excellence in Analysis and Dynamics Research (E.V.) and the Finnish Centre of Excellence in Randomness and Structures ``FiRST'' (T.H.).
K. Li was supported by the National Natural Science Foundation of China through project number 12001400.

\section{Necessary conditions for weighted estimates}

We begin by constructing some concrete examples of Zygmund weights, something that seems to have been missing from the literature despite a number of papers \cite{DLOPW2,FP,HLLT,HLLTW}, where the theory of these weights has been developed.
A weight $w$---a locally integrable a.e. positive function on $\R^d$---belongs to the weight class $A_p = A_p(\R^d)$, $1 < p < \infty$, if
$$
[w]_{A_p} := \sup_{I}\, \ave{w}_I  \ave{w^{1-p'}}^{p-1}_I
=  \sup_{I}\, \ave{w}_I  \ave{w^{-\frac{1}{p-1}}}^{p-1}_I
 < \infty,
$$
where the supremum is over all cubes $I \subset \R^d$. In the product space variant we replace the supremum over
cubes by a supremum over all rectangles -- in particular, in the tri-parameter case in $\R^3$ the supremum is over
all rectangles $I=I^1 \times I^2 \times I^3 \subset \R^3$. 
In the Zygmund case, the definition, due to \cite[p.\ 339]{FP}, is the following.
A weight $w$ on $\R^3$ belongs to the weight class $A_{p, Z}$, $1 < p < \infty$, if
$$
[w]_{A_{p,Z}} := \sup_{I}\, \ave{w}_I  \ave{w^{1-p'}}^{p-1}_I
 < \infty,
$$
where the supremum is over all rectangles $I\in \mathcal R_Z$, here $\mathcal R_Z:=\{I^1 \times I^2 \times I^3: \ell(I^3) = \ell(I^1)\ell(I^2)\}$. This is a trickier class than the full product
case, since we cannot simply fix some of the variables and always expect the weight to belong uniformly to $A_p$ with respect to the remaining variables. 

A good perspective to the Zygmund case is the following: with $x_2, x_3$ fixed, it is a one-parameter theory on $\R$ and
with $x_1$ fixed, it is an appropriately dilated one-parameter theory on $\R^2$. In particular, if $w \in A_{p,Z}$ we have:
\begin{enumerate}
\item $x_1 \mapsto w(x_1, x_2, x_3) \in A_p(\R)$ uniformly on $x_2, x_3$ and
\item $\langle w \rangle_{I^1} \in A_{p, \ell(I^1)}(\R^2)$ uniformly on $I^1$.
\end{enumerate}
The dilated $A_p$ on $\R^2$ that appears above is defined as follows. Let
\begin{equation}\label{eq:dilatedweights}
  [w]_{A_{p, \lambda}} = \sup_{K}\, \ave{w}_K  \ave{w^{-\frac{1}{p-1}}}^{p-1}_K,
\end{equation}
where $K = K^1 \times K^2$ is a product of intervals with $\ell(K^2) = \lambda \ell(K^1)$.

\begin{rem}
In \cite{DLOPW2,FP,HLLT,HLLTW}, 
 the subscript $\mathfrak z$ is used instead of $Z$ for objects satisfying the Zygmund scaling.
\end{rem}

\subsection{Examples of $A_{p,Z}$ weights}
To facilitate the verification of the $A_{p, Z}$ condition in concrete instances, we observe:

\begin{lem}\label{lem:zrec}
Suppose that $w:\R^3\to\R_+$ is even with respect to each variables separately, i.e., $w(\sigma_1x_1,\sigma_2x_2,\sigma_3x_3)=w(x_1,x_2,x_3)$ whenever $\sigma_1,\sigma_2,\sigma_3\in\{-1,+1\}$. Then the $A_{p,Z}$ condition holds for $w$, if and only if it holds for all rectangles $I=I_1\times I_2\times I_3\in \mathcal R_Z$ of the special type, where  each $I_i$ has either form $[0,t_i]$ or $[c_i,c_i+t_i]$ for some $c_i\geq t_i$.
\end{lem}
\begin{proof}
For each $J\in\mathcal R_{Z}$, we claim that there exists an $I\in\mathcal R_{Z}$ of the special form such that all $w$ as in the assumptions satisfy $\ave{w}_J\lesssim\ave{w}_I$. Since also $w^{-1/{(p-1)}}$ satisfies the same assumptions, this implies that
\begin{equation*}
  \ave{w}_J\ave{w^{-\frac 1{p-1}}}_J^{p-1}\lesssim\ave{w}_I\ave{w^{-\frac 1{p-1}}}_I^{p-1},
\end{equation*}
which proves the lemma.

To verify the claim, first, we can replace each $J_i$ by the longer of the intervals $J_i\cap[-\infty,0]$ and $J_i\cap[0,\infty]$, and if necessary by its reflection onto $[0,\infty]$. Second, if $\dist(J_i,0)<\ell(J_i)$, we replace $J_i$ by $[0,2\ell(J_i)]\supset J_i$ for each $i$. The resulting rectangle $J=J_1\times J_2\times J_3$ has all other properties except that it may fail to be in $\mathcal R_{Z}$. In the first replacement, the length of $J_i$ may be at most divided by two, and in the second step, at most multiplied by two. So in each case $J_i$ deviates from the original by at most a factor of $2$, and hence $\ell(J_1)\ell(J_2)/\ell(J_3)\in[\frac18,8]$. If the ratio is not $1$, let $J_i=[a_i,a_i+t_i]$ be an interval that is ``too short''. If $a_i=0$, or $a_i\geq 8t_i$, then replacing $J_i$ by $[a_i,a_i+t_i']$ for a suitable $t_i'\in[t_i,8t_i]$, we obtain a desired Zygmund rectangle. If $J_i=[a_i,a_i+t_i]$ is ``too short'', but $a_i\in(0,8t_i)$, then we replace this $J_i$ by $[0,9t_i]$, which contains $J_i$ and is most $9$ times longer. This makes some other interval $J_j$ too short, and we can repeat the same consideration. After each iteration, we either achieve a Zygmund rectangle of the desired form, or we replace an interval $J_i=[a_i,a_i+t_i]$ with $a_i>0$ by one with $a_i=0$. This latter can happen at most three times, so after at most a fixed number of iterations, we achieve a Zygmund rectangle of the desired form.
\end{proof}
  
  Now we can provide the proof of Lemma \ref{lem:zygexa} from the Introduction:
  
  \begin{proof}[Proof of Lemma \ref{lem:zygexa}]
   For necessity, we note first that both $w$ and $w^{-1/{(p-1)}}$ need to be locally integrable. In particular, the integrability in the neighborhood of $(1,1,0)$ shows that $\gamma\in(-1,p-1)$, and in the neighborhood of $(0,1,1)$ (or $(1,0,1)$), we get $\alpha-\delta\in(-1,p-1)$.

The final necessary condition connecting $\alpha$ and $\gamma$ is a little trickier. We note that $w_\eps(x_1):=w(x_1,\eps,\eps^2)$ needs to be in $A_p(\R)$, uniformly in $\eps>0$. For $x_1\geq\eps$, we have $w_\eps(x_1)=(x_1\eps)^{\alpha-\gamma}(\eps^2)^\gamma= c(\eps)x_1^{\alpha-\gamma}$, and hence
\begin{equation*}
  \ave{w_\eps}_{(\eps,1)}\ave{w_\eps^{-1/{(p-1)}}}_{(\eps,1)}^{p-1}
  \sim\Big(\int_\eps^1 x_1^{\alpha-\gamma}\ud x_1\Big)\Big(\int_\eps^1 x_1^{-\frac{\alpha-\gamma}{p-1}}\ud x_1\Big)^{p-1}.
\end{equation*}
If $\alpha-\gamma\le -1$, then $\int_\eps^1 x_1^{-\frac{\alpha-\gamma}{p-1}}\ud x_1\sim 1$ and $\int_\eps^1 x_1^{\alpha-\gamma}\ud x_1\to \infty$ as $\eps \to 0$. Hence $\alpha-\gamma>-1$. Similarly, $-\frac{\alpha-\gamma}{p-1}>-1$. This completes the proof of the necessity.

Turning to the sufficiency, Lemma \ref{lem:zrec} allows us to reduce the problem to  verifying 
\[
\sup_I \langle w\rangle_I \langle w^{-\frac 1{p-1}}\rangle_I^{p-1} <\infty,
\]where the supremum is taken over rectangles of the form $I=I_1\times I_2\times I_3$, each $I_i=[0,t_i]$ or $I_i=[c_i,c_i+t_i]$ with $c_i\geq t_i$, and $t_1 t_2=t_3$. First, let us consider the case $I_1=[0,t_1]$. Fix $x_2\in I_2, x_3\in I_3$.  If $t_1\le x_3/{x_2}$, then we have 
\begin{align*}
\langle w\rangle_{I_1}=\frac 1{t_1}\int_0^{t_1}\abs{x_1x_2}^{\alpha-\delta}\abs{x_3}^\delta\ud x_1\sim t_1^{\alpha-\delta} |x_2|^{\alpha-\delta}|x_3|^\delta.
\end{align*}
If  $t_1> x_3/{x_2}$, then 
\begin{align*}
\langle w\rangle_{I_1}&=\frac 1{t_1}\left( \int_0^{x_3/{x_2}}\abs{x_1x_2}^{\alpha-\delta}\abs{x_3}^\delta \ud x_1+ \int_{x_3/{x_2}}^{t_1}\abs{x_1x_2}^{\alpha-\gamma}\abs{x_3}^\gamma \ud x_1\right)\\
&\lesssim t_1^{-1}\big(x_2^{-1}x_3^{\alpha+1}+t_1^{\alpha-\gamma+1}x_2^{\alpha-\gamma}x_3^\gamma\big)\lesssim t_1^{\alpha-\gamma}x_2^{\alpha-\gamma}x_3^\gamma.
\end{align*}
Replacing  $(\gamma, \alpha, \delta)$ by $(-\frac \gamma{p-1}, -\frac\alpha{p-1}, -\frac\delta{p-1})$ gives
\begin{equation*}
\langle w^{-\frac 1{p-1}}\rangle_{I_1}\lesssim \begin{cases}(t_1x_2)^{\frac{\gamma-\alpha}{p-1}}x_3^{-\frac\gamma{p-1}}, & x_3\leq t_1x_2, \\
  (t_1x_2)^{\frac{\delta-\alpha}{p-1}} x_3^{-\frac\delta{p-1}}, &  x_3\geq t_1x_2.\end{cases}
\end{equation*}
Then if $I_2=[0, t_2]$, similar calculus gives us that  
\begin{equation*}
\langle w\rangle_{I_1\times I_2}\lesssim \begin{cases}t_3^{\alpha-\gamma} x_3^{\gamma}, &  x_3\leq t_3, \\
  t_3^{\alpha-\delta} x_3^{\delta}, &  x_3\geq t_3.\end{cases}\quad \langle w^{-\frac1{p-1}}\rangle_{I_1\times I_2}\lesssim \begin{cases}t_3^{\frac{\gamma-\alpha}{p-1}} x_3^{-\frac\gamma{p-1}}, &  x_3\leq t_3, \\
  t_3^{\frac{\delta-\alpha}{p-1}} x_3^{-\frac\delta{p-1}}, &  x_3\geq t_3.\end{cases}
\end{equation*}
Then the case $I_3=[c_3, c_3+t_3]$ is immediate. And for the case $I_3=[0, t_3]$ we just need to use the fact that $|x_3|^\gamma\in A_p$. Now suppose $I_2=[c_2, c_2+t_2]$. Then if $I_3=[c_3, c_3+t_3]$ one can split it to the following cases
\[
c_3+t_3\le t_1 c_2,\quad c_3\ge t_1 (c_2+ t_2),\quad \max\{t_3, t_1c_2-t_3\}\le c_3< t_1 (c_2+ t_2),
\]each of them is immediate. If $I_3=[0, t_3]$ it forces $x_3\le t_1 x_2$ and we still just need to use  that $|x_3|^\gamma\in A_p$.

It remains to consider the case $I_1=[c_1, c_1+t_1]$. In this case, if $I_2=[0, t_2]$ then it is just symmetrical to what we have discussed in the above. Suppose $I_2=[c_2, c_2+t_2]$. Then if $I_3=[0, t_3]$, this forces $x_1x_2\ge x_3$ and again we just need to use that $|x_3|^\gamma\in A_p$. If $I_3=[c_3, c_3+t_3]$ this is just trivial. We are done.

\end{proof}
\subsection{Zygmund SIOs and related results in the literature}
Now let us recall the setting of \cite{DLOPW2,HLLT,HLLTW}.
Let 
\begin{equation*}
  D_\theta(x):=D_\theta(x_1,x_2,x_3):=\Big(\frac{\abs{x_1x_2}}{\abs{x_3}}+\frac{\abs{x_3}}{\abs{x_1x_2}}\Big)^{-\theta},
\end{equation*}
where $\theta \in (0,1]$. Let 
$
K \colon \R^3 \setminus \{ x \colon x_1x_2x_3=0\} \to \C
$
be a function. Define
$$
D^{a_1}_{h_1}K(x)
=a_1K(x_1+h_1,x_2,x_3)-K(x), \quad a_1 \in \{0,1\}.
$$
The fact that $D^{a_1}_{h_1}$ acts on the first variable is understood from the subscript $1$.
Analogously, we define the operators $D^{a_i}_{h_i}$ for $i \in \{2,3\}$, which act on the $i$th parameter. 
Let $x \in \R^3 \setminus \{x \colon x_1x_2x_3=0\}$ and let $h_i \in \R$, $i \in \{1,2,3\}$, 
be such that $|h_i| \le |x_i|/2$.

\begin{defn}\label{def:R}
We say that $K$ satisfies the {\em regularity conditions of Han et al.}\ with parameters $\alpha,\theta\in(0,1]$ if
\begin{equation}\tag{R}\label{E:eq132}
|D^{a_1}_{h_1} D^{a_2}_{h_2} D^{a_3}_{h_3} K(x)|
\lesssim \prod_{i=1}^3 \frac{|h_i|^{a_i\alpha}}{|x_i|^{1+a_i\alpha}} D_\theta(x_1,x_2,x_3)
\end{equation}
for all $a_i \in \{0,1\}$ such that $a_1+a_2+a_3 \le 2$.
\end{defn}

In \cite{DLOPW2,HLLT,HLLTW}, the parameters $\alpha$ and $\theta$
were denoted by $\alpha=\theta_1$ and $\theta=\theta_2$.

In addition to the regularity conditions we formulate the following cancellation conditions:

\begin{defn}\label{def:C}
We say that $K$ satisfies the {\em basic cancellation conditions of Han et al.}\ with parameter $\alpha\in(0,1]$ if
\begin{equation}\tag{C1.a}\label{E:eq133}
   \Babs{\iiint_{\delta_i<\abs{x_i}<r_i:i\in\{1,2,3\}}K(x)\ud x}\lesssim 1,
\end{equation}
and
\begin{equation}\tag{C1.j}\label{E:eq134}
   \Babs{\iint_{\delta_i<\abs{x_i}<r_i:i\in\{1,2,3\}\setminus\{j\}}D^{a_j}_{h_j} K(x)\ud x_{\{1,2,3\}\setminus\{j\}}  }
   \lesssim \frac{|h_j|^{a_j\alpha}}{|x_j|^{1+a_j\alpha}}, \quad |h_j | \le |x_j|/2,
\end{equation}
for each $j=1,2,3$.

We say that $K$ satisfies the {\em Zygmund cancellation condition of Han et al.}\ with parameters $\alpha,\theta\in(0,1]$ if 
\begin{equation}\tag{C2.b}\label{E:eq135}
   \Babs{\int_{\delta_1<\abs{x_1}<r_1}D^{a_2}_{h_2} D^{a_3}_{h_3} 
   K(x)\ud x_1}
   \lesssim \prod_{i=2}^3 \frac{|h_i|^{a_i\alpha}}{|x_i|^{1+a_i\alpha}} 
   (D_\theta(\delta_1,x_2,x_3)+D_\theta(r_1,x_2,x_3)),
\end{equation}
 where $a_2+a_3 \le 1$ and $|h_i| \le |x_i|/2$ for $i \in \{2,3\}$.
 
 We say that $K$ satisfies the {\em cancellation condition of Han et al.} with parameters $\alpha,\theta\in(0,1]$ if it satisfies both basic and Zygmund cancellation conditions of Han et al.\ with these same parameters.
 \end{defn}
 
In \cite{HLLT}, several alternative sets of conditions, with partial mutual overlap, are introduced, and accordingly these same conditions are also denoted by (C1.a) = (C2.a) = (C2'.a), (C1.1) = (C1.c) = (C2.c) and (C1.2) = (C1.d) = (C2'.c) and (C1.3) = (C1.b). 

\begin{defn}\label{def:limits}
We say that $K$ has {\em relevant limits in the sense of Han et al.}, if the integral of
\begin{equation}\label{E:eq136}
  K_{\eps}^N(x):=K(x)\prod_{j=1}^3 1_{(\eps_j,N_j)}(\abs{x_j}).
\end{equation}
converges, as each $\eps_j\to 0$ and $N_j\to\infty$, over each of the following sets:
\begin{enumerate}
  \item over $\{x\in\R^3:\abs{x_i}\leq 1\ \forall i=1,2,3\}$;
  \item over $\{(x_2,x_3)\in\R^2:\abs{x_2},\abs{x_3}\leq 1\}$, for a.e. fixed $x_1\in\R$; and
  \item over $\{(x_1,x_3)\in\R^2:\abs{x_1},\abs{x_3}\leq 1\}$, for a.e. fixed $x_2\in\R$.
\end{enumerate}

We say that $K$ satisfies {\em all assumptions of Han et al.}\ with parameters $\alpha,\theta\in(0,1]$, if it has relevant limits in the sense of Han et al.\ and satisfies both regularity and cancellation conditions of Han et al.\ with these same parameters. We denote by $CZH(\alpha,\theta)$ (for Calder\'on--Zygmund-type conditions of Han et al.) the class of all such kernels, and by $\Norm{K}{CZH(\alpha,\theta)}$ the smallest admissible constant in all of \eqref{E:eq132}, \eqref{E:eq133}, \eqref{E:eq134} and \eqref{E:eq135}.
\end{defn}

In the following statement, we collect the main results of \cite{HLLT} concerning these kernels:

\begin{thm}[\cite{HLLT}, Theorems 1.1--1.3]\label{thm:HLLT}
Suppose that $K$ satisfies the regularity conditions and the basic cancellation conditions of Han et al.\ with parameters $\alpha,\theta\in(0,1]$. Then
$$
  \Norm{K_{\eps}^N*f}{L^2(\R^3)}\lesssim\Norm{f}{L^2(\R^3)},
$$
where the implied constant only depends on those in the assumptions.

If, moreover, $K$ has relevant limits in the sense of Han et al., then the limit $K*f:=\lim K_{\eps}^N*f$ (as each $\eps_j\to 0$ and $N_j\to\infty$) exists in $L^2(\R^3)$, and satisfies
$
  \Norm{K*f}{L^2(\R^3)}\lesssim\Norm{f}{L^2(\R^3)}.
$

If, moreover, $K$ satisfies the Zygmund cancellation condition of Han et al.\ (and hence, in fact, all assumptions of Han et al.) with parameters $\alpha,\theta\in(0,1]$, then $K*f$, initially defined as above on $L^2\cap L^p$, extends to a bounded linear operator on $L^p(\R^3)$, and satisfies
\begin{equation*}
  \Norm{K*f}{L^p(\R^3)}\lesssim\Norm{f}{L^p(\R^3)},\qquad 1<p<\infty.
\end{equation*}
\end{thm}


For later use below, we record the following technical observation:

\begin{lem}\label{lem:CZHBanach}
For $\alpha,\theta\in(0,1]$, the kernel class $CZH(\alpha,\theta)$, equipped with $\Norm{\ }{CZH(\alpha,\theta)}$, is a Banach space.
\end{lem}

\begin{proof}
It is routine to see that $CZH(\alpha,\theta)$ equipped with $\Norm{\ }{CZH(\alpha,\theta)}$ is a normed space, so it remains to check completeness. Let $(K_n)_{n=1}^\infty$ be a Cauchy sequence in $CZH(\alpha,\theta)$. From the regularity condition, we see that it is a Cauchy sequence in the space $C(E)$ of continuous function on $E$, for each compact subset $E\subset\R^3\setminus\{x:x_1x_2x_3=0\}$, and by the well-known completeness of these spaces, $K_n$ converges to some limit $K:\R^3\setminus\{x:x_1x_2x_3=0\}\to\C$, uniformly on compact subsets of $\R^3\setminus\{x:x_1x_2x_3=0\}$. Since the cancellation conditions involve integrals over such sets, it follows readily that $K_n-K$ satisfies both regularity and cancellation conditions of Han et al.\ with constants that become arbitrarily small as $n\to\infty$.

Then it only remains to check that the limit function $K$ also inherits the relevant limits. To see this, we denote by
\begin{equation*}
\begin{split}
  F_k(\eps) &:=\{(x_k,x_3)\in\R^2:\eps_i<\abs{x_i}<1, i=k,3\},\qquad k=1,2, \\
  F_3(\eps) &:=\{x\in\R^3:\eps_i<\abs{x_i}<1, i=1,2,3\}
\end{split}
\end{equation*}
the integration domains in the definition of relevant limits.
Then for $k\in\{1,2,3\}$ and $\eps,\eps'\in(0,1)^3$, we have
\begin{equation*}
\begin{split}
  \Babs{\int_{F_k(\eps)}K(x)-\int_{F_k(\eps')}K(x)}
  &\leq\Babs{\int_{F_k(\eps)}(K(x)-K_n(x))}+\Babs{\int_{F_k(\eps')}(K(x)-K_n(x))} \\
  &\qquad+\Babs{\int_{F_k(\eps)}K_n(x)-\int_{F_k(\eps')}K_n(x)}=:I+II+III,
\end{split}
\end{equation*}
where we omitted the differential $\ud x$ for $k=3$, or $\ud x_k\ud x_3$ for $k=1,2$, for brevity.

For $k\in\{1,2\}$, the set $F_k(\eps)$ is an example of the integration domains appearing in \eqref{E:eq134} (with $j=3-k$), while $F_3(\eps)$ is an example of the integration domain in \eqref{E:eq133}. By what we already observed, $I$ and $II$ above become as small as we like, uniformly in $\eps,\eps'$, as $n\to\infty$. Given $\delta>0$, we can hence fix an $n$ so that $I+II<\delta/2$. For this fixed $n$, we find that $III<\delta/2$, as soon as $\eps,\eps'$ are close enough to zero. Thus $\int_{F_k(\eps)}K(x)$ satisfies the Cauchy criterion, and hence converges as $\eps\to 0$. This shows that $K$ has relevant limits, and completes the proof of completeness of $CZH(\alpha,\theta)$.
\end{proof}

\subsection{Failure of weighted estimates with $w\in A_{p,Z}$ for $\theta<1$}\label{sec:failure}

Below, we will give a concrete example of a Zygmund singular integral operator satisfying the full set of assumptions of Theorem \ref{thm:HLLT}.
We will then derive some necessary conditions for weighted estimates
 \begin{equation}\label{H:conj}
   \Norm{K*f}{L^p(w)}\lesssim\Norm{f}{L^p(w)},\qquad 1<p<\infty,\quad w\in A_{p,Z},
 \end{equation}
and use them to show that weighted bounds cannot, in general, hold
if the kernel exponent $\theta$ satisfies $\theta<1$. This limits what we can do in the future -- we have to put our efforts to proving
the Zygmund weighted estimates with $\theta = 1$. That positive result is the most subtle result of the paper.
Similarly, the counterexamples for $\theta < 1$ also point to some issues with Theorem 1.1 of \cite{DLOPW2}.

\begin{lem}\label{lem:HanOk}
Let $\phi$ be a smooth compactly supported bump function satisfying $\int \phi=0$, and let $t=(t_1,t_2,t_3)\in\R_+^3$. Then the kernel
\begin{equation}\label{eq:exCZZ}
  K(x):=D_\theta(t)\prod_{i=1}^3\frac{1}{t_i}\phi(\frac{x_i}{t_i})
\end{equation}
satisfies all assumptions of Han et al.\ with parameters $\alpha=1$ and the same $\theta\in(0,1]$ as in \eqref{eq:exCZZ}.
\end{lem}

Without the decay factor in front, $K(x)$ would be a simple example of a smooth product Calder\'on--Zygmund kernel in $\R^3$. With the decay factor, we check that it satisfies all assumptions of Han et al.

\begin{proof}
We verify the different conditions one by one.

\subsubsection*{Regularity conditions \eqref{E:eq132}}
We observe that
\begin{equation*}
  \partial^\alpha K(x)=D_\theta(t)\prod_{i=1}^3\frac{1}{t_i^{1+\alpha_i}}\phi^{(\alpha_i)}(\frac{x_i}{t_i});
\end{equation*}
hence,
\begin{equation*}
  \abs{\partial^\alpha K(x)}\prod_{i=1}^3\abs{x_i^{1+\alpha_i}}\lesssim D_\theta(t)\prod_{i=1}^3(\frac{\abs{x_i}}{t_i})^{1+\alpha_i}1_{\{\abs{x_i}\lesssim t_i\}}.
\end{equation*}
Now
\begin{equation*}
\begin{split}
  (\frac{\abs{x_1x_2}}{\abs{x_3}})^\theta &
  \abs{\partial^\alpha K(x)}\prod_{i=1}^3\abs{x_i^{1+\alpha_i}} \\
  &\lesssim (\frac{t_1t_2}{t_3})^\theta D_\theta(t)\times \prod_{i=1}^2(\frac{\abs{x_i}}{t_i})^{1+\theta+\alpha_i}1_{\{\abs{x_i}\lesssim t_i\}}\times(\frac{\abs{x_3}}{t_3})^{1-\theta+\alpha_3}1_{\{\abs{x_3}\lesssim t_3\}},
\end{split}
\end{equation*}
where each factor is bounded by one. Similarly,
\begin{equation*}
\begin{split}
  (\frac{\abs{x_3}}{\abs{x_1x_2}})^\theta &
  \abs{\partial^\alpha K(x)}\prod_{i=1}^3\abs{x_i^{1+\alpha_i}} \\
  &\lesssim (\frac{t_3}{t_1t_2})^\theta D_\theta(t)\times \prod_{i=1}^2(\frac{\abs{x_i}}{t_i})^{1-\theta+\alpha_i}1_{\{\abs{x_i}\lesssim t_i\}}\times(\frac{\abs{x_3}}{t_3})^{1+\theta+\alpha_3}1_{\{\abs{x_3}\lesssim t_3\}},
\end{split}
\end{equation*}
where again each factor is bounded by one.
These derivative estimates imply the condition \eqref{E:eq132}.

\subsubsection*{Basic cancellation conditions  \eqref{E:eq133} and \eqref{E:eq134}}
We simply estimate
\begin{equation*}
\begin{split}
  \Babs{\iiint_{\delta_i<\abs{x_i}<r_i:i\in\{1,2,3\}}K(x)\ud x}
  &\leq\iiint_{\R^3}\abs{K(x)}\ud x 
  \leq \prod_{i=1}^3\int_{\R}\frac{1}{t_i}\abs{\phi(\frac{x_i}{t_i})}\ud x_i=\Norm{\phi}{1}^3\lesssim 1;
\end{split}
\end{equation*}
the decay factor $D_\theta(t)$ was not even needed here.
In a similar way, we have
\begin{equation*}
\begin{split}
   & \Babs{\iint_{\delta_i<\abs{x_i}<r_i:i\in\{1,2,3\}\setminus\{j\}}\partial_j^\gamma K(x)\ud x_{\{1,2,3\}\setminus\{j\}}} \leq\iint_{\R^2}\abs{\partial_j^\gamma K(x)}\ud x_{\{1,2,3\}\setminus\{j\}}\\
  &\hspace{1cm}  \leq \prod_{i\in\{1,2,3\}\setminus\{j\}} \int_{\R}\frac{1}{t_i}\abs{\phi(\frac{x_i}{t_i})}\ud x_i \times\frac{1}{t_j^{1+\gamma}}\Babs{\phi^{(\gamma)}(\frac{x_j}{t_j})} 
  \lesssim 1\times\frac{1}{t_j^{1+\gamma}}1_{\{\abs{x_j}\lesssim t_j\}}\lesssim\frac{1}{\abs{x_j^{1+\gamma}}}.
\end{split}
\end{equation*}

\subsubsection*{Zygmund cancellation condition \eqref{E:eq135}}
Observe that
\begin{equation*}
  \int_{\delta_1<\abs{x_1}<r_1}\partial_2^{\beta_2}\partial_3^{\beta_3} K(x)\ud x_1
  =D_\theta(t)\int_{\delta_1<\abs{x_1}<r_1}\frac{1}{t_1}\phi(\frac{x_1}{t_1})\ud x_1 \prod_{i=2}^3 \frac{1}{t_i^{1+\beta_i}}\phi^{(\beta_i)}(\frac{x_i}{t_i}).
\end{equation*}
The last factors are estimated as before,
\begin{equation*}
    \Babs{\frac{1}{t_i^{1+\beta_i}}\phi^{(\beta_i)}(\frac{x_i}{t_i})}\lesssim \frac{1}{t_i^{1+\beta_i}}1_{\{\abs{x_i}\lesssim t_i\}}.
\end{equation*}
For the integral
\begin{equation*}
  \int_{\delta_1<\abs{x_1}<r_1}\frac{1}{t_1}\phi(\frac{x_1}{t_1})\ud x_1 
  =\int_{\delta_1/t_1<\abs{y}<r_1/t_1}\phi(y)\ud y,
\end{equation*}
we make some alternative estimates. Since $\phi$ is compactly supported, this integral vanishes unless $\delta_1\lesssim t_1$, so without loss of generality we may assume $\delta_1\lesssim t_1$. 
If $r_1\lesssim t_1$, then
\begin{equation*}
  \Babs{\int_{\delta_1/t_1<\abs{y}<r_1/t_1}\phi(y)\ud y}\leq\int_{\abs{y}<r_1/t_1}\abs{\phi(y)}\ud y\lesssim \frac{r_1}{t_1}.
\end{equation*}
Finally, if $r_1\geq Ct_1$, we use $\int\phi=0$ and the fact that $\supp\phi\subset[-C,C]$ to arrive at
\begin{equation*}
\begin{split}
   \Babs{\int_{\delta_1/t_1<\abs{y}<r_1/t_1}\phi(y)\ud y}
   &=\Babs{\int_{\abs{y}\le\delta_1/t_1}\phi(y)\ud y+\int_{\abs{y}\ge r_1/t_1}\phi(y)\ud y} \\
   &=\Babs{\int_{\abs{y}\le \delta_1/t_1}\phi(y)\ud y}\lesssim\frac{\delta_1}{t_1}.
\end{split}
\end{equation*}
Putting these bounds together, we have proved that
\begin{equation*}
   \Babs{ \int_{\delta_1<\abs{x_1}<r_1}\frac{1}{t_1}\phi(\frac{x_1}{t_1})\ud x_1 }
   \lesssim \begin{cases} r_1/t_1, & r_1\lesssim t_1, \\ \delta_1/t_1, & \text{else},\end{cases}
\end{equation*}
and hence
\begin{equation*}
\begin{split}
  A:=\abs{x_2^{1+\beta_2}x_3^{1+\beta_3}}\Babs{\int_{\delta_1<\abs{x_1}<r_1}\partial_2^{\beta_2}\partial_3^{\beta_3}K(x)\ud x}  
  \lesssim D_\theta(t) \prod_{i=2}^3(\frac{\abs{x_i}}{t_i})^{1+\beta_i}1_{\{\abs{x_i}\lesssim t_i\}}
  \times\frac{s_1}{t_1},
\end{split}
\end{equation*}
where $s_1=r_1$ if $r_1\lesssim t_1$, and $s_1=\delta_1$ otherwise. Thus
\begin{equation*}
\begin{split}
  (\frac{s_1\abs{x_2}}{\abs{x_3}})^\theta A 
  \lesssim D_\theta(t)(\frac{t_1t_2}{t_3})^\theta\times(\frac{\abs{x_2}}{t_2})^{1+\theta+\beta_2}1_{\{\abs{x_2}\lesssim t_2\}}\times(\frac{\abs{x_3}}{t_3})^{1-\theta+\beta_3}1_{\{\abs{x_3}\lesssim t_3\}}\times
  (\frac{s_1}{t_1})^{1+\theta}
\end{split}
\end{equation*}
and
\begin{equation*}
\begin{split}
  (\frac{\abs{x_3}}{s_1\abs{x_2}})^\theta A 
  \lesssim D_\theta(t)(\frac{t_3}{t_1t_2})^\theta\times(\frac{\abs{x_2}}{t_2})^{1-\theta+\beta_2}1_{\{\abs{x_2}\lesssim t_2\}}\times(\frac{\abs{x_3}}{t_3})^{1+\theta+\beta_3}1_{\{\abs{x_3}\lesssim t_3\}}\times
  (\frac{s_1}{t_1})^{1-\theta},
\end{split}
\end{equation*}
where in both estimates, each factor on the right is bounded by one. For the last factor in both bounds, we recall that $s_1=r_1$ if and only if $r_1\lesssim t_1$, while if $s_1=\delta_1$, by our convention we also have $\delta_1\lesssim t_1$. Thus it follows that
$
  A\lesssim D_\theta(s_1,x_2,x_3).
$
Recalling the definition of $A$ and the two possible values of $s_1$, this can be written as
\begin{equation*}
    A=\abs{x_2^{1+\beta_2}x_3^{1+\beta_3}}\Babs{\int_{\delta_1<\abs{x_1}<r_1}\partial_2^{\beta_2}\partial_3^{\beta_3}K(x)\ud x_1} 
    \lesssim D_\theta(\delta_1,x_2,x_3)+D_\theta(r_1,x_2,x_3),
\end{equation*}
which is Condition \eqref{E:eq135} that we wanted to verify.

\subsubsection*{Relevant limits}
As a smooth and compactly supported function, $K$ is absolutely integrable. Hence the relevant limits of the integrals of its truncations readily exist by the dominated convergence theorem.
\end{proof}

\begin{lem}\label{lem:convoLower}
Suppose that $\phi\geq 1$ on $[-1,1]$, and that $f\geq 0$ is supported in a rectangle $R=I_1\times I_2\times I_3$. If $t_i=\ell(I_i)$ and
\begin{equation*}
  \Phi(x):=\prod_{i=1}^3\frac{1}{t_i}\phi(\frac{x_i}{t_i}),
\end{equation*}
then $1_R\Phi*f\geq 1_R\ave{f}_R$.
\end{lem}

\begin{proof}
Let $I_i=[c_i-t_i/2,c_i+t_i/2]$. If $x\in R$, then $\abs{x_i-c_i}\leq t_i/2$, and
\begin{equation*}
  \Phi*f(x)=\iiint \prod_{i=1}^3\frac{1}{t_i}\phi(\frac{y_i}{t_i})f(x-y)\ud y.
\end{equation*}
By the support assumption on $f$, we may restrict the integral to be over those $y\in\R^3$ with $x-y\in R$. This means that $\abs{(x_i-y_i)-c_i}\leq t_i/2$. Since also $x\in R$, we have $\abs{x_i-c_i}\leq t_i/2$, and thus we may restrict the integral to be over $\abs{y_i}\leq t_i$. But $\phi(y_i/t_i)\geq 1$ for all these values. Since also $f\geq 0$, we conclude that
\begin{equation*}
\begin{split}
   \Phi*f(x) &=\iiint_{\abs{y_i}\leq t_i:i=1,2,3} \prod_{i=1}^3\frac{1}{t_i}\phi(\frac{y_i}{t_i})f(x-y)\ud y \\
   &\ge \iiint_{\abs{y_i}\leq t_i:i=1,2,3} \prod_{i=1}^3\frac{1}{t_i} f(x-y)\ud y= \frac{1}{\abs{R}}\int_R f(z)\ud z=\ave{f}_R.
\end{split}
\end{equation*}
\end{proof}

For a rectangle $R=I_1\times I_2\times I_3$, let us denote by
\begin{equation*}
  \operatorname{ecc}_{Z}(R):=\max\Big\{\frac{\ell(I_1)\ell(I_2)}{\ell(I_3)},\frac{\ell(I_3)}{\ell(I_1)\ell(I_2)}\Big\}
\end{equation*}
its {\em Zygmund eccentricity}, i.e., the deviation from being a Zygmund rectangle.

\begin{lem}\label{lem:ApZecc}
Let $\alpha,\theta\in(0,1]$ and $p\in(1,\infty)$.
Suppose that $w$ is a weight with the following property: Any family of kernels $K$ that uniformly satisfies all assumptions of Han et al.\ with parameters $\alpha$ and $\theta$, also uniformly satisfies the weighted estimates
\begin{equation}\label{H:conj2}
  \Norm{K*f}{L^p(w)}\lesssim\Norm{f}{L^p(w)}.
\end{equation}
Then this weight satisfies the estimate
\begin{equation}\label{eq:ApZecc}
  \ave{w}_R^{1/p}\ave{w^{-1/(p-1)}}_R^{1/p'}\lesssim \operatorname{ecc}_{Z}(R)^\theta
\end{equation}
uniformly for all axes-parallel rectangles $R\subset\R^3$.
\end{lem}


\begin{proof}
Let $R$ be an arbitrary rectangle, and let $f\geq 0$ be supported on that rectangle. Let $\phi$ be a smooth compactly supported bump function with $\int\phi=0$ and $\phi\geq 1$ on $[-1,1]$. (There is of course no conflict between these conditions; we can let $\phi$ be negative somewhere outside the interval $[-1,1]$.) Let $\Phi$ and $t$ be as in Lemma \ref{lem:convoLower}. By Lemma \ref{lem:HanOk}, $K=D_\theta(t)\Phi$ satisfies  all the assumptions of Han et al., with bounds uniform in $t$. Note that $D_\theta(t)\sim\operatorname{ecc}_{Z}(R)^{-\theta}$. Thus, the assumption of bounded action of such convolution operators on $L^p(w)$, combined with Lemma \ref{lem:convoLower}, shows that
\begin{equation*}
  \Norm{1_R\ave{f}_R}{L^p(w)}\leq\Norm{\Phi*f}{L^p(w)}\sim\operatorname{ecc}_{Z}(R)^\theta\Norm{K*f}{L^p(w)}
  \lesssim\operatorname{ecc}_{Z}(R)^\theta\Norm{f}{L^p(w)}.
\end{equation*}
We make the usual trick by choosing $f=1_R(w+\eps)^{-1/(p-1)}$, so that
\begin{equation*}
  f^p w=1_R(w+\eps)^{-p/(p-1)}w\leq 1_R(w+\eps)^{-p/(p-1)+1}=1_R(w+\eps)^{-1/(p-1)},
\end{equation*}
and we get
\begin{equation*}
  w(R)^{1/p}\ave{(w+\eps)^{-1/(p-1)}}_R\lesssim \operatorname{ecc}_{Z}(R)^\theta [(w+\eps)^{-1/(p-1)}(R)]^{1/p}.
\end{equation*}
Dividing both sides by $\abs{R}^{1/p}$ and then $\ave{(w+\eps)^{-1/(p-1)}}_R^{1/p}$, this gives
\begin{equation*}
  \ave{w}_R^{1/p}\ave{(w+\eps)^{-1/(p-1)}}_R^{1/p'}\lesssim \operatorname{ecc}_{Z}(R)^\theta
\end{equation*}
and, after monotone convergence as $\eps\to 0$, we deduce the claimed \eqref{eq:ApZecc}.
\end{proof}

Towards the failure of the bound \eqref{H:conj2} for all $w\in A_{p,Z}$, we will next check that the bound \eqref{eq:ApZecc} need not hold for all $w\in A_{p,Z}$ unless $\theta=1$.

\begin{prop}\label{prop:ApZeccGrowth}
Suppose that $w\in A_{2,Z}$.
\begin{enumerate}[\rm(a)]
  \item Then $\ave{w}_R\ave{w^{-1}}_R\leq[w]_{A_{2,Z}}\operatorname{ecc}_{Z}(R)^2$ for all axes-parallel rectangles $R\subset\R^3$.
  \item In general, one cannot make any estimate of the form $\ave{w}_R\ave{w^{-1}}_R\leq C\operatorname{ecc}_{Z}(R)^{2-\eta}$, for any $C,\eta>0$ independent of $R$.
\end{enumerate}
\end{prop}

\begin{proof}
If $\ell(R_1)\ell(R_2)<\ell(R_3)$, we let $\tilde R_1\supset R_1$ have $\ell(\tilde R_1)=\ell(R_3)/\ell(R_2)$, so that $\tilde R=\tilde R_1\times R_2\times R_3\in\mathcal R_{Z}$ contains $R$ and has volume
\begin{equation*}
  \abs{\tilde R}=\frac{\ell(R_3)}{\ell(R_2)}\ell(R_2)\ell(R_3)=\frac{\ell(R_3)}{\ell(R_1)\ell(R_2)}\abs{R}=\operatorname{ecc}_{Z}(R)\abs{R},
\end{equation*}
while if $\ell(R_1)\ell(R_2)>\ell(R_3)$, we let $\tilde R_3\supset R_3$ have $\ell(\tilde R_3)=\ell(R_1)\ell(R_2)$, so that $\tilde R=R_1\times R_2\times \tilde R_3\in\mathcal R_{Z}$ contains $R$ and has volume
\begin{equation*}
  \abs{\tilde R}=(\ell(R_1)\ell(R_2))^2=\frac{\ell(R_1)\ell(R_2)}{\ell(R_3)}\abs{R}=\operatorname{ecc}_{Z}(R)\abs{R}.
\end{equation*}
So in either case we find that $R\subset\tilde R\in\mathcal R_{Z}$ with $\abs{\tilde R}=\operatorname{ecc}_{Z}(R)\abs{R}$, and hence
\begin{equation*}
\begin{split}
  \ave{w}_R\ave{w^{-1}}_R
  &=\frac{1}{\abs{R}^2}\int_R w\times\int_R w^{-1} 
  \leq\frac{\operatorname{ecc}_{Z}(R)^2}{\abs{\tilde R}^2}\int_{\tilde R} w\times\int_{\tilde R} w^{-1} \\
  &=\operatorname{ecc}_{Z}(R)^2\ave{w}_{\tilde R}\ave{w^{-1}}_{\tilde R}
  \leq \operatorname{ecc}_{Z}(R)^2[w]_{A_{2,Z}}.
\end{split}
\end{equation*}

To see that this estimate is optimal in terms of the power of $\operatorname{ecc}_{Z}(R)$, let us consider a weight $w$ as in Lemma \ref{lem:zygexa} on the rectangle $R=(0,\sqrt\eps)\times(0,\sqrt\eps)\times(\eps,1)$, which has $\operatorname{ecc}_{Z}(R)\sim \eps^{-1}$. Since $\abs{x_1x_2}<\eps<\abs{x_3}$ for $x\in R$, we have
$w(x)=\abs{x_1x_2}^{\alpha-\delta}\abs{x_3}^\delta$, $x\in R$.
Hence
\begin{equation*}
  \ave{w}_R \sim\Big(\prod_{i=1}^2\frac{1}{\sqrt\eps}\int_0^{\sqrt\eps} x_i^{\alpha-\delta}\ud x_i\Big)\times\int_{\eps}^1 x_3^\delta\ud x_3,
\end{equation*}
where
\begin{equation*}
  \prod_{i=1}^2\frac{1}{\sqrt\eps}\int_0^{\sqrt\eps} x_i^{\alpha-\delta}\ud x_i
  \sim\prod_{i=1}^2\frac{1}{\sqrt\eps}(\sqrt\eps)^{\alpha-\delta+1}
  =\eps^{\alpha-\delta}.
\end{equation*}
Since $w^{-1}$ has the same form with $(\alpha,\delta)$ replaced by $(-\alpha,-\delta)$, we find that
\begin{equation*}
  \ave{w}_R\ave{w^{-1}}_R\sim \int_{\eps}^1 x_3^\delta\ud x_3\times \int_{\eps}^1 x_3^{-\delta}\ud x_3
  \sim \int_{\eps}^1 x_3^{-\abs{\delta}}\ud x_3,
\end{equation*}
since $\int_\eps^1 x_3^{\abs{\delta}}\ud x_3\sim 1$, and in particular
\begin{equation*}
  \ave{w}_R\ave{w^{-1}}_R\sim\eps^{1-\delta}\sim\operatorname{ecc}_{Z}(R)^{\delta-1},\qquad\text{if}\quad\delta>1.
\end{equation*}
It remains to observe that, for any $\eta\in(0,1)$, the lemma allows us to choose
\begin{equation*}
  \gamma=1-\eta,\quad\alpha=\gamma+1-\eta=2-2\eta,\quad\delta=\alpha+1-\eta=3-3\eta,
\end{equation*}
so that $\delta-1=2-3\eta$ can be arbitrarily close to $2$.
\end{proof}

We are now ready to state our main counterexamples:

\begin{prop}\label{prop:mainCounterex}
Let $\alpha\in(0,1]$ and $\theta\in(0,1)$. Then we can find a weight $w\in A_{2,Z}$ with the following properties:
\begin{enumerate}[\rm(a)]
  \item\label{it:unifFails} For some family of kernels $K$ that uniformly satisfy all assumptions of Han et al.\ with parameters $\alpha$ and $\theta$, there is no uniform bound
\begin{equation}\label{H:conj3}
  \Norm{K*f}{L^2(w)}\lesssim\Norm{f}{L^2(w)}.
\end{equation}
  \item\label{it:singleUnbounded} For some single kernel $K$ that satisfies all assumptions of Han et al.\ with parameters $\alpha$ and $\theta$, the bound \eqref{H:conj3} fails.
\end{enumerate}
\end{prop}

\begin{proof}
\eqref{it:unifFails}: Since $\theta\in(0,1)$, Proposition \ref{prop:ApZeccGrowth} shows that we can find some $w\in A_{2,Z}$ that fails to satisfy
\begin{equation}\label{eq:ApZecc2}
  \ave{w}_R^{1/2}\ave{w^{-1}}_R^{1/2}\lesssim\operatorname{ecc}_Z(R)^\theta.
\end{equation}
If, for this weight $w$, we would in fact have the uniform bound \eqref{H:conj3}, then by Lemma \ref{lem:ApZecc}, we would also have \eqref{eq:ApZecc2}. This obvious contradiction proves \eqref{it:unifFails}.

\eqref{it:singleUnbounded}: This formal improvement of \eqref{it:unifFails} will in fact be obtained as a consequence of \eqref{it:unifFails} and a standard application of the closed graph theorem. Suppose that \eqref{it:singleUnbounded} is not true, or in other words that bound \eqref{H:conj3} holds for each individual $K\in CZH(\alpha,\theta)$, but with the implied constant necessarily (due to \eqref{it:unifFails} that we already proved) depending on $K$. This allows us to define a linear operator $\Lambda$ from $CZH(\alpha,\theta)$ to the space $\mathcal L(L^2(w))$ of bounded linear operators on $L^2(w)$ by setting $\Lambda(K)(f):=K*f$.

We claim that $\Lambda$ is closed. To see this, let $K_i\to K$ in $CZH(\alpha,\theta)$ and $\Lambda(K_i)\to T$ in $\mathcal L(L^2(w))$, where the latter implies that $K_i*f=\Lambda(K_i)(f)\to Tf$ in $L^2(w)$ for each $f\in L^2(w)$. On the other hand, by Theorem \ref{thm:HLLT}, the convergence $K_i\to K$ in $CZH(\alpha,\theta)$ implies that $K_i*f\to K*f$ in $L^2$ for each $f\in L^2$. Passing to subsequences that converge almost everywhere, it follows that $Tf=\lim K_i*f=K*f=\Lambda(K)(f)$ for all $f\in L^2(w)\cap L^2$. Since this space is dense in $L^2(w)$, we conclude that $T=\Lambda(K)$, and hence $\Lambda$ is a closed operator.

By the closed graph theorem, the closed operator $\Lambda$ from the Banach space $CZH(\alpha,\theta)$ (according to Lemma \ref{lem:CZHBanach}) to the Banach space $\mathcal L(L^2(w))$ is bounded. But this implies the uniform estimate \eqref{H:conj3}, which we already know to be false. The proof of \eqref{it:singleUnbounded} is hence finished by contradiction.
\end{proof}

\section{Ricci-Stein operators and Fefferman-Pipher multipliers}\label{E:sec3}

Having seen that more stringent abstract kernel estimates are required for the weighted boundedness, we aim at 
identifying them by looking at some known singular integrals of Zygmund type. One of the main
conclusions of this section will be that the so-called truncated Fefferman-Pipher multipliers \cite{FP} belong
to the class we will introduce in Section \ref{E:sec1}. This will be stated and proved in Proposition \ref{E:prop5}.

We begin by recalling a class of multi-parameter singular integral operators in the Zygmund setting studied by Fefferman-Pipher \cite{FP}.
This class is a special case of a more general class considered in Ricci-Stein \cite{RS}.

Fix some large number $N \in \N$ and define
$$
\calS_N=
\Big\{\psi \in C^N(\R^3) 
\colon \| \psi \|_{\calS_N}:= \sup_{x \in \R^3} (1+|x|^N) \sum_{\substack{\alpha \in \N^3 \\ |\alpha|_\infty \le N}} 
|\partial^\alpha \psi(x)| < \infty\Big\}.
$$
Suppose that for $k,j \in \Z$ we have a function $\psi_{k,j} \in \calS_N$. We assume the
size and smoothness condition
$
\sup_{k,j} \| \psi_{k,j}\|_{\calS_N} < \infty.
$
We also assume the following cancellation conditions.
For every $y \in \R^3$ and $\alpha \in \N^3$, $|\alpha|_\infty \le N$, there holds that
\begin{equation}\label{E:eq88}
\begin{split}
&\int_{\R^2} x_1^{\alpha_1} x_2^{\alpha_2} \psi_{k,j}(x_1,x_2,y_3) \ud x_1 \ud x_2 =0, \\
&\int_{\R^2} x_1^{\alpha_1} x_3^{\alpha_3} \psi_{k,j}(x_1,y_2,x_3) \ud x_1 \ud x_3 =0 \quad \text{and} \\
&\int_{\R^2} x_2^{\alpha_2} x_3^{\alpha_3} \psi_{k,j}(y_1,x_2,x_3) \ud x_2 \ud x_3 =0
\end{split}
\end{equation}
for all $k,j$.
The operators studied in Fefferman-Pipher \cite{FP} are of the form
\begin{equation}\label{E:eq87}
Tf=f*K,
\end{equation}
where
$$
K(x):=\sum_{k,j \in \Z} 2^{-2(k+j)} \psi_{k,j}\Big( \frac{x_1}{2^k},\frac{x_2}{2^j},\frac{x_3}{2^{j+k}}\Big).
$$
We will call this kind of operators Ricci-Stein operators.

It was shown in \cite{FP} that a Ricci-Stein operator can be written as $f*K=T_{m_1}f+T_{m_2}f$, 
where $\widehat{T_{m_i}f}=m_i \widehat f$ is a Fourier multiplier operator
and $m_i \in \calM_Z^i$ is a multiplier that we will call a Fefferman-Pipher multiplier. 
The multiplier collections $\calM_Z^i$ are defined in Section \ref{E:subsec1}. These collections are symmetrical;
$\calM^1_Z$ is related to the partitioning $\{1\}$, $\{2,3\}$ of the parameters and $\calM^2_Z$ is related to $\{2\}$, $\{1,3\}$.
According to \cite[p.\ 356]{FP},
an arbitrary Fefferman-Pipher multiplier operator $T_m$, where $m \in \calM^1_Z \cup \calM^2_Z$, satisfies the weighted estimate
$
\|T_{m}f\|_{L^p(w)} \lesssim \| f \|_{L^p(w)},
$
where $1<p<\infty$ and $w \in A_{p,Z}$. 
In particular by \cite[Theorem 2.4]{FP}, it follows that also the Ricci-Stein operators \eqref{E:eq87} satisfy the weighted estimate.
(According to \cite{DLOPW2}, ``the proof in this paper'' (i.e., in \cite{DLOPW2}) ``fills a gap in the argument of \cite{FP}.'')

In the next sections we will study some properties of the Fefferman-Pipher multipliers. 
In this paper we work with the parameter grouping $\{1\}$, $\{2,3\}$; symmetrically we could 
of course choose the other grouping. 
Thus, in the next sections we consider multipliers in $\calM^1_Z$.
We prove some properties for these that will form the basis of our definition of singular integrals in 
Section \ref{E:sec1}. In particular, (truncated) Fefferman-Pipher multipliers will belong to our class.
Since our aim is to develop dyadic-probabilistic techniques, we look for a different kind of
set of assumptions than that used in the papers \cite{{DLOPW2}} and \cite{HLLT}. The structure of our
assumptions is analogous to the one used in bi-parameter analysis \cite{Ma1}.

\subsection{Fefferman-Pipher multipliers}\label{E:subsec1}
Let
$$
\rho_{s,t}(x)=(sx_1,tx_2,stx_3), \quad x \in \R^3, s,t>0.
$$
Define
$$
A^1 :=\{\xi\in\R^3:\tfrac12<\abs{\xi_1}\leq 1,\tfrac12<\abs{(\xi_2,\xi_3)}\leq 1\}
$$
and
\begin{equation}\label{E:eq64}
\begin{split}
\mathcal M^1_{Z} &:=\Big\{  m\in C^N\big((\R\setminus\{0\})\times(\R^2\setminus\{(0,0)\})\big): \\
&\qquad\Norm{m}{\mathcal M^1_{Z}}:=
\sup_{\norm{\alpha}{\infty}\leq N}\sup_{s,t>0} \sup_{\xi\in A^1}\abs{\partial^\alpha(m\circ\rho_{s,t})(\xi)}<\infty\Big\}.
\end{split}
\end{equation}
The sets $A^2$ and $\mathcal M^2_{Z}$ are defined similarly by swapping the roles of coordinates $1$ and $2$ throughout.

Let $m \in \calM_Z^1$. Another way to see the differential estimate in \eqref{E:eq64} is
the following. If $\xi\in A^1$ one has by definition that
\[
\abs{ (\partial^\alpha m)(s\xi_1, t\xi_2, st\xi_3)}\le  \Norm{m}{\mathcal M^1_{Z}} s^{-\alpha_1} t^{-\alpha_2} (st)^{-\alpha_3}=\Norm{m}{\mathcal M^1_{Z}} s^{-\alpha_1+\alpha_2} (st)^{-\alpha_2-\alpha_3}.
\]
For arbitrary $\eta_1\neq 0$ and $(\eta_2, \eta_3)\neq 0$ denote 
\[
s=|\eta_1|,\qquad st=|(s\eta_2, \eta_3)|,\qquad (\xi_1, \xi_2, \xi_3)= (\frac{\eta_1}s, \frac{\eta_2}t, \frac{\eta_3}{st}).
\]
Then $(\xi_1, \xi_2, \xi_3)\in A^1$ and so 
\begin{equation}\label{E:eq66}
\abs{ (\partial^\alpha m)(\eta_1, \eta_2, \eta_3)}\le  \Norm{m}{\mathcal M^1_{Z}} |\eta_1|^{-\alpha_1+\alpha_2} |(|\eta_1|\eta_2, \eta_3)|^{-\alpha_2-\alpha_3}.
\end{equation}

As we already mentioned above, Fefferman--Pipher \cite{FP} shows that any Ricci--Stein operator $Tf=K*f$ as in  \eqref{E:eq87} can be written as $\widehat{Tf}=m_1 \hat{f}+m_2\hat{f}$, where $m_i\in\mathcal M^i_{Z}$. 
Let us conversely see what can be said about $K_i=\check m_i$ for $m_i\in\mathcal M^i_{Z}$. 
Without loss of generality, we consider $i=1$.

\begin{lem}\label{E:lem6}
Let $m_1 \in \mathcal M^1_{Z}$. Then $K_1=\check m_1$ satisfies
\begin{equation*}
  \abs{\partial^\beta K(x)}\lesssim 
  \frac{1}{(|x_1x_2|+|x_3|)^{2+\beta_3}}\frac{1}{|x_1|^{\beta_1}|x_2|^{\beta_2}} 
  \min\Big\{1,\Big(\frac{|x_1x_2|}{|x_3|}\Big)^{\min(\beta_1,\beta_2)}\Big\}L_\beta(x),
\end{equation*}
where
\begin{equation*}
  L_\beta(x):=L(x):=1+\log_+\frac{|x_3|}{|x_1x_2|}
\end{equation*}
if $\beta_1=\beta_2$ and $L_\beta(x):=1$ otherwise.
\end{lem}

\begin{proof}
With standard resolutions of unity $\phi^1$ on $\R\setminus\{0\}$ and $\phi^{23}$ on $\R^2\setminus\{(0,0)\}$, we have
\begin{equation*}
  1=\sum_{j\in\Z}\phi^1(2^{-j}x_1)=\sum_{k\in\Z}\phi^{23}(2^{-k}(x_2,x_3))=\sum_{j,k}\phi^1(2^{-j}x_1)\phi^{23}(2^{-k}x_2,2^{-j-k}x_3),
\end{equation*}
and thus
\begin{equation}\label{E:eq126}
\begin{split}
  m&=\sum_{j,k}((\phi^1\otimes\phi^{23})\circ\rho_{2^{-j},2^{-k}})\cdot m \\
 & =\sum_{j,k}((\phi^1\otimes\phi^{23})\cdot (m\circ\rho_{2^j,2^k}))\circ\rho_{2^{-j},2^{-k}}=:\sum_{j,k}m_{j,k}.
\end{split}
\end{equation}
Since $\phi^1$ is supported in $[-2,-1/2] \cup [1/2,2]$ and $\phi^{23}$ in $\bar B(0,2) \setminus B(0,1/2)$, there holds that
\begin{equation*}
\supp m_{j,k}\subseteq \supp ((\phi^1\otimes\phi^{23})\circ\rho_{2^{-j},2^{-k}})\subseteq \rho_{2^j,2^k}\wt A^1,
\end{equation*}
where
$$
\wt A^1:= ([-2,- 1/2] \cup [ 1/2,2 ]) \times (\bar B(0,2) \setminus B(0,1/2)).
$$
If $\xi \in \rho_{2^j,2^k}\wt A^1$, then one can find $s \sim 2^j$, $t \sim 2^k$ and $\eta \in A^1$ so that $\xi = \rho_{s,t}(\eta)$.
Using this one gets that
\begin{equation*}
\Norm{\partial^\alpha m_{j,k}}{\infty}\lesssim 2^{-(j,k,j+k)\cdot\alpha},\qquad
\Norm{\partial^\alpha m_{j,k}}{1}\lesssim 2^{(j,k,j+k)\cdot(\mathbf 1-\alpha)},\qquad\mathbf 1:=(1,1,1).
\end{equation*}

If $K_{j,k}:=\check m_{j,k}$, then
\begin{equation*}
\begin{split}
  \Norm{x^\alpha\partial^\beta K_{j,k}}{\infty}
  &\lesssim\Norm{\partial^\alpha(\xi^\beta m_{j,k})}{1}
  \leq\sum_{\gamma\leq\alpha}\binom{\alpha}{\gamma}\Norm{\partial^{\gamma}\xi^\beta\cdot\partial^{\alpha-\gamma}m_{j,k}}{1}  \\
  &\lesssim\sum_{\gamma\leq\alpha\wedge\beta}\Norm{\xi^{\beta-\gamma}\cdot\partial^{\alpha-\gamma}m_{j,k}}{1}
  \leq\sum_{\gamma\leq\alpha\wedge\beta}\Norm{\xi^{\beta-\gamma}1_{\rho_{2^j,2^k} \wt A^1}}{\infty}\Norm{\partial^{\alpha-\gamma}m_{j,k}}{1} \\
  & \lesssim\sum_{\gamma\leq\alpha\wedge\beta} 2^{(j,k,j+k)\cdot(\beta-\gamma)} 2^{(j,k,j+k)\cdot(\mathbf 1-(\alpha-\gamma))}
    \lesssim 2^{(j,k,j+k)\cdot(\beta+\mathbf 1-\alpha)},
\end{split}  
\end{equation*}
and thus
\begin{equation*}
  \abs{x^{\beta+\mathbf 1}\partial^\beta K_{j,k}(x)}\lesssim 2^{(j,k,j+k)\cdot(\beta+\mathbf 1-\alpha)}\abs{x^{\beta+\mathbf 1-\alpha}}.
\end{equation*}
Here the left-hand side is independent of $\alpha$, so we are free to choose one which is convenient for us. 
Taking either $\alpha_i\in\{0,N\}$, we find that
$
   K(x):=\sum_{j,k}K_{j,k}(x)
$
satisfies
\begin{equation}\label{eq:FPmk}
\begin{split}
  \abs{&x^{\beta+\mathbf 1}\partial^\beta K(x)}\lesssim
  \sum_{j} \min\{(2^j \abs{x_1})^{\beta_1+1},(2^j \abs{x_1})^{\beta_1+1-N}\} \\
  &\times\sum_k\min\{(2^k \abs{x_2})^{\beta_2+1},(2^k \abs{x_2})^{\beta_2+1-N}\} 
  \min\{(2^{j+k} \abs{x_3})^{\beta_3+1},(2^{j+k} \abs{x_3})^{\beta_3+1-N}\}.
\end{split}
\end{equation}
For convenience, we will assume without loss of generality that $x_i>0$ for each $i=1,2,3$. The general case can be handled by writing $\abs{x_i}$ in place of $x_i$.

The inner sum in \eqref{eq:FPmk} can be estimated either by
\begin{equation*}
    \sum_{2^k\leq 1/x_2} (2^k x_2)^{\beta_2+1}(2^{k+j}x_3)^{\beta_3+1}
    +\sum_{2^k\geq 1/x_2}(2^k x_2)^{\beta_2+1-N}(2^{k+j}x_3)^{\beta_3+1} 
    \lesssim \Big(\frac{2^j x_3}{x_2}\Big)^{\beta_3+1},
\end{equation*}
or by 
\begin{equation*}
    \sum_{2^k\leq 2^{-j}/x_3} (2^k x_2)^{\beta_2+1}(2^{k+j}x_3)^{\beta_3+1}
    +\sum_{2^k\geq 2^{-j}/x_3}(2^k x_2)^{\beta_2+1}(2^{k+j}x_3)^{\beta_3+1-N} 
    \lesssim \Big(\frac{x_2}{2^{j} x_3}\Big)^{\beta_2+1},
\end{equation*}
in both cases provided that $\beta_2+\beta_3<N-2$.


The outer sum in \eqref{eq:FPmk} can then be estimated either by
\begin{equation*}
  \sum_{2^j\leq 1/x_1} (2^j x_1)^{\beta_1+1}(2^j x_3/x_2)^{\beta_3+1}
  +\sum_{2^j\geq 1/x_1}(2^j x_1)^{\beta_1+1-N}(2^j x_3/x_2)^{\beta_3+1}
  \lesssim\Big(\frac{x_3}{x_1x_2}\Big)^{\beta_3+1}
\end{equation*}
or, if $x_1x_2\leq x_3$, by
\begin{equation*}
\begin{split}
  \sum_{2^j\leq x_2/x_3} & (2^j x_1)^{\beta_1+1}(2^j x_3/x_2)^{\beta_3+1}
  +\sum_{x_2/x_3\leq 2^j\leq 1/x_1} (2^j x_1)^{\beta_1+1}\Big(\frac{x_2}{2^j x_3}\Big)^{\beta_2+1} \\
  &+\sum_{2^j\geq 1/x_1}(2^j x_1)^{\beta_1+1-N}\Big(\frac{x_2}{2^j x_3}\Big)^{\beta_2+1}=:I+II+III.
\end{split}
\end{equation*}
It is straightforward that
\begin{equation*}
  I\sim \Big(\frac{x_1x_2}{x_3}\Big)^{\beta_1+1},\qquad
  III\sim \Big(\frac{x_1x_2}{x_3}\Big)^{\beta_2+1},
\end{equation*}
provided that $\beta_1-\beta_2<N$. For $II$, we have an increasing, decreasing, or constant geometric series, depending on the relative size of $\beta_1$ and $\beta_2$, and we find that
\begin{equation*}
  II \sim \Big(\frac{x_1x_2}{x_3}\Big)^{\min(\beta_1,\beta_2)+1}L_\beta(x),
\end{equation*}
where
\begin{equation*}
  L_\beta(x):=L(x):=1+\log_+\frac{x_3}{x_1x_2}
\end{equation*}
if $\beta_1=\beta_2$ and $L_\beta(x):=1$ otherwise.

In conclusion, we have two qualitatively different estimates depending on the relative size of $x_1x_2$ and $x_3$:
\begin{equation*}
\begin{split}
  \abs{x^{\beta+\mathbf 1}\partial^\beta K(x)} &\lesssim \Big(\frac{x_3}{x_1x_2}\Big)^{\beta_3+1},\qquad \frac{x_3}{x_1x_2}\leq 1, \\
   \abs{x^{\beta+\mathbf 1}\partial^\beta K(x)} &\lesssim \Big(\frac{x_1x_2}{x_3}\Big)^{\min(\beta_1,\beta_2)+1}L_\beta(x),
   \qquad \frac{x_3}{x_1x_2}\geq 1,
\end{split}
\end{equation*}
or
\begin{equation*}
  \abs{\partial^\beta K(x)} \lesssim \frac{1}{x_1^{1+\beta_1}x_2^{1+\beta_2}x_3^{1+\beta_3}} \Big(\frac{x_3}{x_1x_2}\Big)^{\beta_3+1}
  =\frac{1}{(x_1x_2)^{2+\beta_3}x_1^{\beta_1} x_2^{\beta_2}},\qquad \frac{x_3}{x_1x_2}\leq 1, \\
\end{equation*}
and, 
\begin{equation*}
\begin{split}
  \abs{\partial^\beta K(x)} &\lesssim \frac{1}{x_1^{1+\beta_1}x_2^{1+\beta_2}x_3^{1+\beta_3}} \Big(\frac{x_1x_2}{x_3}\Big)^{\min(\beta_1,\beta_2)+1}L_\beta(x) \\
  &=\frac{1}{x_3^{2+\beta_3}x_1^{\beta_1} x_2^{\beta_2}}\Big(\frac{x_1x_2}{x_3}\Big)^{\min(\beta_1,\beta_2)}L_\beta(x) ,\qquad \frac{x_3}{x_1x_2}\geq 1.
\end{split}
\end{equation*}
Noting that $x_1x_2+x_3\sim\max(x_1x_2,x_3)$, the two bounds may be finally combined into
\begin{equation*}
  \abs{\partial^\beta K(x)}\lesssim 
  \frac{1}{(x_1x_2+x_3)^{2+\beta_3}}\frac{1}{x_1^{\beta_1}x_2^{\beta_2}} \min\Big\{1,\Big(\frac{x_1x_2}{x_3}\Big)^{\min(\beta_1,\beta_2)}\Big\}L_\beta(x).
\end{equation*}
 Recall that we assumed that $x_i>0$; the general case $x_i \not=0$ can be obtained by replacing $x_i$ by $|x_i|$ in the right
hand side.
\end{proof}

Let us explicitly write out some of the derivatives. Let $x \in \R^3$ be such that $x_i \not=0$ for every $i$.  
First, taking $\beta=0$ we have the size estimate
\begin{equation}\label{E:eq51}
|K(x)|
\lesssim \frac{L(x)}{(|x_1x_2|+|x_3|)^{2}}.
\end{equation}
The first order derivatives are
\begin{equation}\label{E:eq52}
|\partial_i K(x)|
\lesssim \frac{1}{(|x_1x_2|+|x_3|)^{2}}\frac{1}{|x_i|}, \quad i=1,2, \quad
|\partial_3 K(x)| \lesssim \frac{L(x)}{(|x_1x_2|+|x_3|)^{3}}.
\end{equation}
Finally, we have the second order derivatives
$$
|\partial_1 \partial_2 K(x)|
\lesssim  \frac{1}{(|x_1x_2|+|x_3|)^{2}}\frac{1}{|x_1||x_2|} \min\Big\{1,\frac{|x_1x_2|}{|x_3|}\Big\}L(x)
\lesssim \frac{1}{(|x_1x_2|+|x_3|)^{2}}\frac{1}{|x_1||x_2|}
$$
and
$$
|\partial_1 \partial_3 K(x)|
\lesssim  \frac{1}{(|x_1x_2|+|x_3|)^{3}}\frac{1}{|x_1|}.
$$

One point to notice here is that differentiation with respect to the first two variables gives the corresponding extra decay $1/|x_i|$,
but differentiation with respect to the third variable produces an extra decay with respect to the full factor 
$1/(|x_1x_2|+|x_3|)$. However, we point out that
not all of the decay available here is needed for the weighted boundedness with abstract kernels.

\subsection{Partial kernel properties}
Our definition of Zygmund singular integrals will involve certain partial kernel representations 
(see Section \ref{E:subsec5}).
In this section we consider the corresponding properties of the Fefferman-Pipher multipliers $m \in \calM^1_Z$. 

Let $m \in \calM^1_Z$. Recall the decomposition \eqref{E:eq126} of $m$.
Define truncations of $m$ by
\begin{equation}\label{E:eq127}
m_{J}:= \sum_{\substack{|j| \le J^1 \ |k| \le J^2}}m_{j,k}, \quad J=(J^1,J^2) \in \N^2. 
\end{equation}

\begin{lem}\label{E:lem3}
Suppose that $m \in \calM^1_Z$. Let  $m_J$ be its truncation as defined in \eqref{E:eq127}
and let $K_J := \check m_J$.
Then  for $x_2 \not=0 \not= x_3$ we have the estimate
$$
\Big| \iint \displaylimits_{I^1\times I^1}  \partial_2^{\beta_2}\partial_3^{\beta_3}K_J(x_1-y_1,x_2,x_3) \ud y_1 \ud x_1 \Big|
\lesssim \frac{|I^1|}{|x_2|^{1+\beta_2}|x_3|^{1+\beta_3}}
\frac{\log\Big(\frac{|I^1||x_2|}{|x_3|}+\frac{|x_3|}{|I^1||x_2|}\Big)}{\frac{|I^1||x_2|}{|x_3|}+\frac{|x_3|}{|I^1||x_2|}},
$$
where $I^1 \subset \R$ is an interval and $\beta_2+\beta_3 \le 1$.
\end{lem}


\begin{proof}[Proof of Lemma \ref{E:lem3}]
Since $m_L$ is supported outside the origin one sees that
\begin{equation}\label{E:eq60}
\int_{\R} \partial_2^{\beta_2} \partial_3^{\beta_3}K_J(x_1,x_2,x_3) \ud x_1=0.
\end{equation}
Indeed, this can be seen by taking the Fourier transform with respect to $x_2$ and $x_3$
of the left hand side of \eqref{E:eq60}.

Suppose first that $|I^1||x_2|/|x_3| \ge 1$. Due to the zero integral \eqref{E:eq60} we may equivalently estimate the integral
over $ I^1 \times (\R \setminus I^1)$ instead of  $I^1 \times I^1$. By the estimates \eqref{E:eq51} and \eqref{E:eq52} there holds that
\begin{equation*}
\begin{split}
\Big| \iint \displaylimits_{I^1\times (\R \setminus I^1)} & \partial_2^{\beta_2}\partial_3^{\beta_3}
K_J(x_1-y_1,x_2,x_3) \ud y_1 \ud x_1 \Big| \\
&\lesssim \iint \displaylimits_{I^1\times (\R \setminus I^1)} \frac{1}{|x_2|^{\beta_2} }
\frac{1+ \log_+ \Big(\frac{|x_3|}{|x_1-y_1||x_2|}\Big)}{(|x_1-y_1||x_2|+|x_3|)^{2+\beta_3}} \ud y_1 \ud x_1.
\end{split}
\end{equation*}
We introduced the logarithm regardless of the value of $\beta_2$ because it does not affect the estimate.
Let $t:= |x_3|/|x_2|$. By a change of variables we reduce to
\begin{equation}\label{E:eq61}
 \iint \displaylimits_{t^{-1}I^1\times (\R \setminus t^{-1}I^1)} 
\frac{1+ \log_+ \Big(\frac{1}{|x_1-y_1|}\Big)}{(|x_1-y_1|+1)^{2+\beta_3}} \ud y_1 \ud x_1 \frac{1}{|x_2|^{2+\beta_2}|x_3|^{\beta_3}},
\end{equation}
where the integral can be estimated by
\begin{equation*}
\int_{t^{-1}I^1} \frac{1}{(d(x_1, \partial(t^{-1}I^1)) +1)^{1+\beta_3}} \ud x_1
\lesssim \begin{cases}
1+ \log_+\Big( \frac{|I^1||x_2|}{|x_3|} \Big), \quad &\beta_3=0, \\
1, \quad & \beta_3=1.
\end{cases}
\end{equation*}
Substituting this into \eqref{E:eq61} gives 
$$
\begin{cases}
\frac{|I^1|}{|x_2|^{1+\beta_2}|x_3|} \frac{|x_3|}{|I^1||x_2|} \Big(1+ \log_+\Big( \frac{|I^1||x_2|}{|x_3|} \Big)\Big)
\sim \frac{|I^1|}{|x_2|^{1+\beta_2}|x_3|}  \frac{1+ \log_+\Big( \frac{|I^1||x_2|}{|x_3|} \Big)}{\frac{|I^1||x_2|}{|x_3|}+\frac{|x_3|}{|I^1||x_2|}}, \quad & \beta_3=0, \\
\frac{|I^1|}{|x_2||x_3|^2} \frac{|x_3|}{|I^1||x_2|}
\sim \frac{|I^1|}{|x_2||x_3|^2}\frac{1}{\frac{|I^1||x_2|}{|x_3|}+\frac{|x_3|}{|I^1||x_2|}}, \quad & \beta_2=0, \beta_3=1.
\end{cases}
$$

Assume then that $|I^1||x_2|/|x_3| < 1$. This time we integrate over $I^1 \times I^1$. Proceeding as above
we reduce to the integral
$$
\iint \displaylimits_{t^{-1}I^1\times  t^{-1}I^1} 
\frac{L_{(0, \beta_2, \beta_3)}(x_1-y_1,1,1)}{(|x_1-y_1|+1)^{2+\beta_3}} \ud y_1 \ud x_1 \frac{1}{|x_2|^{2+\beta_2}|x_3|^{\beta_3}}.
$$
Notice that here $|x_1-y_1| \le 1$, so the denominator is $\sim1$. Using the estimate
$$
\int_0^s \log \Big( \frac 1u \Big) \ud u \lesssim s\Big(1+ \log \Big( \frac 1s \Big) \Big), \quad s \in (0,1],
$$
we get 
$$
\begin{cases}
\Big(\frac{|I^1||x_2|}{|x_3|}\Big)^2 \Big(1+ \log_+\Big(\frac{|x_3|}{|I^1||x_2|}\Big) \Big) \frac{1}{|x_2|^{2}|x_3|^{\beta_3}}
\sim \frac{|I^1|}{|x_2||x_3|^{1+\beta_3}} \frac{1+ \log_+\Big( \frac{|x_3|}{|I^1||x_2|} \Big)}{\frac{|I^1||x_2|}{|x_3|}+\frac{|x_3|}{|I^1||x_2|}},
\quad &\beta_2=0, \\
\Big(\frac{|I^1||x_2|}{|x_3|}\Big)^2 \frac{1}{|x_2|^{3}}
\sim \frac{|I^1|}{|x_2|^2|x_3|}\frac{1}{\frac{|I^1||x_2|}{|x_3|}+\frac{|x_3|}{|I^1||x_2|}}, \quad  &\beta_2=1, \beta_3=0.
\end{cases}
$$

\end{proof}

\begin{lem}\label{E:lem5}
Let $m \in \calM^1_Z$ and 
denote by $T_{m}$ the corresponding Fourier multiplier operator.

Let $f_1,g_1 \in L^2(\R)$ and $f_{2,3}, g_{2,3} \in L^2(\R^2)$.
Then
$$
\langle T_{m} (f_1 \otimes f_{2,3}), g_1 \otimes g_{2,3} \rangle
= \ave{T_{m_{f_{2,3}, g_{2,3}}} f_1,g_1},
$$ 
where $m_{ f_{2,3}, g_{2,3}}$ is a standard H\"ormander-Mihlin multiplier in $\R$ satisfying the estimates
$$
|(\ud/\ud \xi_1)^{\alpha} m_{f_{2,3}, g_{2,3}}(\xi_1)|
\lesssim \| m \|_{\calM_Z^1}\|f_{2,3}\|_{L^2} \| g_{2,3}\|_{L^2} |\xi_1|^{-\alpha}.
$$
Thus,  $T_{m_{f_{2,3}, g_{2,3}}}$ is a convolution form Calder\'on-Zygmund operator.
In particular, there exists a standard Calder\'on-Zygmund kernel $K_{m, f_{2,3}, g_{2,3}}$ (see \eqref{E:eq91}) 
such that  $\|K_{m, f_{2,3}, g_{2,3}}\|_{CZ_1(\R^2)} \lesssim \|f_{2,3}\|_{L^2} \| g_{2,3}\|_{L^2}$. Moreover, if
$f_1$ and  $g_1$ are supported for example in disjoint intervals, then
$$
\langle T_{m} (f_1 \otimes f_{2,3}), g_1 \otimes g_{2,3} \rangle
= \iint K_{m, f_{2,3}, g_{2,3}}(x_1,y_1) f_1(y_1) g_1(x_1) \ud y_1 \ud x_1.
$$
\end{lem}

\begin{proof}
We have that
\begin{equation*}
\begin{split}
\int T_{m} (f_1 \otimes f_{2,3})(x) (g_1 \otimes g_{2,3})(x)\ud x 
&=\int m(\xi) \wh f_1(\xi_1)  \wh f_{2,3}(\xi_{2,3}) \wh g_1(-\xi_1) \wh g_{2,3}(-\xi_{2,3})\ud \xi  \\
&= \int  m_{f_{2,3}, g_{2,3}}(\xi_1) \wh f_1(\xi_1)\wh g_1(-\xi_1) \ud \xi_1,
\end{split}
\end{equation*}
where
$$
m_{f_{2,3}, g_{2,3}}(\xi_1)
=\int m(\xi)   \wh f_{2,3}(\xi_{2,3})  \wh g_{2,3}(-\xi_{2,3})\ud \xi_{2,3}.
$$
Notice that
\begin{equation*}
\begin{split}
|(\ud/\ud \xi_1)^{\alpha} &m_{ f_{2,3}, g_{2,3}}(\xi_1)|
=\Big| \int \partial_1^{\alpha}m(\xi)   \wh f_{2,3}(\xi_{2,3})  \wh g_{2,3}(-\xi_{2,3})\ud \xi_{2,3} \Big| \\
& \lesssim \| \partial_1^{\alpha}m(\xi_1, \cdot) \|_{L^{\infty}(\R^2)} \|f_{2,3}\|_{L^2} \| g_{2,3}\|_{L^2} 
\le  \| m \|_{\calM_Z^1} |\xi_1|^{-\alpha} \|f_{2,3}\|_{L^2} \| g_{2,3}\|_{L^2},
\end{split}
\end{equation*}
where in the last step we used \eqref{E:eq66}.
It is a well known fact that this estimate implies the rest of the statements of the lemma, 
see for example Stein \cite[p. 245, Proposition 2]{Stein:book}.

\end{proof}

\section{Zygmund singular integrals}\label{E:sec1}
We denote by $CZ_\gamma(\R)$ the standard Calder\'on-Zygmund kernels in $\R$ which satisfy the H\"older
continuity estimates with exponent $\gamma\in (0,1]$. More precisely, we have that $K \in CZ_\gamma (\R)$ if $K$
is a function
\begin{equation}\label{E:eq91}
K \colon \R \times \R \setminus \{(x,y) \colon x=y\} \to \C
\end{equation}
that satisfies the size estimate
\begin{equation}\label{E:eq115}
|K(x,y)| \le \frac{C}{|x-y|}
\end{equation}
and the H\"older estimates
\begin{equation}\label{E:eq109}
|K(x',y) -K(x,y)|\le C\frac{|x'-x|^\gamma}{|x-y|^{1+\gamma}} \quad \text{and} \quad |K(y,x')-K(y,x)| \le C\frac{|x'-x|^\gamma}{|x-y|^{1+\gamma}}
\end{equation}
whenever $|x'-x| \le |x-y|/2$.
The smallest possible constant in these inequalities is denoted by $\| K \|_{CZ_\gamma(\R)}$.
 
Let $T$ be a linear operator defined on an appropriate class of functions on $\R^3$. The class should at least include
finite linear combinations of indicators of rectangles. We define what it means for $T$ to be a Zygmund singular integral.
Let $f=1_{I}$ and $g=1_{J}$, where $I=I^1 \times I^2 \times I^3=I^1 \times I^{2,3}$ is a rectangle (not necessarily Zygmund) and 
similarly with $J$. We say that two intervals or rectangles are disjoint if their interiors are disjoint.

\subsection{Full kernel representation}\label{E:subsec6}
The operator $T$ is related to a full kernel $K$ in the following way. 
The kernel $K$ is  a function
$
K \colon (\R^3 \times \R^3) \setminus \Delta \to \C,
$
where
$$
\Delta=\{(x,y) \in \R^3 \times \R^3 \colon |x_1-y_1||x_2-y_2||x_3-y_3|=0\}. 
$$
If $I^1$ and  $J^1$ are disjoint and also $I^{2,3}$ and $J^{2,3}$ are disjoint, then
$$
\ave {T1_I,1_J}
= \iint K(x,y) 1_I(y) 1_J(x) \ud y \ud x.
$$
The kernel $K$ satisfies the following estimates.

Let $(x,y) \in (\R^3 \times \R^3)\setminus \Delta$. First, we assume that $K$ satisfies the size estimate
\begin{equation}\label{E:eq30}
|K(x,y)|
\lesssim \frac{1}{|x_1-y_1||x_2-y_2||x_3-y_3| }D_\theta(x_1-y_1, x_2-y_2, x_3-y_3), 
\end{equation}
where we recall that 
\begin{equation}\label{E:eq120}
  D_\theta(x):=D_\theta(x_1,x_2,x_3):=\Big(\frac{\abs{x_1x_2}}{\abs{x_3}}+\frac{\abs{x_3}}{\abs{x_1x_2}}\Big)^{-\theta},\qquad \theta\in (0,1].
\end{equation}
Let $x'=(x_1',x_2',x_3')$ be such that $|x_i'-x_i| \le |x_i-y_i|/2$ for $i=1,2,3$. We assume that $K$ satisfies the mixed size and H\"older estimates
\begin{equation}\label{E:eq25}
\begin{split}
|K((x'_1,x_2,x_3),y)-K(x,y)|
&\lesssim  \Big(\frac{|x_1'-x_1|}{|x_1-y_1|}\Big)^{\alpha_1}\frac{D_\theta(x_1-y_1, x_2-y_2, x_3-y_3)}{|x_1-y_1||x_2-y_2||x_3-y_3| }
\end{split}
\end{equation}
and
\begin{equation}\label{E:eq28}
\begin{split}
|K((x_1&,x_2',x_3'),y)-K(x,y)|
\\
&\lesssim  \Big(\frac{|x_2'-x_2|}{|x_2-y_2|}+ \frac{|x_3'-x_3|}{|x_3-y_3|} \Big)^{\alpha_{23}}
\frac{D_\theta(x_1-y_1, x_2-y_2, x_3-y_3)}{|x_1-y_1||x_2-y_2||x_3-y_3| },
\end{split}
\end{equation}
where $\alpha_1, \alpha_{23} \in (0,1]$.
Finally, we assume that  $K$ satisfies the H\"older estimate
\begin{equation}\label{E:eq29}
\begin{split}
|K(x',&y)-K((x_1',x_2,x_3),y)-K((x_1,x_2',x_3'),y)+K(x,y)| \\
&\lesssim  \Big(\frac{|x_1'-x_1|}{|x_1-y_1|}\Big)^{\alpha_1}\Big(\frac{|x_2'-x_2|}{|x_2-y_2|}
+\frac{|x_3'-x_3|}{|x_3-y_3|}\Big)^{\alpha_{23}}   \frac{D_\theta(x_1-y_1, x_2-y_2, x_3-y_3)}{|x_1-y_1||x_2-y_2||x_3-y_3| }.
\end{split}
\end{equation}

Define the adjoint kernels $K^*$, $K^*_1$ and $K^*_{2,3}$ via the following relations:
\begin{equation*}
K^*(x,y)=K(y,x), \quad
K^*_1(x,y)=K((y_1,x_2,x_3),(x_1,y_2,y_3))
\end{equation*}
and
\begin{equation}\label{E:eq75}
K^*_{2,3}(x,y)=K((x_1,y_2,y_3),(y_1,x_2,x_3)).
\end{equation}
We assume that each adjoint kernel satisfies the same estimates as the kernel $K$.

\subsection{Partial kernel representations}\label{E:subsec5}

Related to every interval $I^1$ there exists a kernel 
$$
K_{I^1} \colon (\R^2 \times \R^2) \setminus \{(x_{2,3}, y_{2,3}) \colon |x_2-y_2||x_3-y_3|=0\} \to \C,
$$
so that if $I^{2,3}$ and   $J^{2,3}$ are disjoint, then
\begin{equation}\label{E:eq67}
\langle T(1_{I^1} \otimes 1_{I^{2,3}}), 1_{I^1} \otimes 1_{J^{2,3}}\rangle
= \iint K_{I^1}(x_{2,3}, y_{2,3}) 1_{I^{2,3}}(y_{2,3})1_{J^{2,3}}(x_{2,3}) \ud y_{2,3} \ud x_{2,3}.
\end{equation}
The kernel $K_{I^1}$ is assumed to satisfy the following estimates.
Let $x_{2,3}, y_{2,3}$ be such that $|x_2-y_2||x_3-y_3|\not=0$. First, we assume the size estimate
\begin{equation}\label{E:eq62}
|K_{I^1}(x_{2,3}, y_{2,3}) | 
\lesssim \frac{|I^1|}{|x_2-y_2||x_3-y_3|} D_{\theta}(|I^1|, x_2-y_2,x_3-y_3).
\end{equation}
Let $x'_{2,3}=(x_2',x_3')$ be such that $|x_i'-x_i| \le |x_i-y_i|/2$ for $i=2,3$. 
We assume the H\"older estimate
\begin{equation}\label{E:eq63}
\begin{split}
&|K_{I^1}(x_{2,3}',y_{2,3})-K_{I^1}(x_{2,3},y_{2,3})|  \\
&\hspace{1cm}\lesssim 
\Big( \frac{|x_2'-x_2|}{|x_2-y_2|}+\frac{|x_3'-x_3|}{|x_3-y_3|} \Big)^{\alpha_{23}}\frac{|I^1|}{|x_2-y_2||x_3-y_3|} 
D_{\theta}(|I^1|, x_2-y_2,x_3-y_3). 
\end{split}
\end{equation}
We assume that the adjoint kernel defined by $K_{I^1}^*(x_{2,3}, y_{2,3})
=K_{I^1}(y_{2,3}, x_{2,3})$ satisfies the same estimates. 

Now we consider the other partial kernel representation.
Related to every rectangle $I^{2,3}$ there exists a standard Calder\'on-Zygmund kernel $K_{I^{2,3}} \in CZ_{\alpha_1}(\R)$ (see \eqref{E:eq91}) 
so that if $I^1$ and  $J^1$ are disjoint, then
\begin{equation}\label{E:eq68}
\langle T(1_{I^1} \otimes 1_{I^{2,3}}), 1_{J^1} \otimes 1_{I^{2,3}} \rangle
= \iint K_{I^{2,3}}(x_1,y_1) 1_{I^1}(y_1) 1_{J^1}(x_1) \ud y_1 \ud x_1.
\end{equation}
The kernel $K_{I^{2,3}}$ satisfies 
\begin{equation}\label{E:eq110}
\begin{split}
\| K_{I^{2,3}}\|_{CZ_{\alpha_1}(\R)} \lesssim |I^{2,3}|.
\end{split}
\end{equation}

\subsection{Cancellation assumptions}\label{E:subsec8} 
Let  $h_{I^1}$ and $h_{I^{2,3}}$
be cancellative Haar functions, see \ref{E:subsec2}. We say that $T$ satisfies the cancellation assumptions if
\begin{equation}\label{E:eq143}
\begin{split}
\ave{T(1 \otimes 1_{I^{2,3}}), h_{I^1} \otimes 1_{J^{2,3}}}
&=\ave{T(h_{I^1} \otimes 1_{I^{2,3}}), 1 \otimes 1_{J^{2,3}}}
=\ave{T(1_{I^1} \otimes 1), 1_{J^1} \otimes h_{I^{2,3}}} \\
&=\ave{T(1_{I^1} \otimes h_{I^{2,3}}), 1_{J^1} \otimes 1}=0.
\end{split}
\end{equation}
See Section \ref{E:subsec7} for a discussion about how these pairings are defined.

In this paper, our Zygmund singular integral operators will always, without further mention, satisfy the above cancellation conditions.
A further generalization would aim to replace these cancellation assumptions with appropriate $T1 \in \BMO$ style assumptions -- however, the well--known
examples fall already into our current class.

The reader might expect that we would also need assumptions of the form
$$
\ave{T1, h_{I^1} \otimes h_{I^{2,3}}}
=\ave{T(1 \otimes h_{I^{2,3}}), h_{I^1} \otimes 1}
= \ave{T(h_{I^1} \otimes 1),  1 \otimes h_{I^{2,3}}}
=\ave{T (h_{I^1} \otimes h_{I^{2,3}}),1}=0.
$$
However, already the conditions in \eqref{E:eq143} imply that all the paraproduct type operators vanish in the representation theorem,
see Section \ref{E:subsec10}.

\subsection{Weak boundedness property}
We say that the operator $T$ satisfies the weak boundedness property if for all Zygmund rectangles $I$ there holds that
\begin{equation}\label{E:eq74}
| \langle T1_I, 1_I \rangle| \lesssim |I|.
\end{equation}

\begin{defn}
We say that a linear operator $T$ is a Calder\'on-Zygmund operator adapted to Zygmund dilations (CZZ operator)
if $T$ has the full kernel representation, 
the partial kernel representations and satisfies the cancellation assumptions and the weak boundedness property.
To emphasize the related parameters we may say that $T$ is a $(\theta, \alpha_1, \alpha_{23})$-CZZ operator.
\end{defn}

\subsection{A slight extension}
As one can already notice from Section \ref{E:sec3}, the Fefferman-Pipher multipliers do not
strictly speaking belong to the class defined above due to the logarithmic factors that appear
in the multiplier estimates. We define here a class which includes the multipliers, 
see Section \ref{E:subsec3}.

Recall the definition of $D_{\theta}$ in \eqref{E:eq120}. Let 
$$
D_{\log}(x):=D_1(x) \log \Big(\frac{|x_1x_2|}{|x_3|} + \frac{|x_3|}{|x_1x_2|}\Big).
$$
The new class is simply defined similarly as we did above, except that
we replace everywhere $D_\theta$ with $D_{\log}$. So this affects the full kernel estimates and 
the partial kernel estimates in Section \ref{E:subsec5}. Again, we allow the H\"older exponents 
$\alpha_1,\alpha_{23} \in (0,1]$ to be general.
We say that an operator $T$ in this class is 
a $(\log, \alpha_1, \alpha_{23})$-CZZ operator.

We can easily cover the relevant examples involving the logarithmic factor, see Corollary \ref{E:cor1}.
Nevertheless, for the majority of the paper we do not explicitly need to consider this case --
see the proof of Corollary \ref{E:cor1} for the details of the reduction.

\subsection{Concrete examples in our set-up}
\subsubsection{The Nagel-Wainger singular integral}
In Nagel-Wainger \cite{NW}, it was shown that the convolution operator $T_{NW}=f*K$ with
$$
K_{NW}(x_1,x_2,x_3) 
=\frac{\sign{(x_1x_2)}}{x_1^2x_2^2+x_3^2}
$$
is bounded in $L^2(\R^3)$. The kernel $K_{NW}$ satisfies  
the full set of assumptions of Theorem \ref{thm:HLLT}, 
see \cite[Remark 4.5, p.\ 2453]{HLLT}. 
In Section \ref{E:subsec9} we show that the operators satisfying the assumptions
of Theorem \ref{thm:HLLT} belong to the class we consider in this paper.

We indicate briefly how one could also quite directly show that $T_{NW}$ belongs to our class. 
Indeed, $K_{NW}$ satisfies the estimates
required of a full kernel as described in Section \ref{E:subsec6}.
It satisfies even better full kernel estimates which can be seen by differentiating it.
The antisymmetry present in $K_{NW}$ implies that one could take all the partial kernels required in
Section \ref{E:subsec5} to be zero. The antisymmetry also implies that the weak boundedness property
 \eqref{E:eq74} holds in the form
$\ave{T_{NW}1_I,1_I}=0$. Finally, since $T_{NW}$ is a principal value integral
and of convolution form one sees that $T_{NW}$ satisfies the translation invariance properties
\eqref{E:eq144} and \eqref{E:eq131}, which implies, by Section \ref{E:subsec7}, that $T_{NW}$ satisfies the cancellation assumptions of 
Section \ref{E:subsec8}.

\subsubsection{Fefferman-Pipher multipliers}\label{E:subsec3}
\begin{prop}\label{E:prop5}
Suppose that $m \in \calM^1_Z$ and let $m_J$ be a related truncation, see \eqref{E:eq127}.
Let $K_J=\check m_J$ be its inverse Fourier transform.
Then $T_Jf=K_J *f$ is a $(\log,1,1)$-CZZ operator.
\end{prop}

\begin{proof}
Notice that $K_J$ is integrable. 
The full kernel of $T_J$ is $(x,y) \mapsto K_J(x-y)$. The calculations in Section \ref{E:subsec1}
imply that this kernel satisfies the estimates required of a full kernel in Section \ref{E:subsec6}.

Let $I^1$ be an interval and let $I^{2,3}$ and $J^{2,3}$ be disjoint rectangles.
Then
$$
\langle T_J(1_{I^1} \otimes 1_{I^{2,3}}), 1_{I^1} \otimes 1_{J^{2,3}}\rangle
= \iint K_{I^1}(x_{2,3}, y_{2,3}) 1_{I^{2,3}}(y_{2,3})1_{J^{2,3}}(x_{2,3}) \ud y_{2,3} \ud x_{2,3},
$$
where
$$
K_{I^1}(x_{2,3}, y_{2,3})
=\iint \displaylimits_{I^1\times I^1}  K_J(x_1-y_1,x_2-y_2,x_3-y_3) \ud y_1 \ud x_1.
$$
Lemma \ref{E:lem3} implies that $K_{I^1}$ satisfies the estimates \eqref{E:eq62} and \eqref{E:eq63}
with $\alpha_{23}=1$ and with $D_{\log}$ instead of $D_\theta$.
Lemma \ref{E:lem5} implies that $T_J$ satisfies the other partial kernel representation \eqref{E:eq68}
and that the related estimate \eqref{E:eq110} holds.

Since $m_J \in L^\infty$ one directly sees via the Fourier transform that $T_J$ is bounded in $L^2$.
This implies that the weak boundedness property \eqref{E:eq74} holds.

Finally, $T_J$ satisfies the cancellation assumptions of Section \ref{E:subsec8}. This follows 
from the argument in Section \ref{E:subsec7}. Indeed, since $T_J f = K_J * f$ and the kernel $K_J$ is integrable, 
it is easy to verify that $T_J$ satisfies the translation invariance properties \eqref{E:eq144}
and \eqref{E:eq131}. Also, due to the integrability of $K_J$, it is easy to verify the cancellation assumptions directly.
\end{proof}

\subsubsection{The class of Han, Li, Lin and Tan \cite{HLLT}}\label{E:subsec9}
In this section we show that the operators considered in \cite{DLOPW2, HLLT} belong to our class.

\begin{prop}\label{E:eqlem6}
Suppose that $K$ is a kernel satisfying the conditions \eqref{E:eq132}, \eqref{E:eq133},
\eqref{E:eq134} with $j=1$ and \eqref{E:eq135}. Suppose also that $K$ has relevant limits in the sense of Definition \ref{def:limits}.
Consider the operator
$$
Tf =\lim_{\substack{\eps_1, \eps_2, \eps_3 \to 0 \\ N_1,N_2,N_3 \to \infty}} K^N_\eps * f
= \lim _{\substack{\varepsilon_i \to 0
\\ N_i \to \infty}} K^N_{\varepsilon} *f,
$$ 
where the truncations $K^N_\varepsilon$ are
defined in \eqref{E:eq136}.
Recall the parameters $\theta$ and $\alpha$ related to the kernel $K$. 
If $\theta<1$, then
$T$ is a $(\theta, \alpha, \alpha)$-CZZ operator. If $\theta =1$, then
$T$ is a $(\log, \alpha, \alpha)$-CZZ operator.
\end{prop}

\begin{proof}
We use from Theorem \ref{thm:HLLT} the fact that the
limit defining $T$ exists in $L^2$ and $T$ is a bounded operator in $L^2$.
In particular, $T$ is well defined on indicators of rectangles.

The condition \eqref{E:eq132} implies that the kernel $(x,y) \mapsto K(x-y)$ satisfies the full kernel estimates
of Section \ref{E:subsec6}. Indeed, the size estimate \eqref{E:eq30} and the mixed H\"older and
size estimate \eqref{E:eq25} are immediate. Consider the other mixed estimate
\eqref{E:eq28}. We have that
\begin{equation*}
\begin{split}
|K((x_1,x_2',x_3'),y)-K(x,y)|
&\le |K((x_1,x_2',x_3'),y)-K((x_1,x_2',x_3),y)| \\
&+|K((x_1,x_2',x_3),y)-K(x,y)|.
\end{split}
\end{equation*}
Here one can apply condition \eqref{E:eq132} to both terms on the right hand side which gives \eqref{E:eq28}.
Similar reasoning implies the H\"older estimate \eqref{E:eq29}.

We turn to the partial kernel estimates and begin with \eqref{E:eq67}.
Let $I^1$ be an interval and let $I^{2,3}$ and  $J^{2,3}$ be disjoint rectangles. 
Then 
\begin{equation}\label{E:eq139}
\begin{split}
&\langle T(1_{I^1} \otimes 1_{I^{2,3}}), 1_{I^1} \otimes 1_{J^{2,3}}\rangle\\
&= \lim_{\substack{\eps_i \to 0\\ N_i \to \infty}}\iint K_{\eps}^N(x_1-y_1, x_2-y_2, x_3-y_3) 1_{I^1\times I^{2,3}}(y)1_{I^1\times J^{2,3}}(x)\ud y  \ud x\\
&= \lim_{\substack{\eps_i \to 0\\ N_i \to \infty}}
\iint 1_{I^{2,3}}(y_{2,3})1_{J^{2,3}}(x_{2,3})
\iint \displaylimits_{I^1\times I^1} K_{\eps}^N(x_1-y_1, x_2-y_2, x_3-y_3) \ud y_1  \ud x_1 \ud y_{2,3}\ud x_{2,3},
\end{split}
\end{equation} 
where in the first step we used the fact that the limit defining $T$ exists in $L^2$.

We show that the limit
\begin{equation}\label{E:eq138}
\begin{split}
K_{I^1}&(x_2-y_2,x_3-y_3) \\
&:= \lim_{\substack{\eps_2, \eps_3\to 0\\ N_2,N_3\to \infty}}\lim_{\substack{\eps_1\to 0\\ N_1\to \infty}}
\iint_{I^1\times I^1} K_{\eps}^N(x_1-y_1, x_2-y_2, x_3-y_3) \ud y_1  \ud x_1
\end{split}
\end{equation}
exists and 
\begin{equation}\label{E:eq211}
\begin{split}
&\lim_{\substack{\eps_2, \eps_3 \to 0\\ N_2, N_3 \to \infty}} \lim_{\substack{\eps_1\to 0 \\ N_1\to \infty}}\iint 1_{I^{2,3}}(y_{2,3})1_{J^{2,3}}(x_{2,3}) \\
&
\hspace{3cm}\iint \displaylimits_{I^1\times I^1} K_{\eps}^N(x_1-y_1, x_2-y_2, x_3-y_3) \ud y_1  \ud x_1 \ud y_{2,3}\ud x_{2,3}\\
&\quad=\iint K_{I^1}(x_2-y_2,x_3-y_3) 1_{I^{2,3}}(y_{2,3})1_{J^{2,3}}(x_{2,3})
  \ud y_{2,3}\ud x_{2,3}.
  \end{split}
\end{equation}
We will also show that
\begin{equation}\label{E:eq137}
K_{I^1}(x_2-y_2,x_3-y_3)
\lesssim \begin{cases}
\frac{|I^1|}{|x_2-y_2||x_3-y_3|} D_\theta (|I^1|,x_2-y_2,x_3-y_3), \quad & \theta < 1, \\
\frac{|I^1|}{|x_2-y_2||x_3-y_3|} D_{\log} (|I^1|,x_2-y_2,x_3-y_3), \quad & \theta = 1. 
\end{cases}
\end{equation}
Note that since the two limits in \eqref{E:eq139} and \eqref{E:eq211} exist they must coincide.
This  implies that $K_{I^1}$ is a
partial kernel of $T$. 
Moreover, $K_{I^1}$ satisfies the size estimate as required
in \eqref{E:eq62}, with $D_{\log}$ instead of $D_1$ if $\theta=1$.
The H\"older estimate \eqref{E:eq63} of $K_{I^1}$ is proved exactly as we will prove \eqref{E:eq137};
for notational convenience we prove the size estimate.

We prove that
\begin{equation}\label{E:eq142}
\lim_{\substack{\eps_1\to 0\\ N_1\to \infty}}
\iint_{I^1\times I^1} K_{\eps}^N(x_1-y_1, u_2, u_3) \ud y_1  \ud x_1
\end{equation}
exists and is dominated by $1_{(\eps_2,N_2)} (u_2)1_{(\eps_3, N_3)} (u_3)$ multiplied with
\begin{equation}\label{E:eq140}
\begin{cases}
\frac {|I^1|}{|u_2||u_3|} D_{\theta}(|I^1|, u_2, u_3), \quad &\theta<1,\\
\frac {|I^1|}{|u_2||u_3|}D_{\log}(|I^1|, u_2, u_3), \quad &\theta=1.
\end{cases}
\end{equation}
From here one sees that the limit \eqref{E:eq138} exists and satisfies \eqref{E:eq137}.

Consider
$$
\iint_{I^1\times I^1} K_{\eps}^N(x_1-y_1, u_2, u_3) \ud y_1  \ud x_1.
$$
By a change of variables we may assume that $I^1$ is centred at the origin. 
 Below we write 
\[
d_1=d_1(y_1)= \dist (y_1, \partial I^1)=\frac{|I^1|}2-|y_1|, \quad y_1 \in I^1.
\]
Denote by $I^1_l$ and $I^1_r$ the left and right halves of $I^1$. 
For the moment write $f(u_1)=K^N_{\eps}(u_1,u_2,u_3)$. We have that
 \begin{align*}
& \iint_{I^1\times I^1} f(x_1-y_1) \ud y_1  \ud x_1 
=\int_{I^1}\ud y_1 \int_{I^{1}-y_1}f(x_1) \ud x_1\\
 &=\int_{I^1}\ud y_1 \int_{-d_1}^{d_1} f(x_1) \ud x_1+\int_{I^1_l}\ud y_1 \int_{d_1}^{|I^1|-d_1}f(x_1) \ud x_1
 + \int_{I^1_r}\ud y_1 \int_{d_1-|I^1|}^{-d_1}f(x_1) \ud x_1\\
 &=\int_{I^1}\ud y_1 \int_{-d_1}^{d_1} f(x_1) \ud x_1+\frac 12\int_{I^1}\ud y_1 \int_{d_1}^{|I^1|-d_1}f(x_1) \ud x_1
 + \frac 12\int_{I^1}\ud y_1 \int_{d_1-|I^1|}^{-d_1}f(x_1) \ud x_1\\
 &=\frac 12 \int_{I^1}\ud y_1 \int_{-d_1}^{d_1} f(x_1) \ud x_1+\frac 12\int_{I^1}\ud y_1 \int_{ |x_1|\le |I^1|-d_1}f(x_1) \ud x_1.
 \end{align*}
 
We  estimate  the second term in the previous line. Write again $K^N_{\eps}(u_1,u_2,u_3)$
instead of $f(u_1)$. Assume that $N_1 \ge |I^1|$. By condition \eqref{E:eq135} there holds that 
 \begin{align*}
&\Big| \int_{I^1}\int_{ |x_1|\le |I^1|-d_1}K_{\eps}^N(x_1, u_2, u_3) \ud x_1\ud y_1 \Big|\\
&= \Big| \int_{I^1} \int_{ \eps_1< |x_1|\le |I^1|-d_1}K(x_1, u_2, u_3) \ud x_1\ud y_1 \Big|1_{(\eps_2,N_2)} (u_2)1_{(\eps_3, N_3)} (u_3)\\
&\lesssim  \frac{|I^1|}{|u_2||u_3|}(D_{\theta}(|I^1|, u_2, u_3)+D_{\theta}(\eps_1, u_2, u_3) )1_{(\eps_2,N_2)} (u_2)1_{(\eps_3, N_3)} (u_3),
 \end{align*}where we have used that $|I^1|-d_1\sim |I^1|$. Likewise, we have for $\eps_1<\eps_1' < \frac{|I^1|}{2}$ that 
 \begin{align*}
 &\Big| \int_{I^1} \int_{ |x_1|\le |I^1|-d_1}K_{\eps}^N(x_1, u_2, u_3) \ud x_1\ud y_1 -\int_{I^1} \int_{ |x_1|\le |I^1|-d_1}K_{(\eps_1', \eps_2, \eps_3)}^N(x_1, u_2, u_3) \ud x_1 \ud y_1 \Big|\\
 &=  \Big| \int_{I^1} \int_{ \eps_1< |x_1|\le \eps_1'}K(x_1, u_2, u_3) \ud x_1\ud y_1 \Big|1_{(\eps_2,N_2)} (u_2)1_{(\eps_3, N_3)} (u_3)\\
 &\lesssim \frac{|I^1|}{|u_2||u_3|}(D_{\theta}(\eps_1', u_2, u_3)+D_{\theta}(\eps_1, u_2, u_3) )1_{(\eps_2,N_2)} (u_2)1_{(\eps_3, N_3)} (u_3).
 \end{align*}Then since 
 $
 \lim_{\eps_1\to 0} D_{\theta}(\eps_1, u_2, u_3)=0,  
 $
 by the Cauchy convergence principle we have that
 \[
 \lim_{\substack{\eps_1\to 0\\ N_1\to \infty}} \int_{I^1} \int_{ |x_1|\le |I^1|-d_1}K_{\eps}^N(x_1, u_2, u_3) \ud x_1 \ud y_1
 \] exists, and it is controlled by 
 \[
 \frac{|I^1|}{|u_2||u_3|}D_{\theta}(|I^1|, u_2, u_3)1_{(\eps_2,N_2)} (u_2)1_{(\eps_3, N_3)} (u_3).
 \]
 
Similarly, we have 
\begin{align*}
&\Big|\int_{I^1} \int_{-d_1}^{d_1} K_{\eps}^N(x_1, u_2, u_3) \ud x_1\ud y_1\Big|\\
&\hspace{1cm}\lesssim \int_{I^1} \frac{1}{|u_2||u_3|}(D_{\theta}(d_1, u_2, u_3)+D_{\theta}(\eps_1, u_2, u_3) )1_{(\eps_2,N_2)} (u_2)1_{(\eps_3, N_3)} (u_3) \ud y_1
\end{align*}
and for $\eps_1<\eps_1' < \frac{|I^1|}{2}$,
\begin{align*}
&\Big| \int_{I^1} \int_{ |x_1|\le  d_1}K_{\eps}^N(x_1, u_2, u_3) \ud x_1\ud y_1 -\int_{I^1} \int_{ |x_1|\le d_1}K_{(\eps_1', \eps_2, \eps_3)}^N(x_1, u_2, u_3) \ud x_1 \ud y_1 \Big|\\
 &=\Big|\int_{|y_1|<\frac{|I^1|}2-\eps_1'}\int_{ \eps_1<|x_1|\le  \eps_1'}K_{\eps}^N(x_1, u_2, u_3) \ud x_1\ud y_1 \\
 &\hspace{2cm}
 + \int_{\frac{|I^1|}2-\eps_1'\le |y^1|\le \frac{|I^1|}2}\int_{ \eps_1<|x_1|\le  d_1}K_{\eps}^N(x_1, u_2, u_3) \ud x_1\ud y_1 \Big|\\
 &\lesssim \Big[\frac{|I^1|}{|u_2||u_3|}(D_{\theta}(\eps_1', u_2, u_3)+D_{\theta}(\eps_1, u_2, u_3) )
 + \frac{\eps_1'}{|u_2||u_3|}\Big]1_{(\eps_2,N_2)} (u_2)1_{(\eps_3, N_3)} (u_3),
\end{align*}
where we simply dominated the $D_\theta$ factors by $1$ in the last term.
Hence
\[
 \lim_{\substack{\eps_1\to 0\\ N_1\to \infty}} \int_{I^1} \int_{-d_1}^{d_1} K_{\eps}^N(x_1, u_2, u_3) \ud x_1 \ud y_1
\]exists and it is controlled by $1_{(\eps_2,N_2)} (u_2)1_{(\eps_3, N_3)} (u_3)$ times
\begin{align*}
\int_{I^1} \frac{1}{|u_2||u_3|}D_{\theta}(d_1, u_2, u_3)\ud y_1
&= \frac{1}{|u_2||u_3|} \int_{I^1} D_{\theta}\Big(\frac{|I^1|}2-|y_1|, u_2, u_3\Big)\ud y_1
=:I+II,
\end{align*}
where 
\begin{align*}
I&= \frac{1}{|u_2||u_3|} \int_{I^1: (\frac{|I^1|}2-|y_1|)|u_2|>|u_3|} D_{\theta}\Big(\frac{|I^1|}2-|y_1|, u_2, u_3\Big)\ud y_1,\\
II&= \frac{1}{|u_2||u_3|} \int_{I^1: (\frac{|I^1|}2-|y_1|)|u_2|\le |u_3|} D_{\theta}\Big(\frac{|I^1|}2-|y_1|, u_2, u_3\Big)\ud y_1.
\end{align*}
Note that in $I$ there holds that $|u_3|/{|u_2|}<|I^1|/2$. We have 
\begin{align*}
I&\le\frac{1}{|u_2||u_3|}\int_{|y_1|< \frac{|I^1|}2-\frac{|u_3|}{|u_2|} }\frac 1{( \frac{|I^1|}2-|y_1|)^{\theta}}\frac{|u_3|^{\theta}}{|u_2|^{\theta}}\ud y_1\\
&= \frac{2}{|u_2||u_3|}\int_{|u_3|/{|u_2|}}^{|I^1|/2}\frac 1{y_1^{\theta}} \frac{|u_3|^{\theta}}{|u_2|^{\theta}}\ud y_1
\lesssim\begin{cases}
\frac {|I^1|}{|u_2||u_3|} D_{\theta}(|I^1|, u_2, u_3), \quad &\theta<1,\\
\frac {|I^1|}{|u_2||u_3|}D_{\log}(|I^1|, u_2, u_3), \quad &\theta=1.
\end{cases}
\end{align*}
Next, we estimate $II$ and we first consider the case $|u_3|/{|u_2|}\ge |I^1|/2$. There holds that
\begin{align*}
II&\le \frac{1}{|u_2||u_3|}\int_{|y_1|\le \frac{|I^1|}{2}} \Big( \frac{|I^1|}2-|y_1|\Big)^{\theta}\frac{|u_2|^{\theta}}{|u_3|^{\theta}}\ud y_1
\lesssim \frac{|I^1|}{|u_2||u_3|}D_{\theta}(|I^1|, u_2, u_3).
\end{align*}If $|u_3|/{|u_2|}<|I^1|/2$, then 
\begin{align*}
II&\le \frac{1}{|u_2||u_3|}\int_{  \frac{|I^1|}{2}-\frac{|u_3|}{|u_2|}
\le |y_1|\le \frac{|I^1|}{2}} \Big( \frac{|I^1|}2-|y_1|\Big)^{\theta}\frac{|u_2|^{\theta}}{|u_3|^{\theta}}\ud y_1\\
&\lesssim \frac{1}{|u_2|^2}\sim \frac{|I^1|}{|u_2||u_3|}D_{1}(|I^1|, u_2, u_3)\le  \frac{|I^1|}{|u_2||u_3|}D_{\theta}(|I^1|, u_2, u_3).
\end{align*}
Therefore, we have shown that \eqref{E:eq142} exists and is controlled by \eqref{E:eq140}.

To finish the proof of the partial kernel property in question, it remains to
consider the limit \eqref{E:eq211}. Looking at the estimates we have proved, 
there is no problem in moving the limits $\eps_1 \to 0$ and $N_1 \to \infty$ inside when $\eps_2, \eps_3, N_2$ and  
$N_3$ are fixed.
The upper bounds in \eqref{E:eq140}
are integrable over $I^{2,3} \times J^{2,3}$ since $I^{2,3}$ and $J^{2,3}$ are pairwise disjoint.
Therefore, we may move the limits $\eps_2, \eps_3 \to 0$ and $N_2,N_3 \to \infty$ also inside, which proves
\eqref{E:eq211}.

Next we consider $\langle T(1_{I^1} \otimes 1_{I^{2,3}}), 1_{J^1} \otimes 1_{I^{2,3}} \rangle$, where $I^1$ and $J^1$
are disjoint. There holds that
\begin{equation}\label{E:eq148}
\begin{split}
&\langle T(1_{I^1} \otimes 1_{I^{2,3}}), 1_{J^1} \otimes 1_{I^{2,3}} \rangle\\&
= \lim_{\substack{\eps_i \to 0 \\ N_i \to \infty}}
\iint K_{\eps}^N(x_1-y_1, x_2-y_2, x_3-y_3) 1_{I^1\times I^{2,3}}(y)1_{J^1\times I^{2,3}}(x)\ud y  \ud x\\
&= \lim_{\substack{\eps_i \to 0 \\ N_i \to \infty}}
\iint 1_{I^{1}}(y_{1})1_{J^{1}}(x_{1})\ud y_{1}\ud x_{1}\iint_{I^{2,3}\times I^{2,3}} K_{\eps}^N(x_1-y_1, x_2-y_2, x_3-y_3) \ud y_{2,3}  \ud x_{2,3}.
\end{split}
\end{equation}

First we demonstrate that the limit
\begin{equation}\label{E:eq147}
K_{I^{2,3}}(u_1)
:=\lim_{\substack{\eps_i \to 0 \\ N_i \to \infty}}\iint_{I^{2,3}\times I^{2,3}} K_{\eps}^N(u_1, x_2-y_2, x_3-y_3) \ud y_{2,3}  \ud x_{2,3}
\end{equation}
exists for a.e. $u_1$. Fix some $(x_2,x_3) \in I^{2,3}$ and let $d_i=d(x_i, \partial I^i)$ and $B_i=B(x_i,d_i)$. 
We argue that
\begin{equation}\label{E:eq145}
\Big | \int_{I^{2,3}} K_{\eps}^N(u_1, x_2-y_2, x_3-y_3) \ud y_{2,3} \Big| 
\lesssim \frac{1}{|u_1|} \Big(\log \Big( \frac{|I^2|}{d_2} \Big) + \log \Big( \frac{|I^3|}{d_3} \Big) \Big).
\end{equation}
Notice that the upper bound is integrable over $I^{2,3}$.

Consider the left hand side of \eqref{E:eq145}.
Split the integration area as
$$
I^{2,3}= (B_2 \times B_3) \cup ((I^2 \setminus B_2) \times I^3) \cup (B_2 \times (I^3 \setminus B_3)).
$$
The condition \eqref{E:eq134} with $j=1$ implies that the integral over the first set is dominated by 
$
|u_1|^{-1}.
$
The size estimate of $K$ (see \eqref{E:eq132}) gives that the integral over $(I^2 \setminus B_2) \times I^3$
is dominated by
$$
\int_{(I^2 \setminus B_2) \times \R} 
\frac{\Big( \frac{|u_1||x_2-y_2|}{|x_3-y_3|} + \frac{|x_3-y_3|}{|u_1||x_2-y_2|}\Big)^{-\theta}}{|u_1||x_2-y_2||x_3-y_3|} \ud y_{2,3}
\lesssim \int_{I^2 \setminus B_2} \frac{1}{|u_1||x_2-y_2|} \ud y_2
\le \frac{1}{|u_1|} \log \Big(\frac{|I^2|}{d_2}\Big).
$$
The integral over $B_2 \times (I^3 \setminus B_3)$ is estimated similarly. This proves \eqref{E:eq145}.

Then, we show that the limit
\begin{equation}\label{E:eq146}
\lim_{\substack{\eps_i \to 0 \\ N_i \to \infty}}\int_{ I^{2,3}} K_{\eps}^N(u_1, x_2-y_2, x_3-y_3) \ud y_{2,3}
\end{equation}
exists for a.e. $x_{2,3} \in I^{2,3}$. We may assume that $x_i$ is not on the boundary of $I^i$, $i=2,3$. Define
$A=\{ y_{2,3} \colon |y_i - x_i| \le 1\} $.  We have that
\begin{equation*}
\int_{ I^{2,3}}
=\int_{ A} +\int_{ I^{2,3}\setminus A} 
-\int_{ A \setminus I^{2,3}}.
\end{equation*}
By the assumption of relevant limits, see Definition \ref{def:limits}, the integral of $K_{\eps}^N(u_1, x_2-\cdot, x_3-\cdot)$
over $A$ has a limit as each $\eps_i \to 0$ and $N_i \to \infty$. By the size estimate of $K$ we see that
$K(u_1, x_2-\cdot, x_3-\cdot)$ is integrable over $I^{2,3}\setminus A$ and $A \setminus I^{2,3}$.
Therefore, the integrals of $K_{\eps}^N(u_1, x_2-\cdot, x_3-\cdot)$
over $I^{2,3}\setminus A$ and $A \setminus I^{2,3}$ have limits as each $\eps_i \to 0$ and $N_i \to \infty$.
This implies that the limit \eqref{E:eq146} exists. Now \eqref{E:eq145} and \eqref{E:eq146} combined
imply via dominated convergence that the limit \eqref{E:eq147} exists.

Next we show that 
\begin{equation}\label{E:eq141}
\Big|\iint_{I^{2,3}\times I^{2,3}} K_{\eps}^N(u_1, x_2-y_2, x_3-y_3) \ud y_{2,3}  \ud x_{2,3} \Big|
\lesssim |I^2||I^3| |u_1|^{-1}. 
\end{equation}
Therefore, by dominated convergence, \eqref{E:eq148} and \eqref{E:eq147} give that
$$
\langle T(1_{I^1} \otimes 1_{I^{2,3}}), 1_{J^1} \otimes 1_{I^{2,3}} \rangle
=\iint K_{I^{2,3}}(y_1-x_1)1_{I^{1}}(y_{1})1_{J^{1}}(x_{1})\ud y_{1}\ud x_{1}
$$
and $K_{I^{2,3}}$ satisfies the size estimate included in \eqref{E:eq110}. The H\"older estimates
are proved in the same way.

We turn to prove \eqref{E:eq141}.
By a change of variables we can assume that both $I^2$ and $I^3$ are centred at the origin. 
Define 
\[
d_i(y_i):=d_i:=\dist(y_i, \partial I^i)=\frac{|I^i|}{2}-|y_i|, \quad i=2,3.
\]
Similarly as before, we have 
\begin{align*}
\iint_{I^{2,3}\times I^{2,3}}& K_{\eps}^N(u_1, x_2-y_2, x_3-y_3) \ud y_{2,3}  \ud x_{2,3}\\
&= \frac 14\iint_{I^2\times I^3} \ud y_2 \ud y_3\int_{|x_2|\le d_2}\int_{|x_3|\le d_3}K_{\eps}^N(u_1, x_2, x_3)\ud x_2 \ud x_3\\
&+\frac 14\iint_{I^2\times I^3} \ud y_2 \ud y_3\int_{|x_2|\le d_2}\int_{|x_3|\le |I^3|- d_3}K_{\eps}^N(u_1, x_2, x_3)\ud x_2 \ud x_3\\
&+\frac 14\iint_{I^2\times I^3} \ud y_2 \ud y_3\int_{|x_2|\le |I^2|-d_2}\int_{|x_3|\le d_3}K_{\eps}^N(u_1, x_2, x_3)\ud x_2 \ud x_3\\
&+\frac 14\iint_{I^2\times I^3} \ud y_2 \ud y_3\int_{|x_2|\le |I^2|-d_2}\int_{|x_3|\le |I^3|-d_3}K_{\eps}^N(u_1, x_2, x_3)\ud x_2 \ud x_3.
\end{align*}
By condition \eqref{E:eq134} with $j=1$ we know that each of the integrals on the right is controlled by 
$
|I^2||I^3| |u_1|^{-1} 1_{(\eps_1, N_1)}(u_1).
$ 
This proves \eqref{E:eq141} and therefore the other partial kernel property is also proved.

Clearly $T$ satisfies the translation invariance properties \eqref{E:eq144} and \eqref{E:eq131}.
Therefore, by Section \ref{E:subsec7}, $T$ satisfies the cancellation conditions of Section \ref{E:subsec8}. 

 It remains to check the weak boundedness property \eqref{E:eq74}. 
 For any $I\in \calD_Z$ we have
\begin{align*}
\langle T1_I, 1_I\rangle&= \lim_{\substack{\eps_1,\eps_2,\eps_3\to 0\\ N_1,N_2,N_3\to \infty}} \iint_{I \times I } K_{\eps}^N(x_1-y_1, x_2-y_2, x_3-y_3) \ud y   \ud x.
\end{align*}Again, by a change of variables, we can assume that $I$ is centred at the origin. 
Similarly as before, there holds that
\[
 \iint_{I \times I } K_{\eps}^N(x_1-y_1, x_2-y_2, x_3-y_3) \ud y   \ud x
 =\sum_{i_1=1}^2\sum_{i_2=1}^2\sum_{i_3=1}^2 C(i_1, i_2, i_3),
\]where 
\[
C(i_1, i_2, i_3)=\frac 18 \int_{I} \ud y \int_{I_{i_1}}\int_{I_{i_2}}\int_{I_{i_3}}K_{\eps}^N(x_1, x_2, x_3)\ud x 
\] and each $I_{i_j}=I_{i_j}(y_j)$ is either $\{|x_j|\le d_j(y_j) \}$ or $\{|x_j|\le |I^j|- d_j(y_j) \}$. 
Assumption \eqref{E:eq133} gives that
$
|C(i_1, i_2, i_3)|\lesssim |I|.
$
This proves the weak boundedness property and concludes the proof of  Proposition \ref{E:eqlem6}.
\end{proof}

\subsection{A sufficient condition for the cancellation assumptions}\label{E:subsec7}
Here we demonstrate that CZZ-operators that have certain translation invariance properties
satisfy the cancellation assumptions of Section \ref{E:subsec8}.
Suppose that $T$ is a CZZ-operator satisfying the translation invariance
property 
\begin{equation}\label{E:eq144}
\ave{T1_I, 1_J}
= \ave{T(1_{I+(0,t^2,t^3)}), 1_{J+(0,t^2,t^3)}}
\end{equation}
for all rectangles $I$ and $J$,
where $t^i \in \R$. Then $T$ satisfies the cancellation conditions
\begin{equation}\label{E:eq128}
\ave{T(1_{I^1} \otimes 1), 1_{J^1} \otimes h_{I^{2,3}}} 
=\ave{T(1_{I^1} \otimes h_{I^{2,3}}), 1_{J^1} \otimes 1}=0.
\end{equation}
Similarly, if $T$ has the translation invariance property 
\begin{equation}\label{E:eq131}
\ave{T1_I, 1_J}
= \ave{T(1_{I+(t^1,0,0)}), 1_{J+(t^1,0,0)}}
\end{equation}
then $T$ satisfies the cancellation conditions
$$
\ave{T(1 \otimes 1_{I^{2,3}}), h_{I^1} \otimes 1_{J^{2,3}}}
=\ave{T(h_{I^1} \otimes 1_{I^{2,3}}), 1 \otimes 1_{J^{2,3}}}=0.
$$
In particular, the full translation invariance
$
  \ave{T1_I, 1_J}= \ave{T(1_{I+t}), 1_{J+t}}  
$
for all $t\in\R^3$ implies all cancellation conditions above.

We show that \eqref{E:eq144} implies \eqref{E:eq128}.
The corresponding implication in the first parameter is done similarly.
Consider for example the pairing $\ave{T(1_{I^1} \otimes 1), 1_{J^1} \otimes h_{I^{2,3}}}$.
We begin by first of all explaining how this is defined.
At this point we specify what the Haar function 
$
h_{I^{2,3} }
\in \{ h_{I^2} \otimes h_{I^3}, h_{I^2} \otimes h^0_{I^3}, h^0_{I^2} \otimes h_{I^3}\}
$
is. Suppose for example that $h_{I^{2,3} }=h_{I^2} \otimes h^0_{I^3}=:h^{1,0}_{I^{2,3}}$.
The point is to identify a variable with cancellation.
Then
\begin{equation*}
\begin{split}
\ave{T(1_{I^1} \otimes 1),& 1_{J^1} \otimes h^{1,0}_{I^{2,3}}}
:= \ave{T(1_{I^1} \otimes 1_{ (3I^{2})^c \times \R}), 1_{J^1} \otimes h^{1,0}_{I^{2,3}}} \\
&+\ave{T(1_{I^1} \otimes 1_{ (3I^{2}) \times (3I^3)^c}), 1_{J^1} \otimes h^{1,0}_{I^{2,3}}} 
+\ave{T(1_{I^1} \otimes 1_{ (3I^{2}) \times (3I^3)}), 1_{J^1} \otimes h^{1,0}_{I^{2,3}}}.
\end{split}
\end{equation*}
Here, the last term on the right hand side is already defined.
Suppose that $I^1$ and $J^1$ are disjoint. Then,
the first pairing on the right hand side is defined by
$$
\iint (K(x,y)-K((x_1,c_{I^2},x_3),y))  1_{I^1 \times (3I^{2})^c \times \R}(y)
1_{J^1} \otimes h^{1,0}_{I^{2,3}}(x) \ud y \ud x.
$$
The subtraction of $K((x_1,c_{I^2},x_3),y)$ makes sense because of the zero integral of $h_{I^2}$.
The mixed H\"older and size estimate \eqref{E:eq28} implies that the integral converges absolutely.
The second pairing is defined by
$$
\iint K(x,y)  1_{I^1 \times 3I^{2} \times (3I^3)^c}(y)
1_{J^1} \otimes h^{1,0}_{I^{2,3}}(x) \ud y \ud x
$$
and the size estimate \eqref{E:eq30} implies implies the absolute convergence of the integral.
If $I^1=J^1$ then one could define these pairings correspondingly with the partial kernel $K_{I^1}$, 
see \eqref{E:eq67}. In the general case when $I^1$ and $J^1$ are not disjoint 
(notice that they may have different lengths),
one splits the indicators $1_{I^1}$ and $1_{J^1}$ so that one can use the partial kernels or the full kernel.

Let $L$ be an interval. If $A>0$ let $AL$ denote the interval with the same center as $L$ and with length $A|L|$.
Let $L_l$ and $L_r$ denote the left and right halves of $L$, respectively. Let $t=|I^2|/2$.
The translation invariance \eqref{E:eq144} implies that
\begin{equation}\label{E:eq129}
\begin{split}
\ave{T(1_{I^1\times AI^2 \times AI^3}), 1_{J^1\times I^{2}_l \times I^3}} 
&=\ave{T(1_{I^1\times (AI^2+t) \times AI^3}), 1_{J^1\times (I^{2}_l+t) \times I^3}} \\
&=\ave{T(1_{I^1\times (AI^2+t) \times AI^3}), 1_{J^1\times I^{2}_r \times I^3}}.
\end{split}
\end{equation}
Define $J^2_{A}=AI^2 \cap (AI^2+t)$. Then $1_{AI^2}=1_{AI^2\setminus(AI^2+t)}+1_{J^2_A}$
and $1_{(AI^2+t)}=1_{(AI^2+t) \setminus AI^2} +1_{J^2_A}$. Using this in \eqref{E:eq129}
implies
\begin{equation}\label{E:eq130}
\begin{split}
\ave{T(1_{I^1\times J^2_A \times AI^3}), &1_{J^1}\otimes (1_{I^{2}_l}-1_{I^2_r}) \otimes 1_{I^3}} \\
&=\ave{T(1_{I^1} \otimes 1_{(AI^2+t)\setminus AI^2} \otimes 1_{AI^3}), 1_{J^1\times I^{2}_r \times I^3}} \\
&-\ave{T(1_{I^1}\otimes 1_{(AI^2)\setminus (AI^2+t)} \otimes 1_{AI^3}), 1_{J^1\times I^{2}_l \times I^3}}.
\end{split}
\end{equation}

Using the above definition of the pairing $\ave{T(1_{I^1} \otimes 1), 1_{J^1} \otimes h^{1,0}_{I^{2,3}}}$
one see that the left hand side of \eqref{E:eq130} approaches 
$|I^2|^{1/2}|I^3|^{1/2}\ave{T(1_{I^1} \otimes 1), 1_{J^1} \otimes h^{1,0}_{I^{2,3}}}$ as $A \to \infty$.
On the other hand, if $I^1$ and $J^1$ are disjoint, 
we will check that the size estimate \eqref{E:eq30} implies that the terms on the right hand side of \eqref{E:eq130} tend to $0$ as $A \to \infty$:
Denoting
\begin{equation*}
  E:=(J^1\times I^{2}_r \times I^3)\times(I^1 \times [(AI^2+t)\setminus AI^2] \times \R),
\end{equation*}
we have
\begin{equation*}
\begin{split}
   &\abs{\ave{T(1_{I^1} \otimes 1_{(AI^2+t)\setminus AI^2} \otimes 1_{AI^3}), 1_{J^1\times I^{2}_r \times I^3}}} 
   \lesssim\iint_E\prod_{i=1}^3\frac{1}{\abs{x_i-y_i}} D_\theta(x-y)\ud x\ud y.
\end{split}
\end{equation*}
Integrating first with respect to $y_3 \in \R$ and then with respect to $x_3 \in I^3$ gives
that the last integral is dominated by 
\begin{equation*}
|I^3| \iint \displaylimits _{J^1 \times I^1   }
\frac{1}{\abs{x_1-y_1}} \ud x_{1} \ud y_{1}
\iint \displaylimits _{I^2_r \times [(AI^2+t)\setminus AI^2]} 
\frac{1}{\abs{x_2-y_2}} \ud x_2 \ud y_2
\lesssim |I^3|(|J^1||I^1|)^{1/2}
\frac{|I^2|}{A}.
\end{equation*}
Here the last quantity goes to $0$ as $A \to \infty$.
Similarly, the other term in the right hand side of \eqref{E:eq130} has the desired convergence to $0$
under the assumption that $I^1$ and $J^1$ are disjoint.

If $I^1$ and $J^1$ are not disjoint,
then one splits the indicators into pairwise disjoint and equal parts. The disjoint parts can be dealt with as above.
With the equal parts one uses the partial kernel property \eqref{E:eq67}, the size estimate of the partial kernel \eqref{E:eq62}
and a similar computation as above. In any case, the right hand side of \eqref{E:eq130} goes to $0$ as $A\to \infty$, which proves that
$
\ave{T(1_{I^1} \otimes 1), 1_{J^1} \otimes h^{1,0}_{I^{2,3}}}=0.
$
Similarly, one proves that the other pairing in \eqref{E:eq128} equals zero.

\section{Dyadic multiresolution analysis in the Zygmund setting}
\subsection{Dyadic notation and martingale differences}
Given a dyadic grid $\calD$ in $\R^d$, $I \in \calD$ and $k \in \Z$, $k \ge 0$, we use the following notation:
\begin{enumerate}
\item $\ell(I)$ is the side length of $I$.
\item $I^{(k)} \in \calD$ is the $k$th parent of $I$, i.e., $I \subset I^{(k)}$ and $\ell(I^{(k)}) = 2^k \ell(I)$.
\item $\ch(I)$ is the collection of the children of $I$, i.e., $\ch(I) = \{J \in \calD \colon J^{(1)} = I\}$.
\item $E_I f=\langle f \rangle_I 1_I$ is the averaging operator, where $\langle f \rangle_I = \fint_{I} f = \frac{1}{|I|} \int _I f$.
\item $\Delta_If$ is the martingale difference $\Delta_I f= \sum_{J \in \ch (I)} E_{J} f - E_{I} f$.
\item $\Delta_{I,k} f$ is the martingale difference block
$$
\Delta_{I,k} f=\sum_{\substack{J \in \calD \\ J^{(k)}=I}} \Delta_{J} f.
$$
\end{enumerate}

\subsection{Resolution of functions}
We work on $\R^3$. However, instead of the tri-parameter viewpoint $\R^3 = \R \times \R \times \R$,
we often view this as some kind of modified bi-parameter theory on $\R^3 = \R \times \R^2$.
A suitable dilation on $\R^2$ is required that depends on the choice of scale on $\R$.

We let
$
\calD = \prod_{m=1}^3 \calD^m,
$
where $\calD^1$, $\calD^2$ and $\calD^3$ are dyadic grids on $\R$.
We define the Zygmund rectangles $\calD_Z \subset \calD$
by setting
\begin{equation}\label{zyglat}
\calD_Z = \Big\{I = \prod_{m=1}^3 I^m \in \calD \colon \ell(I^1) \ell(I^2) = \ell(I^3)\Big\}.
\end{equation}
Given $I =  I^1 \times I^2 \times I^3 = I^1 \times I^{2,3} \in \calD_Z$ we define the Zygmund martingale difference operator
$
\Delta_{I,Z} f := \Delta_{I^1} \Delta_{I^{2,3}} f,
$
where we always mean the following:
\begin{enumerate}
\item The martingale difference $\Delta_{I^{2,3} }$ is the \textbf{one-parameter} martingale difference on the rectangle $I^{2,3}$ -- i.e.,
$$
\Delta_{I^{2,3}} f = \sum_{\substack{ S^2 \in \ch(I^2) \\ S^3 \in \ch(I^3)}} E_{S^2 \times S^3} - E_{I^{2,3}} f.
$$
It is \textbf{not} the bi-parameter martingale difference -- for the latter we will use no other notation than the explicit iterative notation $\Delta_{I^2} \Delta_{I^3}$.
In particular, we have
$$
\Delta_{I^{2,3}} = \Delta_{I^2} \Delta_{I^3}  + E_{I^2} \Delta_{I^3} + \Delta_{I^2} E_{I^3} \ne \Delta_{I^2} \Delta_{I^3}.
$$
\item We do not write $\Delta_{I^1}^1$ or  $\Delta_{I^{2,3}}^{2,3}$ -- but this is what we mean. That is, these operators
act on the full product space but only on the given parameters -- for instance, $\Delta_{I^1} f(x_1, x_2, x_3)
= \Delta_{I^1}^1 f(x_1, x_2, x_3) =  (\Delta_{I^1} f(\cdot, x_2, x_3))(x_1)$.
\end{enumerate}
There is also the equally natural symmetric version
$
\wt \Delta_{I, Z}f =  \Delta_{I^2} \Delta_{I^{1,3} },
$
but we will build our analysis around the operators $\Delta_{I,Z}$. 
For a dyadic $\lambda > 0$ define the dilated lattices
\begin{equation}\label{dillat}
\calD_{\lambda}^{2,3} = \{I^{2,3}\in \calD^{2,3} := \calD^2 \times \calD^3 \colon \ell(I^3) = \lambda \ell(I^2)\}.
\end{equation}
We may expand
\begin{equation}\label{zygexp}
\begin{split}
f = \sum_{I^1 \in \calD^1} \Delta_{I^1} f = \sum_{I^1 \in \calD^1} \sum_{I^{2,3} \in \calD^{2,3}_{\ell(I^1)}} 
\Delta_{I^1} \Delta_{I^{2,3}}  f
= \sum_{I \in \calD_Z} \Delta_{I, Z} f.
\end{split}
\end{equation}

\subsection{Properties of Zygmund martingale differences}
\begin{lem}\label{H:lem1}
For $I, J \in \calD_Z$ we have
$$
\Delta_{I,Z} \Delta_{J,Z} f = \left\{ \begin{array}{ll}
\Delta_{I,Z} & \textup{if } I = J, \\
0 & \textup{if } I \ne J.
\end{array} \right.
$$
\end{lem}
\begin{proof}
Suppose $\Delta_{I,Z} \Delta_{J,Z} f \ne 0$. Then automatically $I^1 = J^1$. Aiming for a contradiction, suppose
e.g. $I^2 \subsetneq J^2$. By the Zygmund relationship we have
$$
\ell(J^3) = \ell(J^1)\ell(J^2) = \ell(I^1) \ell(J^2) > \ell(I^1)\ell(I^2) = \ell(I^3).
$$
Thus we must also have $I^3 \subsetneq J^3$ and it is clear that $\Delta_{J^{2,3}}(\Delta_{I^{2,3}} f) = 0$. 
Following this logic we must have that $I^2 = J^2$ and then $I^3 = J^3$. The claim is now clear.
\end{proof}

Notice also that the Zygmund martingale differences satisfy
$$
\int_{\R} \Delta_{I,Z}f \ud x_1 = 0 \qquad \textup{and} \qquad \int_{\R^2} \Delta_{I,Z}f \ud x_2 \ud x_3 = 0.
$$
Moreover, we have
$$
\int (\Delta_{I, Z} f) g = \int f \Delta_{I,Z} g.
$$

\subsection{Zygmund maximal function and extrapolation}
We define
$$
M_{\calD_Z} f(x) = \sup_{R \in \calD_Z} \langle |f| \rangle_R 1_R(x). 
$$
The following is proved in Fefferman--Pipher \cite{FP}.

\begin{prop}\label{prop:FPmax}
We have
$
\|M_{\calD_Z}f\|_{L^p(w)} \lesssim \|f\|_{L^p(w)}
$
for all $p \in (1,\infty)$ and $w \in A_{p, Z}$.
\end{prop}

It will be useful to observe that the same weight class 
$A_{p,Z}$ remains admissible for a slightly bigger maximal operator as well. Namely, let us consider the class
\begin{equation*}
  \calD_{\sub}:=\{K\in\calD: \ell(K^1)\ell(K^2)\geq\ell(K^3)\},
\end{equation*}
which we may refer to as {\em sub-Zygmund rectangles}, as the third interval has at most the size that it should have in the case of proper Zygmund rectangles.
For the corresponding maximal operator
\begin{equation*}
M_{\calD_{ \sub}} f(x) = \sup_{R \in \calD_{\sub}} \langle |f| \rangle_R 1_R(x),
\end{equation*}
we still have the following result.

\begin{prop}\label{prop:FPext}
We have
$
\|M_{\calD_{\sub}}f\|_{L^p(w)} \lesssim \|f\|_{L^p(w)}
$
for all $p \in (1,\infty)$ and $w \in A_{p, Z}$.
\end{prop}

\begin{proof}
We will check the pointwise bound
\begin{equation}\label{eq:MZpw}
   M_{\calD_{\sub}}\leq M_{\calD^1}\circ M_{\calD_Z}.
\end{equation}
The proposition follows, since $M_{\calD_Z}$ is bounded on $L^p(w)$ by Proposition \ref{prop:FPmax}, and $M_{\calD^1}$ is bounded on $L^p(w)$ by the classical Muckenhoupt theorem, Fubini's theorem, and the observation that $x_1 \mapsto w(x_1, x_2, x_3) \in A_p(\R)$ uniformly in $x_2, x_3$ for $w\in A_{p,Z}$.

To prove \eqref{eq:MZpw}, let $x\in K \in \calD_{\sub}$, and let
\begin{equation*}
  \calI^1:=\{I^1\in\calD^1:I^1\subset K^1, \ell(I^1)=\ell(K^3)/\ell(K^2)\},
\end{equation*}
which is a partition of $K^1$, since $\ell(K^3)/\ell(K^2)\leq\ell(K^1)$ by the sub-Zygmund condition. Now each $I^1\times K^{2,3}$ belongs to $\calD_Z$, and hence
\begin{equation}\label{eq:MZpwProof}
\begin{split}
  \ave{f}_K
  &=\fint_{K^1}\Big(\fint_{K^{2,3}}f(y_1,y_{2,3})\ud y_{2,3}\Big)\ud y_1 
  =\fint_{K^1}\sum_{I^1\in\calI^1}1_{I^1}(z_1)  \Big(\fint_{I^1\times K^{2,3}}f(y)\ud y\Big)\ud z_1 \\
  &\leq\fint_{K^1}\sum_{I^1\in\calI^1}1_{I^1}(z_1) M_{\calD_Z}f(z_1,x_{2,3})\ud z_1
    \leq M_{\calD^1}(M_{\calD_Z}f)(x).
\end{split}
\end{equation}
This proves \eqref{eq:MZpw}, and hence the proposition.
\end{proof}

The Zygmund variant of the Rubio de Francia extrapolation result -- also proved in \cite{FP} -- follows from the boundedness of the Zygmund maximal function
via the standard approach \cite{DU}.

\begin{prop}
Let $f, g$ be a given pair of functions and $p_0 \in (1,\infty)$. If we have
$$
\|g\|_{L^{p_0}(w)} \lesssim \|f\|_{L^{p_0}(w)}
$$
for all  $w \in A_{p_0, Z}$, then
$$
\|g\|_{L^{p}(w)} \lesssim \|f\|_{L^{p}(w)}
$$
for all $p \in (1,\infty)$ and $w \in A_{p, Z}$.
\end{prop}

As usual, extrapolation is extremely useful -- maybe the best motivation for weighted estimates.
For example, we readily now have that
\begin{equation}\label{E:eq8}
\Big\| \Big( \sum_i |M_{\calD_{\sub}}f_i|^r \Big)^{\frac{1}{r}} \Big\|_{L^p(w)} \lesssim \Big\| \Big( \sum_i |f_i|^r \Big)^{\frac{1}{r}} \Big\|_{L^p(w)} 
\end{equation}
for all $p,r \in (1,\infty)$ and $w \in A_{p, Z}$.

We will still need another variant of the maximal function in addition to $M_{\calD_Z}$ and $M_{\calD_{\sub}}$.
For $\lambda=2^k$, $k \in \Z$, define
$
\calD_{\lambda} = \{K \in \calD: \lambda \ell(K^1)\ell(K^2) = \ell(K^3)\} 
$
and the related maximal operator
\begin{equation*}
M_{\calD_{\lambda}} f(x) = \sup_{R \in \calD_{\lambda}} \langle |f| \rangle_R 1_R(x).
\end{equation*}
Notice that if $\lambda = 2^k$ and $k \le 0$, then
$M_{\calD_\lambda}$ is pointwise dominated by $M_{\calD_{\sub}}$.

In order to prove a certain weighted estimate for the operators $M_{\calD_{\lambda}}$ we recall the following interpolation result
due to Stein and Weiss, see \cite[Theorem 2.11]{SW}.

\begin{prop}\label{E:prop2}
Suppose that $1 \le p_0,p_1 \le \infty$ and
let $w_0$ and $w_1$ be positive weights. Suppose that $T$ is a sublinear operator that satisfies the estimates
$$
\| T f  \|_{L^{p_i}(w_i)} \le M_i \| f \|_{L^{p_i}(w_i)}, \quad i=1,2.
$$
Let $t \in (0,1)$ and define
$
1/p=(1-t)/p_0+t/p_1
$
and $w=w_0^{p(1-t)/p_0}w_1^{pt/p_1}$. Then $T$ satisfies the estimate
$$
\| T f   \|_{L^{p}(w)} \le M_0^{1-t}M_1^t \| f  \|_{L^{p}(w)}.
$$
\end{prop}

Now we state the weighted estimates for the maximal operators $M_{\calD_{\lambda}}$.

\begin{prop}\label{E:prop3}
Let $\lambda=2^k$, $k=0,1,2,\dots$, $p \in (1, \infty)$ and $w \in A_{p,Z}$. Then, there exist constants
$\eta=\eta(p,w)>0$ and $C=C(p,w)>0$ so that
$$
\| M_{\calD_{\lambda}} f \|_{L^p(w)} \le C \lambda^{1-\eta} \| f \|_{L^p(w)}.
$$
\end{prop}

\begin{proof}
Notice that $M_{\calD_{\lambda}} f \le M_\calD f$, where $M_\calD$ is the tri-parameter strong maximal function
$$
M_\calD f(x) = \sup_{R \in \calD} \langle | f | \rangle_R 1_R(x).
$$
Therefore, we have the unweighted estimate
$
\| M_{\calD_{\lambda}} f \|_{L^p} \le C_0  \| f \|_{L^p},
$
where $C_0=C_0(p)$.

On the other hand, if $K=K^1 \times K^2 \times K^3 \in \calD_{\lambda}$, then
$K \subset (K^1)^{(k)} \times K^2 \times K^3 \in \calD_Z$, and thus
$M_{\calD_{\lambda}}f \le \lambda M_{\calD_Z} f$.
Using Proposition \ref{prop:FPmax} this gives that
\begin{equation}\label{E:eq92}
\| M_{\calD_{\lambda}} f \|_{L^p(v)} \le C_1 \lambda \| f \|_{L^p(v)}, \quad v \in A_{p,Z},
\end{equation}
where $C_1=C_1(p,v)$.

Consider now the weight $w \in A_{p,Z}$. Recall that  
$[w(\cdot, x_2,x_3)]_{A_p}  \le [w]_{A_{p,Z}}$ uniformly on $x_2, x_3$  and 
$[\langle w \rangle_{I^1}]_{A_{p, \ell(I^1)}} \le [ w ]_{A_{p,Z}}$ uniformly on $I^1$.
Therefore, if $I$ is a Zygmund rectangle, by the standard one-parameter reverse H\"older there exists
an $\epsilon=\epsilon([w]_{A_{p,Z}})>0$ so that
$$
\langle w^{1+\epsilon} \rangle_I
=\ave{ \ave{ w^{1+\epsilon} }_{I^1}}_{I^{2,3}}
\lesssim \ave{ \ave{ w }^{1+\epsilon}_{I^1}}_{I^{2,3}}
\lesssim \ave{ \ave{ w }_{I^1}}^{1+\epsilon}_{I^{2,3}}
=\langle w \rangle_I^{1+\epsilon}.
$$
Likewise, since $w^{-1/(p-1)} \in A_{p',Z}$ (with $[w^{-1/(p-1)}]_{A_{p',Z}}=[w]_{A_{p,Z}}^{p'-1}$) there exists an 
$\epsilon=\epsilon([w]_{A_{p,Z}})>0$ so that 
$$
\bave{w^{-\frac{1+\epsilon}{p-1}}}_I \lesssim \bave{w^{\frac{-1}{p-1}}}_I^{1+\epsilon}.
$$
These two estimates imply that there exists an $\epsilon=\epsilon([w]_{A_{p,Z}})>0$ so that
$w^{1+\epsilon} \in A_{p,Z}$. Therefore, \eqref{E:eq92} shows that
\begin{equation}\label{E:eq93}
\| M_{\calD_{\lambda}} f \|_{L^p(w^{1+\epsilon})} \le C_1(p,w) \lambda \| f \|_{L^p(w^{1+\epsilon})}.
\end{equation}

Apply the above interpolation theorem, Proposition \ref{E:prop2}, with
$p_0=p_1=p$, $w_0=1$, $w_1=w^{(1+\epsilon)}$ and $t=1/(1+\epsilon)$. Then
$1^{(1-t)}w^{t(1+\epsilon)}=w$ and we get
$$
\| M_{\calD_{\lambda}} f \|_{L^p(w)} \le C_0^{\epsilon/(1+\epsilon)} (C_1(p,w) \lambda)^{1/(1+\epsilon)}
\| f \|_{L^p(w)}.
$$
We are done.
\end{proof}

\begin{cor}
Let $\lambda=2^k$, $k=0,1,2,\dots$, $p \in (1, \infty)$ and $w \in A_{p,Z}$. Then, there exist constants
$\eta=\eta(p,w)>0$ and $C=C(p,w)>0$ so that
\begin{equation}\label{E:eq121}
\Big\| \Big(\sum_{j=1}^\infty \big[M_{\calD_{\lambda}} f_j\big]^2\Big)^{\frac 12}\Big\|_{L^p(w)} \le C \lambda^{1-\eta} \Big\| \Big(\sum_{j=1}^\infty |f_j|^2\Big)^{\frac 12} \Big\|_{L^p(w)}.
\end{equation}
\end{cor}
\begin{proof}
It is clear that when $p=2$, this is an immediate corollary of Proposition \ref{E:prop3}. Moreover, it is easy to see that in this case one can take  $\epsilon\sim [w]_{A_{2,Z}}^{-1}$ and this means that 
\[
\Big\| \Big(\sum_{j=1}^\infty \big[M_{\calD_{\lambda}} f_j\big]^2\Big)^{\frac 12}\Big\|_{L^2(w)}  \le K([w]_{A_{2,Z}}) \lambda^{1-c[w]_{A_{2,Z}}^{-1}}\Big\| \Big(\sum_{j=1}^\infty |f_j|^2\Big)^{\frac 12} \Big\|_{L^2(w)},
\]where $K$ is some increasing function. Hence $N(x)=K(x)\lambda^{1-cx^{-1}}$ is also an increasing function on $[1,\infty)$. Note that the standard extrapolation \cite[Theorem 3.1]{DU} holds for general Muckenhoupt basis. Thus, by extrapolation we have 
\[
\Big\| \Big(\sum_{j=1}^\infty \big[M_{\calD_{\lambda}} f_j\big]^2\Big)^{\frac 12}\Big\|_{L^p(w)}  \le N(c_p[w]_{A_{p,Z}}^{\alpha(p)})\Big\| \Big(\sum_{j=1}^\infty |f_j|^2\Big)^{\frac 12} \Big\|_{L^p(w)},
\]which gives the desired estimate with $\eta\sim c(c_p[w]_{A_{p,Z}}^{\alpha(p)})^{-1}$. 
\end{proof}

\subsection{Haar functions}\label{E:subsec2}
For an interval $J \subset \R$ we denote by $J_{l}$ and $J_{r}$ the left and right
halves of $J$, respectively. We define
$h_{J}^0 = |J|^{-1/2}1_{J}$ and $h_{J}^1 = |J|^{-1/2}(1_{J_{l}} - 1_{J_{r}})$.
The reader should carefully notice that $h_I^0$ is the non-cancellative Haar function for us and that
in some other papers a different convention is used.

Let now $J = J_1 \times \cdots \times J_d \subset \R^d$ be a rectangle, and define the Haar function $h_J^{\eta}$, $\eta = (\eta_1, \ldots, \eta_d) \in \{0,1\}^d$, by setting
$
h_J^{\eta} = h_{J_1}^{\eta_1} \otimes \cdots \otimes h_{J_d}^{\eta_d}.
$
If $\eta \ne 0$ the Haar function is cancellative: $\int h_J^{\eta} = 0$. We mostly
exploit notation by suppressing the presence of $\eta$, and write $h_J$ for some $h_J^{\eta}$, $\eta \ne 0$. 

The next point is crucial for the correct understanding of this paper.
If we e.g. view $\R^d$ as the bi-parameter product space $\R^d = \R^{d_1} \times \R^{d_2}$, and we have a rectangle
$I^1 \times I^2$, where $I^m \subset \R^{d_m}$, then $h_{I^1 \times I^2}$ still means just the one-parameter Haar function: $h_{I^1 \times I^2} = h_{I^1 \times I^2}^{\eta}$
for some $\eta \in \{0,1\}^d \setminus \{0\}$. It is \textbf{not} the bi-parameter Haar function on the rectangle $I^1 \times I^2$, which would have the different form
$h_{I^1}^{\eta} \otimes h_{I^2}^{\rho}$ for some $\eta  \in \{0,1\}^{d_1} \setminus \{0\}$ and $\rho \in \{0,1\}^{d_2} \setminus \{0\}$. As with the martingale differences,
when we need a product object, we write it explicitly with this tensor / iterative notation.

In particular, given $I = I^1 \times I^2 \times I^3 \in \calD_Z \subset \prod_{m=1}^3 \calD^m$, we have
$$
\Delta_{I, Z}f = \Delta_{I^1} \Delta_{I^{2,3}} f = \langle f, h_{I^1} \otimes h_{I^{2,3}} \rangle h_{I^1} \otimes h_{I^{2,3}} =: 
\langle f, h_{I,Z} \rangle h_{I,Z}.
$$
We can always
write
$
\langle f, h_{I, Z} \rangle = \langle \Delta_{I,Z} f, h_{I, Z}\rangle.
$
This is used often.

\subsection{Zygmund square functions}
We start with an obvious lemma.
\begin{lem}\label{H:lem2}
  Let $\calD = \calD^1 \times \calD^2$ and $\calD_{\lambda} = \{K \in \calD \colon \ell(K^2) = \lambda \ell(K^1)\}$, where $\lambda$ is dyadic.
  We have uniformly on $\lambda$ the square function estimate
  $$
  \Big\| \Big( \sum_{K \in \calD_{\lambda}} |\Delta_K g|^2 \Big)^{1/2} \Big\|_{L^p(w)} \sim \|g\|_{L^p(w)}
  $$
  for all $p \in (1,\infty)$ and $w \in A_{p, \lambda}$, see \eqref{eq:dilatedweights}.
  \end{lem}
  \begin{proof}
  Follows from a simple change of variable from the classical case $\lambda = 1$.
  \end{proof}
We now prove the weighted boundedness of the basic Zygmund-type square function
$$
S_{Z} f := \Big( \sum_{I \in \calD_Z} |\Delta_{I,Z} f|^2 \Big)^{1/2}.
$$
\begin{thm}\label{H:thm1}
We have
$
\|S_Zf\|_{L^p(w)} \sim \|f\|_{L^p(w)}
$
for all $p \in (1,\infty)$ and $w \in A_{p, Z}$.
\end{thm}
\begin{proof}
We prove the upper bound first.
By extrapolation, it is enough to take $p=2$. We write
\begin{align*}
\Big\| \Big( \sum_{I \in \calD_Z} |\Delta_{I, Z} f|^2 \Big)^{1/2}\Big \|_{L^2(w)}^2 &= \sum_{I \in \calD_Z} |\langle f, h_{I,Z} \rangle|^2 \langle w \rangle_I \\
&= \sum_{I^1} \sum_{I^{2,3} \in \calD^{2,3}_{\ell(I^1)}} |\langle \langle f, h_{I^1} \rangle_1, h_{I^{2,3}}\rangle|^2 \langle \langle w \rangle_{I^1} \rangle_{I^{2,3}}.
\end{align*}
Recall the notation from \eqref{dillat}.

As $\langle w \rangle_{I^1} \in A_{2, \ell(I^1)}(\R^{2})$ we know from Lemma \ref{H:lem2} that
\begin{align*}
\sum_{I^{2,3} \in \calD^{2,3}_{\ell(I^1)}} |\langle \langle f, h_{I^1} \rangle_1, h_{I^{2,3}} \rangle|^2 &\langle \langle w \rangle_{I^1} \rangle_{I^{2,3}} 
=\sum_{I^{2,3} \in \calD^{2,3}_{\ell(I^1)}} |\langle \Delta_{I^{2,3}} \langle f, h_{I^1} \rangle_1, h_{I^{2,3}} \rangle|^2 \langle \langle w \rangle_{I^1} \rangle_{I^{2,3}} \\
&\le \sum_{I^{2,3} \in \calD^{2,3}_{\ell(I^1)}} \langle |\Delta_{I^{2,3}} \langle f, h_{I^1} \rangle_1| \rangle_{I^{2,3}}^2 |I^2||I^3| \langle \langle w \rangle_{I^1} \rangle_{I^{2,3}}  \\
&\lesssim \sum_{I^{2,3} \in \calD^{2,3}_{\ell(I^1)}}  \| \Delta_{I^{2,3}} \langle f, h_{I^1} \rangle_1 \|_{L^2(\langle w \rangle_{I^1})}^2 
\lesssim \| \langle f, h_{I^1} \rangle_1 \|_{L^2(\langle w \rangle_{I^1})}^2.
\end{align*}
Notice now that
$$
\int_{\R^{2}} \sum_{I^1} |\langle f, h_{I^1} \rangle_1|^2 \langle w \rangle_{I^1} 
\lesssim \int_{\R^{2}} \sum_{I^1} \| \Delta_{I^1} f \|_{L^2(w)}^2 \lesssim \int_{\R^3} |f|^2 w
$$
since $x_1 \mapsto w(x_1, x_2, x_3) \in A_2(\R)$ uniformly on $x_2, x_3$. 

We move on to proving the lower bound, which can be done with a standard argument thanks to Lemma \ref{H:lem1}.
For $\|g\|_{L^p(\sigma)} \le 1$, where $\sigma = w^{1-p'}$ is the dual weight of $w$, we write
\begin{align*}
\int fg = \sum_{I, J \in \calD_{Z}} \int \Delta_{I,Z} f \Delta_{J,Z} g &= \sum_{I, J \in \calD_{Z}} \int (\Delta_{J,Z} \Delta_{I,Z} f) g \\
&= \sum_{I \in \calD_Z} \int (\Delta_{I,Z} \Delta_{I, Z} f)g = \sum_{I \in \calD_Z} \int \Delta_{I,Z} f \Delta_{I,Z} g.
\end{align*}
We can thus estimate
$$
\Big| \int fg  \Big| \le \int S_Z f \cdot S_Z g \cdot w^{\frac{1}{p}} w^{-\frac{1}{p}} \le \|S_Z f \|_{L^p(w)} \|S_Z g\|_{L^{p'}(\sigma)}.
$$
By the upper square function estimate we have $\|S_Z g\|_{L^{p'}(\sigma)} \lesssim \|g\|_{L^p(\sigma)} \le 1$, and so the
claim follows by taking supremum over such functions $g$.
\end{proof}

We need a more complicated variant.
Let $k=(k^1,k^2,k^3)$, where $k^i \in \{0,1,2,\dots\}$. Define
$$
\calU_k f:= \Big( \sum_{K \in \calD_{2^{-k^1-k^2+k^3}}} |\calU_{K,k} f|^2\Big)^{1/2},
$$
where 
\begin{equation}\label{E:eq6}
\calU_{K,k}f:= 
1_K \sum_{\substack{L \in \calD_Z \colon L^1 \subset K^1, \   \ell(L^1)  \ge 2^{-k^1}\ell(K^1) \\
 2^{-k^1} \ell(K^2) \le 2^{k^2} \ell(L^2) \le 2^{\max(k^2,k^3)} \ell(K^2) }}  \Delta_{L, Z} f.
\end{equation}

\begin{lem}\label{E:lem1}
We have
$
\|\calU_kf\|_{L^p(w)} \lesssim  (|k|+1) \|f\|_{L^p(w)}
$
for all $p \in (1,\infty)$ and $w \in A_{p, Z}$.
\end{lem}

\begin{proof}
As in the definition of $\calU_{K,k}$, suppose that $K \in \calD_{2^{-k^1-k^2+k^3}}$ and $L \in \calD_Z$ are such that
$L \cap K \not = \emptyset$,  $\ell(L^1)  \ge 2^{-k^1}\ell(K^1)$ and
$2^{-k^1} \ell(K^2) \le 2^{k^2} \ell(L^2) \le 2^{\max(k^2,k^3)} \ell(K^2)$. 
Then, we see that $\ell(L^1) \le \ell(K^1)$, $\ell(L^2) \le 2^{\max(0,k^3-k^2)} \ell(K^2)$ and
\begin{equation*}
\begin{split}
\ell(L^3)&=\ell(L^1)\ell(L^2) \le \ell(K^1) 2^{\max(0,k^3-k^2)} \ell(K^2) 
=2^{k^1+k^2}2^{\max(0,k^3-k^2)} 2^{-k^1-k^2}\ell(K^1) \ell(K^2) \\
&= 2^{k^1+k^2}2^{\max(0,k^3-k^2)} 2^{-k^3} \ell(K^3) 
=2^{k^1+\max(k^2-k^3,0)} \ell(K^3).
\end{split}
\end{equation*}
Thus, if we define $\ell=(0, \max(0, k^3-k^2), k^1+\max(k^2-k^3,0))$, we have that $L \subset K^{(\ell)}$.
For some $\lambda=\lambda(k)$ there holds that $K^{(\ell)} \in \calD_\lambda$ (one could check that actually $K^{(\ell)} \in \calD_Z$, but we do not need this).

By extrapolation we may assume that $p=2$. There holds that
\begin{equation*}
\begin{split}
\| \calU_k f \|_{L^2(w)}^2
= \sum_{K \in \calD_{2^{-k^1-k^2+k^3}}} \| \calU_{K,k} f \|_{L^2(w)}^2
&=\sum_{G \in \calD_\lambda} \sum_{\substack{K \in \calD_{2^{-k^1-k^2+k^3}} \\ K^{(\ell)}=G}}
\| \calU_{K,k} f \|_{L^2(w)}^2 \\
&=\sum_{G \in \calD_\lambda} 
\Big\| \sum_{\substack{K \in \calD_{2^{-k^1-k^2+k^3}} \\ K^{(\ell)}=G}} 
\calU_{K,k} f \Big\|_{L^2(w)}^2,
\end{split}
\end{equation*}
where, abbreviating by $*$ the summation conditions
\begin{equation*}
 (*):\qquad   \ell(L^1)  \ge 2^{-k^1}\ell(G^1),\quad 
   2^{-k^1} 2^{-\ell^2} \ell(G^2) \le 2^{k^2} \ell(L^2) \le 2^{\max(k^2,k^3)} 2^{-\ell^2}\ell(G^2),
\end{equation*}
we have from \eqref{E:eq6} that
\begin{equation*}
\begin{split}
\sum_{\substack{K \in \calD_{2^{-k^1-k^2+k^3}} \\ K^{(\ell)}=G}} 
\calU_{K,k} f
&=\sum_{\substack{K \in \calD_{2^{-k^1-k^2+k^3}} \\ K^{(\ell)}=G}}
1_K\sum_{L \in \calD_Z: L^1 \subset G^1}^*\Delta_{L,Z}f
=\sum_{L\in\calD_Z:L\subset G}^*\Delta_{L,Z}f,
\end{split}
\end{equation*}

Using the square function estimate from Theorem \ref{H:thm1} we get that
\begin{equation*}
\begin{split}
\sum_{G \in \calD_\lambda} 
\Big\| \sum_{\substack{K \in \calD_{2^{-k^1-k^2+k^3}} \\ K^{(\ell)}=G}} 
\calU_{K,k} f \Big\|_{L^2(w)}^2
&\sim \sum_{G \in \calD_\lambda}  \sum_{L \in \calD_Z \colon  L \subset G}^*  \| \Delta_{L, Z} f \|_{L^2(w)}^2 \\
&= \sum_{L \in \calD_Z} \| \Delta_{L, Z} f \|_{L^2(w)}^2
\sum_{G \in \calD_\lambda \colon L \subset G}^*1. \\
\end{split}
\end{equation*}
Condition $(*)$ in the last sum can be written in terms of $G$ as
\begin{equation*}
  \ell(L^1)\leq\ell(G^1)\leq 2^{k^1}\ell(L^1),\qquad
  2^{k^2+\ell^2-\max(k^2,k^3)} \ell(L^2) \le \ell(G^2) \le 2^{k^1+k^2+\ell^2} \ell(L^2).
\end{equation*}
Since $G\in\calD_\lambda$, the length $\ell(G^3)$ is uniquely determined by $\ell(G^1)$ and $\ell(G^2)$, and the position of $G$ is uniquely determined by its dimensions and the condition that $L\subset G$. From this we see that there are $k^1+k^2 + \ell^2-k^2-\ell^2+\max(k^2,k^3) +1 = k^1+\max(k^2,k^3)+1$
different possibilities for the length of $G^2$. Thus, the inner sum is dominated by
$(k^1+1)(k^1+\max(k^2,k^3)+1)$, which finishes the estimate.
\end{proof}

\subsection{Resolution of operators}
Given an operator $T$, it is, of course, possible to write -- using \eqref{zygexp} -- the direct Zygmund decomposition
$$
\langle Tf, g \rangle = 
\sum_{I, J \in \calD_Z} \pair{T \Delta_{I, Z}f}{\Delta_{J, Z} g}.
$$
However, we want to have $\ell(I) = \ell(J)$ (understood to mean that $\ell(I^m) = \ell(J^m)$ for $m = 1,2,3$) -- 
this is of fundamental use to us, and we are happy to pay the price that comes with it: we will have less
cancellation.

First, we expand
\begin{equation}\label{E:1par}
\begin{split}
\langle Tf,g \rangle
&= \sum_{I^1,J^1 \in \calD^1} \langle \Delta_{I^1}f, \Delta_{J^1} g \rangle \\
&=\sum_{\substack{I^1,J^1 \in \calD^1 \\ \ell(I^1)=\ell(J^1)}} \big[\langle TE_{I^1}f, \Delta_{J^1} g \rangle
+\langle T\Delta_{I^1}f, E_{J^1} g \rangle+\langle T\Delta_{I^1}f, \Delta_{J^1} g \rangle \big].
\end{split}
\end{equation}
The last equation is seen by dividing to $\ell(I^1) = \ell(J^1)$, $\ell(I^1) > \ell(J^1)$ and $\ell(J^1) > \ell(I^1$), and collapsing
the appropriate sums of martingale differences.
Then, for example, there holds that
\begin{equation*}
\begin{split}
\langle TE_{I^1}f, \Delta_{J^1} g \rangle
&= \sum_{I^{2,3}, J^{2,3} \in \calD^{2,3}_{\ell(I^1)}}
\langle TE_{I^1}\Delta_{I^{2,3}}f, \Delta_{J^1} \Delta_{J^{2,3}} g \rangle \\
&=\sum_{\substack{I^{2,3}, J^{2,3} \in \calD^{2,3}_{\ell(I^1)} \\ \ell(I^2)=\ell(J^2)}}
\big[\langle TE_{I^1}E_{I^{2,3}}f, \Delta_{J^1} \Delta_{J^{2,3}} g \rangle \\
&\qquad+\langle TE_{I^1}\Delta_{I^{2,3}}f, \Delta_{J^1} E_{J^{2,3}} g \rangle
+\langle TE_{I^1}\Delta_{I^{2,3}}f, \Delta_{J^1} \Delta_{J^{2,3}} g \rangle\big].
\end{split}
\end{equation*}
Doing this for all of the three terms in \eqref{E:1par} we arrive at our main Zygmund resolution of operators
\begin{align}
\langle Tf, g\rangle = \sum_{\substack{I, J \in \calD_Z \\ \ell(I) = \ell(J)}} \Big[ &
   \pair{T \Delta_{I, Z}f}{\Delta_{J, Z} g} \label{eq:dec1} \\
   &+\pair{T E_{I^1} \Delta_{I^{2,3}}f}{\Delta_{J, Z} g}
     +\pair{T \Delta_{I, Z}f}{E_{J^1}  \Delta_{J^{2,3}} g} \label{eq:dec2} \\
    &+\pair{T \Delta_{I^1} E_{I^{2,3}}f}{\Delta_{J, Z} g} 
    +\pair{T \Delta_{I, Z}f}{\Delta_{J^1} E_{J^{2,3}} g}  \label{eq:dec3} \\
    &+\pair{T E_{I^1} E_{I^{2,3}}f}{\Delta_{J, Z} g} 
    +\pair{T \Delta_{I, Z}f}{E_{J^1} E_{J^{2,3}} g} \label{eq:dec4} \\
    &+\pair{T E_{I^1} \Delta_{I^{2,3}}f}{\Delta_{J^1} E_{J^{2,3}} g} 
    +\pair{T \Delta_{I^1} E_{I^{2,3}}f}{E_{J^1}  \Delta_{J^{2,3}} g} \label{eq:dec5} \Big].
\end{align}

\section{Zygmund shifts}

We need certain more complicated functions $H$ to accompany our usual Haar functions $h$.
Given $I^m, J^m \in \calD^m$ with $\ell(I^m) = \ell(J^m)$ the functions $H_{I^m,J^m}$ satisfy
\begin{enumerate}
\item $H_{I^m,J^m}$ is supported on $I^m \cup J^m$ and constant on the children of $I^m$ and $J^m$, i.e., we have
$$
H_{I^m,J^m} = \sum_{L^m \in \ch(I^m) \cup \ch(J^m)} b_{L^m} 1_{L^m},\,\, b_{L^m} \in \R,
$$
\item $|H_{I^m,J^m}| \le |I^m|^{-1/2}$ and
\item $\int H_{I^m,J^m} = 0$.
\end{enumerate}
In practice, we will have
$
H_{I^m,J^m} \in \{h_{J^m}^0 - h_{I^m}^0, h_{I^m}^0 - h_{J^m}^0, h_{I^m}, h_{J^m}\},
$
but the abstract definition contains enough information to bound our upcoming operators. For example, instead
of $h_{I^1}$ or $h_{J^1}$ we may sometimes now have $H_{I^1, J^1}$.

We will also need to replace e.g. $h_{I^{2,3}}$ or $h_{J^{2,3}}$ with $H_{I^{2,3}, J^{2,3}}$ if
$\ell(I^2) = \ell(J^2)$ and $\ell(I^3) = \ell(J^3)$.
These functions are the obvious abstraction of the concrete choices
$$
H_{I^{2,3}, J^{2,3}} \in \{ h_{I^{2,3}}^0 - h_{J^{2,3}}^0,  h_{J^{2,3}}^0 - h_{I^{2,3}}^0, h_{I^{2,3}}, h_{J^{2,3}}\}.
$$

Given a rectangle $I = I^1 \times I^2 \times I^3 \in \calD$ and $k = (k^1, k^2, k^3)$, $k^i \in \{0,1,2,\ldots\}$, we define the parent 
$I^{(k)}=(I^1)^{(k^1)} \times (I^2)^{(k^2)} \times (I^3)^{(k^3)} \in \calD$.

\begin{defn}\label{E:defn1}
Let $k = (k^1, k^2,k^3)$, $k^i \in \{0,1,2,\ldots\}$, be fixed. The Zygmund shift $Q = Q_k$ has four different forms. We either have
$$
\langle Q f, g \rangle = 
\sum_{K \in \calD_{2^{-k^1-k^2+k^3}}} \sum_{ \substack{ I, J \in \calD_{Z} \\ I^{(k)} = J^{(k)} = K}} a_{IJK}
 \langle f, H_{I^1, J^1} \otimes H_{I^{2,3}, J^{2,3}} \rangle \langle g, h_{J, Z} \rangle
$$
or that $\langle Qf, g\rangle$ has the symmetric form, where we have $\langle f, h_{I,Z} \rangle \langle g,  H_{I^1, J^1} \otimes H_{I^{2,3}, J^{2,3}} \rangle$, or
$$
\langle Q f, g \rangle = 
\sum_{K \in \calD_{2^{-k^1-k^2+k^3}}} \sum_{ \substack{ I, J \in \calD_{Z} \\ I^{(k)} = J^{(k)} = K}} a_{IJK}
 \langle f, H_{I^1, J^1} \otimes h_{I^{2,3}} \rangle \langle g, h_{J^1} \otimes H_{I^{2,3}, J^{2,3}}\rangle,
$$
or that  $\langle Qf, g\rangle$ has the symmetric form, where we have
$\langle f, h_{I^1} \otimes H_{I^{2,3}, J^{2,3}} \rangle \langle g, H_{I^1, J^1} \otimes h_{J^{2,3}}\rangle$.
The coefficients satisfy the normalization
$$
|a_{IJK}| \le \frac{|I|}{|K|}.
$$
\end{defn}

Next we prove the weighted bound for the Zygmund shifts.

\begin{thm}\label{H:thm3}
Let $Q = Q_k$ be a Zygmund shift of complexity $k = (k^1, k^2, k^3)$. 
Let $p \in (1, \infty)$ and $w \in A_{p, Z}$. Then, there exist $C=C(p,w)$ and $\eta=\eta(p,w)>0$ so that
\begin{equation}\label{E:eq124}
\|Qf\|_{L^p(w)} \le  
C2^{\max(-k^2+k^3,0)(1-\eta)}(|k|+1)^2 \|f\|_{L^p(w)}.
\end{equation}
Also, if $w \in A_p(\R^3)$ is a tri-parameter weight, then
\begin{equation}\label{E:eq122}
\|Qf\|_{L^p(w)} 
\le  C(p,w)(|k|+1)^2 \|f\|_{L^p(w)}.
\end{equation}
\end{thm}
\begin{proof}
We consider the three different kind of pairings
\begin{enumerate}
\item $\langle f, H_{I^1, J^1} \otimes H_{I^{2,3}, J^{2,3}} \rangle$,
\item $\langle f, H_{I^1, J^1} \otimes h_{I^{2,3}} \rangle$ and
\item $\langle f, h_{I^1} \otimes H_{I^{2,3}, J^{2,3}} \rangle$
\end{enumerate}
that appear here in addition to the easy pairing $\langle f, h_{I,Z}\rangle$ for which
$$\langle f, h_{I,Z} \rangle = \Big \langle \sum_{\substack{L \in \calD_Z \\ L^{(k)} = K}} \Delta_{L,Z} f, h_{I,Z}\Big\rangle.$$ 
With what can we replace $f$ with in these other three cases?
There are many possibilities, but the main thing is the following: we have to maintain the Zygmund structure. This is the difficult part in the proof.

We start with the case (1). Write
$$
\langle f, H_{I^1, J^1} \otimes H_{I^{2,3}, J^{2,3}} \rangle
= \sum_{L \in \calD_Z} \langle \Delta_{L,Z}f, H_{I^1, J^1} \otimes H_{I^{2,3}, J^{2,3}} \rangle.
$$
Suppose that $L \in \calD_Z$ is such that the above pairing is non-zero. 
First, we notice that
by the properties of $H_{I^1,J^1}$ we must have that $L^1 \subset K^1$ and $\ell(L^1) \ge 2^{-k^1} \ell(K^1)$.
Then, we notice that by the properties of $H_{I^{2,3}, J^{2,3}}$ we must have that $\ell(L^m) \ge \ell(I^m)$ for $m=2$ or $m=3$ and 
$L^m \subset K^m$ for $m=2$ or $m=3$.
We see that if $\ell(L^2) < 2^{-k^1} \ell(I^2)$, then
$$
\ell(L^3)=\ell(L^1)\ell(L^2) 
< 2^{k^1}\ell(I^1)2^{-k^1}\ell(I^2)
= \ell(I^3),
$$
and the condition $\ell(L^m) \ge \ell(I^m)$ for $m=2$ or $m=3$ can not be satisfied.
If $\ell(L^2) > 2^{\max(k^2,k^3)}\ell(I^2)$, then $\ell(L^2) > \ell(K^2)$ and
$$
\ell(L^3) = \ell(L^1)\ell(L^2) > \ell(I^1) 2^{\max(k^2,k^3)} \ell(I^2) \ge 2^{k^3} \ell(I^3)=\ell(K^3),
$$
and the condition $L^m \subset K^m$ for $m=2$ or $m=3$ can not be satisfied.
These observations show that
\begin{equation*}
\begin{split}
\langle f, H_{I^1, J^1} \otimes H_{I^{2,3}, J^{2,3}} \rangle
&=  \sum_{\substack{L \in \calD_Z :\ L^1 \subset K^1, \   \ell(L^1)  \ge 2^{-k^1}\ell(K^1) \\
 2^{-k^1} \ell(I^2)   \le  \ell(L^2) \le 2^{\max(k^2,k^3)} \ell(I^2) }} \langle \Delta_{L,Z}f, H_{I^1, J^1} \otimes H_{I^{2,3}, J^{2,3}} \rangle \\
& =\langle \calU_{K,k}f, H_{I^1, J^1} \otimes H_{I^{2,3}, J^{2,3}} \rangle,
\end{split}
\end{equation*}
where $\calU_{K,k}$ was defined in \eqref{E:eq6} (note that $\ell(K^2)=2^{k^2}\ell(I^2)$).

Consider the pairings in (2) and (3). Since $h_{I^{2,3}}$ can be seen as  
$H_{I^{2,3}, J^{2,3}}$ and $h_{I^1}$ can be seen as $H_{I^1, J^1}$, we can use the above to
write
$$
\langle f, H_{I^1, J^1} \otimes h_{I^{2,3}} \rangle
=\langle \calU_{K,k} f, H_{I^1, J^1} \otimes h_{I^{2,3}} \rangle
$$
and
$$
\langle f, h_{I^1} \otimes H_{I^{2,3}, J^{2,3}} \rangle
=\langle \calU_{K,k} f, h_{I^1} \otimes H_{I^{2,3}, J^{2,3}} \rangle.
$$
We point out that we could use the fact that there are normal Haar functions present in (2) and (3)
and use some smaller blocks than $\calU_{K,k}$ which
would lead to a better dependence on the complexity in the estimates. 

Having the decompositions at hand we start estimating the shifts. The different forms of shifts are handled similarly.
Suppose for example that $Q$ has the form
$$
\langle Qf,g \rangle
=\sum_{K \in \calD_{\lambda}} \sum_{ \substack{ I, J \in \calD_{Z} \\ I^{(k)} = J^{(k)} = K}} a_{IJK}
 \langle f, H_{I^1, J^1} \otimes h_{I^{2,3}} \rangle \langle g, h_{J^1} \otimes H_{I^{2,3}, J^{2,3}}\rangle,
$$
where $\lambda =2^{-k^1-k^2+k^3}$.
Replace $f$ with $\calU_{K,k}f$ and $g$ with $\calU_{K,k}g$ inside the pairings. 
This gives that
\begin{equation}\label{E:eq43}
\begin{split}
&|\langle Qf,g \rangle| \\
&\le \sum_{K \in \calD_{\lambda}} \sum_{ \substack{ I, J \in \calD_{Z} \\ I^{(k)} = J^{(k)} = K}} \frac{|I|}{|K|}
\langle |\calU_{K,k}f|, |H_{I^1, J^1} \otimes h_{I^{2,3}}| \rangle 
\langle |\calU_{K,k}f|, |h_{J^1} \otimes H_{I^{2,3}, J^{2,3}}|\rangle.
\end{split}
\end{equation}

Recall that $|H_{I^1, J^1}| \le h^0_{I^1}+h^0_{J^1}$ and $|H_{I^{2,3}, I^{2,3}}|
\le h^0_{I^{2,3}} + h^0_{J^{2,3}}$. This splits the above sum into four sums which we consider separately. 
The part related to $h_{I^1}$ and $h_{J^{2,3}}$ equals
\begin{equation}\label{E:eq94}
\begin{split}
\sum_{K \in \calD_{\lambda}} & \sum_{ \substack{ I, J \in \calD_{Z} \\ I^{(k)} = J^{(k)} = K}}  \frac{|I|}{|K|}
\langle |\calU_{K,k}f|, h^0_{I} \rangle 
\langle |\calU_{K,k}g|, h^0_{J} \rangle \\
& = \sum_{K \in \calD_{\lambda}} \frac{1}{|K|}
\langle |\calU_{K,k}f|, 1_K \rangle 
\langle |\calU_{K,k}g|, 1_K \rangle 
 \le \sum_{K \in \calD_{\lambda}}
\langle M_{\calD_\lambda}  \calU_{K,k}f, \calU_{K,k}g \rangle.
\end{split}
\end{equation} 
Using the estimate \eqref{E:eq121} for $M_{\calD_\lambda}$ and the square function estimate from Lemma \ref{E:lem1} we have that the last
term in \eqref{E:eq94} is dominated by
\begin{equation}\label{E:eq98}
\begin{split}
\Big \|&\Big(\sum_{K \in \calD_{\lambda}}
(M_{\calD_\lambda}  \calU_{K,k}f)^2 \Big)^{1/2} \Big \|_{L^p(w)}
\Big \|\Big(\sum_{K \in \calD_{\lambda}}
| \calU_{K,k}g|^2 \Big)^{1/2} \Big \|_{L^{p'}(w^{1-p'})} \\
& \lesssim 2^{\max(-k^1-k^2+k^3,0)(1-\eta)} (|k|+1)^2 \| f \|_{L^p(w)} \| g \|_{L^{p'}(w^{1-p'})},
\end{split}
\end{equation}
and thus this part satisfies the desired estimate.

We consider the next part of the shift related to $h_{J^1}$ and $h_{J^{2,3}}$.
This part is
\begin{equation*}
\begin{split}
\sum_{K \in \calD_{\lambda}} &\sum_{ \substack{ I, J \in \calD_{Z} \\ I^{(k)} = J^{(k)} = K}} \frac{|I|}{|K|}
\langle |\calU_{K,k}f|, h^0_{J^1\times I^{2,3}} \rangle 
\langle |\calU_{K,k}g|, h^0_{J} \rangle \\
& =\sum_{K \in \calD_{\lambda}}
\sum_{ (J^1)^{(k^1)}=K^1} \frac{1}{|K|}
\langle |\calU_{K,k}f|, 1_{J^1\times K^{2,3}} \rangle 
\langle |\calU_{K,k}g|, 1_{J^1 \times K^{2,3}} \rangle \sum_{(I^1)^{(k^1)}=K^1} 1.
\end{split}
\end{equation*}
Since
$$
\sum_{(I^1)^{(k^1)}=K^1} 1=\sum_{(I^1)^{(k^1)}=K^1} \frac{|I^1|}{|J^1|}
=\frac{|K^1|}{|J^1|}
$$
we get
\begin{equation}\label{E:eq95}
\begin{split}
\sum_{K \in \calD_{\lambda}}
\sum_{ (J^1)^{(k^1)}=K^1} \frac{1}{|J^1 \times K^{2,3}|}
&\langle |\calU_{K,k}f|, 1_{J^1\times K^{2,3}} \rangle 
\langle |\calU_{K,k}g|, 1_{J^1 \times K^{2,3}} \rangle \\
& \le \sum_{K \in \calD_{\lambda}}  
\langle M_{\calD_{2^{-k^2+k^3}}}  \calU_{K,k}f, \calU_{K,k}g \rangle.
\end{split}
\end{equation}
As in \eqref{E:eq98} this leads to the right estimate.

The third part of the shift related to $h_{I^1}$ and $h_{I^{2,3}}$ equals
\begin{equation*}
\begin{split}
&\sum_{K \in \calD_{\lambda}} \sum_{ \substack{ I, J \in \calD_{Z} \\ I^{(k)} = J^{(k)} = K}} \frac{|I|}{|K|}
\langle |\calU_{K,k}f|, h^0_{I} \rangle 
\langle |\calU_{K,k}g|, h^0_{J^1 \times I^{2,3}} \rangle \\
&=\sum_{K \in \calD_{\lambda}} \sum_{(I^{2,3})^{(k^{2,3})}=K^{2,3}} \frac{1}{|K|}
\langle |\calU_{K,k}f|, 1_{K^1 \times I^{2,3}} \rangle 
\langle |\calU_{K,k}g|, 1_{K^1 \times I^{2,3}} \rangle \sum_{(J^{2,3})^{(k^{2,3})}=K^{2,3}} 1.
\end{split}
\end{equation*}
The last sum gives $|K^{2,3}|/|I^{2,3}|$ so we get
\begin{equation}\label{E:eq96}
\begin{split}
\sum_{K \in \calD_{\lambda}} \sum_{(I^{2,3})^{(k^{2,3})}=K^{2,3}} & \frac{1}{|K^1 \times I^{2,3}|}
\langle |\calU_{K,k}f|, 1_{K^1 \times I^{2,3}} \rangle 
\langle |\calU_{K,k}g|, 1_{K^1 \times I^{2,3}} \rangle \\
& \le \sum_{K \in \calD_{\lambda}} \langle M_{\calD_{\sub}}  \calU_{K,k}f, \calU_{K,k}g \rangle.
\end{split}
\end{equation}
We can again conclude the estimate as before, this time applying the weighted estimate \eqref{E:eq8} of $M_{\calD_{\sub}}$. 

Finally, the last part of the shift related to $h_{J^1}$ and $h_{I^{2,3}}$ is
\begin{equation*}
\begin{split}
\sum_{K \in \calD_{\lambda}} &\sum_{ \substack{ I, J \in \calD_{Z} \\ I^{(k)} = J^{(k)} = K}} \frac{|I|}{|K|}
\langle |\calU_{K,k}f|, h^0_{J^1 \times I^{2,3}} \rangle 
\langle |\calU_{K,k}g|, h^0_{J^1 \times I^{2,3}} \rangle \\
&=\sum_{K \in \calD_{\lambda}} \sum_{ \substack{ J^1\times I^{2,3} \in \calD_{Z} \\ (J^1 \times I^{2,3})^{(k)} = K}} 
\frac{1}{|J^1 \times I^{2,3}|}
\langle |\calU_{K,k}f|, 1_{J^1 \times I^{2,3}} \rangle 
\langle |\calU_{K,k}g|, 1_{J^1 \times I^{2,3}} \rangle \\
& \le \sum_{K \in \calD_{\lambda}} \langle M_{\calD_{Z}}  \calU_{K,k}f, \calU_{K,k}g \rangle,
\end{split}
\end{equation*}
and we can conclude as in \eqref{E:eq96}. This ends the weighted estimate if $w \in A_{p,Z}$.

Suppose then that $w \in A_p(\R^3) \subset A_{p,Z}$ is a tri-parameter weight.
We can prove the estimate precisely as above, except that this time we do not get any complexity
dependence from the maximal functions. Indeed, all the maximal functions are pointwise dominated by
the tri-parameter maximal function
$$
Mf:= \sup_{R \in \calD} \langle | f | \rangle_R 1_R,
$$
where the supremum is now over all rectangles  $R \in \calD$. 
The maximal function $M$ satisfies the required weighted estimates for the proof,
since $w \in A_p(\R^3)$. This concludes the proof.
\end{proof}

\section{Dyadic representation theorem}
Let $\calD_0$ be the standard dyadic grid in $\R$. For $\omega \in \{0,1\}^{\Z}$, $\omega = (\omega_i)_{i \in \Z}$, 
we define the shifted lattice
$$
\calD(\omega) := \Big\{L + \omega := L + \sum_{i\colon 2^{-i} < \ell(L)} 2^{-i}\omega_i \colon L \in \calD_0\Big\}.
$$
Let $\bbP_{\omega}$ be the product probability measure on $\{0,1\}^{\Z}$. 
We recall the following notion of a $k$-good, $k \ge 2$, interval from \cite{GH}.
We say that $G \in \calD(\omega, k)$, $k \ge 2$,
if $G \in \calD(\omega)$ and
\begin{equation}\label{eq:DefkGood}
d(G, \partial G^{(k)}) \ge \frac{\ell(G^{(k)})}{4} = 2^{k-2} \ell(G).
\end{equation}
Notice that for all $L \in \calD_0$ and $k \ge 2$ we have
\begin{equation}\label{eq:gprob}
  \bbP_{\omega}( \{ \omega\colon L + \omega \in \calD(\omega, k) \})  = \frac{1}{2}.
\end{equation}
The key implication of $G \in \calD(\omega, k)$ is that for $n \in \Z$ with $|n| \le 2^{k - 2}$ we have
\begin{equation}\label{eq:kparent}
(G \dotplus n)^{(k)} = G^{(k)}, \qquad G \dotplus n = G + n\ell(G).
\end{equation}
Notice that if $L \in \calD_0$,
the position of $L + \omega$ depends only on the values of
$\omega_i$ for those $i$ such that $2^{-i} < \ell(L)$, and the $k$-goodness $L + \omega \in \calD(\omega, k)$
of $L + \omega$ for some $k \ge 2$ depends on $\omega_i$ for those $i$ such that
$\ell(L) \le 2^{-i} < \ell(L^{(k)})$. Thus, the position is independent of $k$-goodness for all $k \ge 2$.

Moving to our actual setting of $\R^3 = \R \times \R^2$ we define for
$$
\sigma = (\sigma^1, \sigma^2, \sigma^3) \in \{0,1\}^{\Z} \times \{0,1\}^{\Z} \times \{0,1\}^{\Z}, \qquad
\sigma^m = (\sigma^m_i)_{i \in \Z},
$$
that
$
\calD(\sigma) := \calD(\sigma^1) \times \calD(\sigma^2) \times \calD(\sigma^3).
$
Let $\bbP_{\sigma} := \bbP_{\sigma^1} \times \bbP_{\sigma^2} \times \bbP_{\sigma^3}.$
For $k = (k^1, k^2, k^3)$, $k^1, k^2, k^3 \ge 2$, we define
$
\calD(\sigma, k) = \calD(\sigma^1, k^1) \times \calD(\sigma^2, k^2) \times \calD(\sigma^3, k^3).
$
We also e.g. write
$$
\calD(\sigma, (k^1, 0, k^3)) = \calD(\sigma^1, k^1) \times \calD(\sigma^2) \times \calD(\sigma^3, k^3),
$$
that is, a $0$ will designate that we do not have goodness in that parameter.

As for most of the argument $\sigma$ is fixed, it makes sense to suppress it from the notation and
abbreviate
$\calD^m = \calD(\sigma^m)$, $m = 1,2,3$.
Then also
$
\calD = \calD(\sigma) = \prod_{m=1}^3 \calD^m.
$ 
The Zygmund rectangles $\calD_Z \subset \calD$ are, of course, defined as before.
For $k$-good Zygmund rectangles we use the obvious notation $\calD_Z(k) = \calD_Z(\sigma, k)$, which consists
of the Zygmund rectangles build using $\calD(k) = \calD(\sigma, k)$.

If $\lambda = 2^k$, $k \in \Z$, we recall that
$$
\calD_{\lambda } = \{I = I^1 \times I^2 \times I^3 \in \calD \colon \lambda\ell(I^1) \ell(I^2) = \ell(I^3)\}.
$$
so that $\calD_{1}=\calD_{Z}$.

\begin{thm}\label{E:thm1}
Let $T$ be a $(\theta, \alpha_1, \alpha_{23})$-CZZ operator. Then
$$
\ave{Tf,g}
= C\E_{\sigma} \sum_{k \in \N^3}\sum_{u=1}^{U_0} (|k|+1)\varphi(k) \langle Q_{k,u,\sigma} f, g \rangle,
$$
where
\begin{equation}\label{E:eq123}
\varphi(k):= 
 2^{-k^1\alpha_1-k^2\min\{\alpha_{23},\theta\}-\max\{k^3-k^1-k^2, 0\}\theta}.
\end{equation}
Here $C$ is a constant depending on the various constants related to the operator $T$,
$U_0 \le 100$ and $Q_{k,u, \sigma}$ is a Zygmund shift of complexity $k=(k^1,k^2,k^3)$
related to $\calD_{\sigma}$.
\end{thm}

\begin{cor}\label{E:cor1}
Let $T$ be a $(\theta, \alpha_1, \alpha_{23})$-CZZ operator or a
$(\log, \alpha_1,\alpha_{23})$-CZZ operator, 
where the parameters $\theta, \alpha_1,\alpha_{23} \in (0,1]$ are arbitrary. 
Let $p \in (1, \infty)$ and let $w \in A_{p}(\R^3)$ be a tri-parameter weight.
Then $T$ satisfies the estimate
$$
\| Tf \|_{L^p(w)} \lesssim \|f \|_{L^p(w)}.
$$

Let $w \in A_{p,Z}$. If $T$ is a
$(1, 1, \alpha_{23})$-CZZ operator or a
$(\log, 1,\alpha_{23})$-CZZ operator, 
where $\alpha_{23} \in (0,1]$, then
$$
\| Tf \|_{L^p(w)} \lesssim \|f \|_{L^p(w)}.
$$
\end{cor}

\begin{proof}
Notice that if $t \ge 1$ and $\gamma \in (0,1]$, then
$
\log t \le \gamma^{-1} t^\gamma.
$
Therefore, a $(\log, \alpha_1, \alpha_{23})$-CZZ operator is also
a corresponding operator with respect to $(\beta, \alpha_1, \alpha_{23})$ with any $\beta \in (0,1)$.
However, the implicit constants in the kernel estimates grow as $\beta \to 1$.
Due to this it suffices to prove the weighted estimate with tri-parameter weights only
for $(\theta, \alpha_1, \alpha_{23})$-CZZ operators. Suppose $T$
is such an operator, $p \in (1, \infty)$ and $w \in A_p(\R^3)$.

Theorem \ref{E:thm1} gives
$$
|\ave{Tf,g}|
\lesssim \E_{\sigma} \sum_{k \in \N^3}\sum_{u=1}^{U_0} 
(|k|+1)\varphi(k) | \langle Q_{k,u,\sigma} f, g \rangle|.
$$
The weighted estimate \eqref{E:eq122} of the shifts shows that
$$
| \langle Q_{k,u,\sigma} f, g \rangle|
\lesssim (|k|+1)^{2} \| f \|_{L^p(w)} \| g \|_{L^{p'}(w^{1-p'})}.
$$
Therefore, we have that
\begin{equation*}
\begin{split}
|\ave{Tf,g}|
&\lesssim \E_{\sigma} \sum_{k \in \N^3}\sum_{u=1}^{U_0} 
(|k|+1)^{3}\varphi(k)  \| f \|_{L^p(w)} \| g \|_{L^{p'}(w^{1-p'})} \\
&\lesssim  \sum_{k \in \N^3} 
(|k|+1)^3\varphi(k) \| f \|_{L^p(w)} \| g \|_{L^{p'}(w^{1-p'})}
\end{split}
\end{equation*}
and it remains to sum the series.

Recalling the definition of $\varphi(k)$ in \eqref{E:eq123} and using for example that
\begin{equation}\label{E:eq125}
(|k|+1)^3
\lesssim ( k^1+k^2+1)^3( k^3+1)^3
\end{equation}
there holds that
\begin{equation*}
\begin{split}
\sum_{k \in \N^3} 
(|k|+1)^3 \varphi(k) 
\lesssim  \sum_{k^1, k^2=0}^\infty &( k^1+k^2+1)^32^{-k^1\alpha_1-k^2\min\{\alpha_{23},\theta\}} \\
&\Big[\sum_{k^3=0}^{k^1+k^2} ( k^3+1)^3 + 
\sum_{k^3=k^1+k^2+1}^\infty( k^3+1)^3 2^{-(k^3-k^1-k^2)\theta}\Big].
\end{split}
\end{equation*}
Here $[ \cdot ] \lesssim (k^1+k^2+1)^4+(k^1+k^2+1)^3 \lesssim (k^1+k^2+1)^4$.
Using this
shows that the series converges. This proves the weighted estimate with tri-parameter weights.

Suppose then that $w \in A_{p,Z}$. Let $T$ be a $(\log, 1, \alpha_{23})$-CZZ operator.
It suffices to prove the estimate for $T$, since a $(1,1,\alpha_{23})$-CZZ operator is clearly also a
$(\log, 1, \alpha_{23})$-CZZ operator.

Let $\eta=\eta(p,w)>0$ be the small number provided by Theorem \ref{H:thm3} so
that shifts satisfy the weighted estimate \eqref{E:eq124}. Let $\theta \in (0,1)$ be such that
$\theta > 1-\eta$. Using the observation we made in the beginning of this proof we can see $T$ also as a
$(\theta, 1, \alpha_{23})$-CZZ operator. Using the representation theorem, Theorem  \ref{E:thm1}, we have
that
$$
|\ave{Tf,g}|
\lesssim \E_{\sigma} \sum_{k \in \N^3}\sum_{u=1}^{U_0} 
(|k|+1)\varphi(k) | \langle Q_{k,u,\sigma} f, g \rangle|,
$$
where
$
\varphi(k):= 
 2^{-k^1-k^2\min\{\alpha_{23},\theta\}-\max\{k^3-k^1-k^2, 0\}\theta}.
$
Notice that here the implicit constant depends on the weight $w$ due to the fact that we viewed $T$ as
related to the parameters $(\theta, 1, \alpha_{23})$.
The weighted estimate \eqref{E:eq124} of the shifts gives that
$$
|\ave{Tf,g}|
\lesssim \sum_{k \in \N^3} 
(|k|+1)^3\varphi(k) 2^{\max(-k^2+k^3,0)(1-\eta)} \| f \|_{L^p(w)} \| g \|_{L^{p'}(w^{1-p'})}.
$$

To end the proof we show that the series converges.
Writing out the definition of $\varphi(k)$ and using \eqref{E:eq125} we get
\begin{equation*}
\begin{split}
\sum_{k^1,k^2=0}^\infty & (k^1+k^2+1)^3  2^{-k^1-k^2\min\{\alpha_{23},\theta\}}
\Big[
\sum_{k^3=0}^{k^2} (k^3+1)^3 \\
&+\sum_{k^3=k^2+1}^{k^1+k^2}(k^3+1)^3 2^{(k^3-k^2)(1-\eta)}
+\sum_{k^3=k^1+k^2+1}^{\infty}(k^3+1)^3 2^{-(k^3-k^1-k^2)\theta} 2^{(k^3-k^2)(1-\eta)}.
\Big]
\end{split}
\end{equation*}
Here inside the brackets the first sum is dominated by $(k^2+1)^4$
and the second by $(k^1+k^2+1)^32^{k^1(1-\eta)}$. The third one is
\begin{equation*}
\begin{split}
\sum_{k^3=1}^{\infty}(k^1+k^2+k^3+1)^3 2^{-k^3\theta} 2^{(k^1+k^3)(1-\eta)} 
&=2^{k^1(1-\eta)}\sum_{k^3=1}^{\infty}(k^1+k^2+k^3+1)^3 2^{-k^3(\theta -1+\eta)} \\
& \lesssim 2^{k^1(1-\eta)}(k^1+k^2+1)^3,
\end{split}
\end{equation*}
where we used the fact that $\theta -1+\eta>0$.
A substitution of this into the previous sum gives
\begin{equation*}
\sum_{k^1,k^2=0}^\infty  (k^1+k^2+1)^7  2^{-k^1\eta-k^2\min\{\alpha_{23},\theta\}}
< \infty.
\end{equation*}
\end{proof}

\begin{proof}[Proof of Theorem \ref{E:thm1}]
We use the decomposition \eqref{eq:dec1} through \eqref{eq:dec5} of $\langle Tf, g\rangle$. We perform it first
on $\calD_{Z} = \calD_{\sigma, Z}$ and then take an expectation $\E_{\sigma}$ of
this identity. Next, we start manipulating this basic decomposition. With a fixed $\sigma$ we denote the $j$th term of the decomposition
 \eqref{eq:dec1} through \eqref{eq:dec5} with $\Sigma_{j, \sigma} = \Sigma_j$, $j=1, \ldots, 9$. That is, we e.g. have
 $$
 \Sigma_1=  \sum_{\substack{I, J \in \calD_{Z} \\ \ell(I) = \ell(J)}} \pair{T \Delta_{I, Z}f}{\Delta_{J, Z} g}.
 $$

The term $\Sigma_{1}$ is a pure resolution term, while the remaining eight all involve some averaging operators as well. 
\begin{enumerate}
\item On each of the lines \eqref{eq:dec1} through \eqref{eq:dec5}, the two terms appearing on the same line are dual to each other.
\item The averaging $\E_{\sigma}$ will only be exploited in the middle of the proof in certain key places. It will be used to induce
some goodness to the appearing cubes so that they can be realised as shifts.
\end{enumerate}

We start with the term
\begin{align*}
\Sigma_{8} &= \sum_{\substack{I, J \in \calD_{Z} \\ \ell(I) = \ell(J)}} \pair{T E_{I^1} \Delta_{I^{2,3}}f}{\Delta_{J^1} E_{J^{2,3}} g} \\
&= \sum_{\substack{I, J \in \calD_{Z} \\ \ell(I) = \ell(J)}}  \langle T(h_{I^1}^0 \otimes h_{I^{2,3}}), h_{J^1} \otimes h_{J^{2,3}}^0\rangle
\langle f, h_{I^1}^0 \otimes h_{I^{2,3}} \rangle \langle g, h_{J^1} \otimes h_{J^{2,3}}^0\rangle.
\end{align*}
After handling this term completely, we have done most of the heavy lifting -- although there are some relevant details
we need to later discuss concerning the other main symmetries.
We now fix $I, J$. Writing 
$$
h_{I^1}^0=h_{I^1}^0-h_{J^1}^0+h_{J^1}^0 
\qquad \textup{and} \qquad
h_{J^{2,3}}^0=h_{J^{2,3}}^0-h_{I^{2,3}}^0+h_{I^{2,3}}^0,
$$ 
defining
$$
H_{I^1, J^1} = h_{I^1}^0 - h_{J^1}^0 \qquad \textup{and} \qquad H_{I^{2,3}, J^{2,3}} = h_{J^{2,3}}^0 - h_{I^{2,3}}^0,
$$
we have that
\begin{equation}\label{E:eq104}
\begin{split}
\langle f, & h_{I^1}^0 \otimes h_{I^{2,3}} \rangle \langle g, h_{J^1} \otimes h_{J^{2,3}}^0\rangle 
= \langle f, H_{I^1, J^1} \otimes h_{I^{2,3}} \rangle \langle g, h_{J^1} \otimes H_{I^{2,3}, J^{2,3}}\rangle \\
&+ \langle f, H_{I^1, J^1} \otimes h_{I^{2,3}} \rangle \langle g, h_{J^1} \otimes h_{I^{2,3}}^0\rangle 
+ \langle f, h_{J^1}^0 \otimes h_{I^{2,3}} \rangle \langle g, h_{J^1} \otimes H_{I^{2,3}, J^{2,3}}\rangle \\
&+  \langle f, h_{J^1}^0 \otimes h_{I^{2,3}} \rangle \langle g, h_{J^1} \otimes h_{I^{2,3}}^0\rangle.
\end{split}
\end{equation}

This gives us the splitting
$
\Sigma_{8} = \sum_{j=1}^4 \Sigma^j_{8}. 
$

\subsection{Shifts} The term $\Sigma^1_{8}$ will yield Zygmund shifts.
We write
\begin{equation}\label{E:eq10}
\Sigma^1_{8} = \sum_{\substack{n \in \Z^3 \\ n^1 \ne 0 \ne (n^2, n^3)}} \sum_{I \in \calD_{Z}} c_{I,n},
\end{equation}
where
\begin{equation}\label{E:eq85}
\begin{split}
c_{I, n} := \langle T(h_{I^1 \dot + n^1}^0& \otimes h_{I^{2,3}}), h_{I^1} \otimes h_{I^{2,3} \dotplus n^{2,3}}^0\rangle \\
& \times \langle f, H_{I^1, I^1 \dotplus n^1} \otimes h_{I^{2,3}} \rangle \langle g, h_{I^1 } \otimes H_{I^{2,3}, I^{2,3} \dotplus n^{2,3}}\rangle.
\end{split}
\end{equation}
Recall from \eqref{eq:kparent} that $\dotplus$ means that we are translating with a multiple of the side length: $I^1 \dotplus n^1 = I^1 + n^1\ell(I^1)$
and $I^{2,3} \dot + n^{2,3}:= (I^2 \dotplus n^2) \times (I^3 \dotplus n^3)$. 

Notice that if $0<\ell \in \Z$, then $\ell \in (2^{k-3}, 2^{k-2}]$ for some $k=2,3,4, \dots$ Therefore,
we can write the part of \eqref{E:eq10} where $n^m \not=0$ for $m=1,2,3$ as
\begin{equation}\label{E:eq99}
\sum_{\substack{n \in \Z^3 \\ n^m \ne 0, \  m=1,2,3}} \sum_{I \in \calD_{Z}} c_{I,n}
= \sum_{k^1,k^2,k^3=2}^\infty \sum_{\substack{n \in \Z^3 \\ |n^m| \in (2^{k^m-3}, 2^{k^m-2}]}} \sum_{I \in \calD_{Z}} c_{I,n}.
\end{equation}

Now, we add goodness, recall \eqref{eq:DefkGood}. This is based on \eqref{eq:gprob} and the independence properties explained after that.
Let $\calD_{0,Z}$ be the collection of Zygmund rectangles related to the standard lattice $\calD_0 \times \calD_0 \times \calD_0$.
If $I \in \calD_{0,Z}$ and $k = (k^1,k^2,k^3)$, $k^1,k^2,k^3 \in \{2,3, \dots\}$, the independence gives that
\begin{equation*}
\begin{split}
\E_\sigma c_{I + \sigma,n}
&=8 \E_{\sigma} 1_{\{(I  + \sigma) \in \calD_{Z}(k)\}} (\sigma)\E_\sigma c_{I + \sigma,n} 
=8\E_{\sigma} 1_{\{(I + \sigma) \in \calD_{Z}(k)\}} (\sigma) c_{I \dot + \sigma,n}.
\end{split}
\end{equation*}
Thus, the expectation of \eqref{E:eq99} can be written as
\begin{equation}\label{E:eq100}
\begin{split}
\E_\sigma & \sum_{k^1,k^2,k^3=2}^\infty \sum_{\substack{n \in \Z^3 \\ n^m \in (2^{k^m-3}, 2^{k^m-2}]}} \sum_{I \in \calD_{0, Z}} c_{I + \sigma,n} \\
&= 8 \E_\sigma \sum_{k^1,k^2,k^3=2}^\infty \sum_{\substack{n \in \Z^3 \\ |n^m| \in (2^{k^m-3}, 2^{k^m-2}]}} \sum_{I \in \calD_{Z}(k)} c_{I ,n}.
\end{split}
\end{equation}
Similarly, the expectations of the parts of \eqref{E:eq10} where $n^2=0$ or $n^3=0$ equal
\begin{equation}\label{E:eq102}
4 \E_\sigma \sum_{k^1,k^3=2}^\infty \sum_{\substack{n^1,n^3 \in \Z \\ |n^m| \in (2^{k^m-3}, 2^{k^m-2}]}} \sum_{I \in \calD_{Z}(k^1,0,k^3)} 
c_{I ,(n^1,0,n^3)}
\end{equation}
and
\begin{equation}\label{E:eq103}
4 \E_\sigma \sum_{k^1,k^2=2}^\infty \sum_{\substack{n^1,n^2 \in \Z \\ |n^m| \in (2^{k^m-3}, 2^{k^m-2}]}} 
\sum_{I \in \calD_{Z}(k^1,k^2,0)} 
c_{I ,(n^1,n^2,0)}.
\end{equation}

Consider \eqref{E:eq100}. It can be written as
\begin{equation}\label{E:eq101}
\begin{split}
8C \E_\sigma \sum_{k^1,k^2,k^3=2}^\infty& (|k|+1)\varphi(k) 
\sum_{K \in \calD_{\lambda}}\sum_{\substack{I \in \calD_{Z}(k) \\ I^{(k)}=K}} 
\sum_{\substack{n \in \Z^3 \\ |n^m| \in (2^{k^m-3}, 2^{k^m-2}]}}
\frac{c_{I,n}}{C(|k|+1)\varphi(k)},
\end{split}
\end{equation}
where $\lambda = 2^{-k^1-k^2+k^3}$ and $C$ is some suitable large constant as in the statement of Theorem \ref{E:thm1}. 
Here also $(I \dot+ n)^{(k)}=K$ by \eqref{eq:kparent}. Recalling the definition
of $c_{I,n}$ in \eqref{E:eq85} we have structurally proved the desired representation in terms of Zygmund shifts.
For the right normalization of the appearing constants, we have to show that
$$
|\langle T(h_{I^1 \dot + n^1}^0 \otimes h_{I^{2,3}}), h_{I^1} \otimes h_{I^{2,3} \dotplus n^{2,3}}^0\rangle| 
\le C (|k|+1)\varphi(k) \frac{|I|}{|K|},
$$
where $I$, $n$, $k$ and $K$ are as in \eqref{E:eq101}. These kernel estimates are proved separately in Section \ref{E:sec2}.

Of course, a similar reasoning applies to the terms in \eqref{E:eq102} and \eqref{E:eq103}. Therefore, we have shown that 
$\E_\sigma \Sigma^1_{8}$ can be written in the form required by the representation theorem.

\subsection{Paraproducts}\label{E:subsec10}
Recall that
\begin{align*}
\Sigma_{8}^4 &=  \sum_{\substack{I, J \in \calD_{Z} \\ \ell(I) = \ell(J)}} 
\langle T(h_{I^1}^0 \otimes h_{I^{2,3}}), h_{J^1} \otimes h_{J^{2,3}}^0\rangle
\langle f, h_{J^1}^0 \otimes h_{I^{2,3}} \rangle \langle g, h_{J^1} \otimes h_{I^{2,3}}^0\rangle.
\end{align*}
Summing first over for example $I^1$ gives
\begin{equation*}
 \sum_{J^1 \times I^{2,3} \in \calD_{Z}} 
 \sum_{\substack{ J^{2,3} \in \calD^{2,3}_{\ell(J^1)} \\ \ell(J^{2,3})=\ell(I^{2,3})}} 
\langle T(1 \otimes h_{I^{2,3}}), h_{J^1} \otimes h_{J^{2,3}}^0\rangle
\Big \langle f, \frac{1_{J_1}}{|J_1|} \otimes h_{I^{2,3}} \Big \rangle 
\langle g, h_{J^1} \otimes h_{I^{2,3}}^0\rangle
=0,
\end{equation*}
since the pairings $\langle T(1 \otimes h_{I^{2,3}}), h_{J^1} \otimes h_{J^{2,3}}^0\rangle$ are zero 
by the cancellation assumptions in Section \ref{E:subsec8}. Notice that formally summing further over $J^{2,3}$
would give
$$
\sum_{J^1 \times I^{2,3} \in \calD_{Z}} \langle T(1 \otimes h_{I^{2,3}}), h_{J^1} \otimes 1\rangle
\Big \langle f, \frac{1_{J_1}}{|J_1|} \otimes h_{I^{2,3}} \Big \rangle \Big \langle g, h_{J^1} \otimes \frac{1_{I^{2,3}}}{|I^{2,3}|} \Big \rangle
$$
which would be a full paraproduct of Zygmund type. 

Similarly, the terms $\Sigma_{8}^2$ and $\Sigma_{8}^3$ would yield some
partial paraproducts of Zygmund type, but they also vanish with our cancellation assumptions.

\subsection{Other main terms}
We consider the other terms $\Sigma_{i}$, $i \in \{1, \dots, 9\} \setminus\{8 \}$, from the
decomposition \eqref{eq:dec1}-\eqref{eq:dec5}. We demonstrate the structural identity with 
$$
\Sigma_{1}
= \sum_{\substack{I, J \in \calD_{Z} \\ \ell(I) = \ell(J)}}  
\langle Th_{I,Z} , h_{J,Z}\rangle
\langle f, h_{I,Z} \rangle \langle g, h_{J,Z}\rangle.
$$
Recall that $h_{I,Z}=h_{I^1} \otimes h_{I^{2,3}}$. Since all the appearing Haar functions are cancellative, we may directly write 
\begin{equation}\label{E:eq105}
\Sigma_{1}=\sum_{I \in \calD_{Z} }  \sum_{n \in \Z^3}c_{I,n},
\end{equation}
where 
\begin{equation}\label{E:eq106}
c_{I,n}=\langle Th_{I,Z} , h_{I \dot + n,Z}\rangle
\langle f, h_{I,Z} \rangle \langle g, h_{I \dot +n,Z}\rangle.
\end{equation}
The difference compared to the term $\Sigma_{8}$, which we treated above, is that in $\Sigma_{8}$ there are
non-cancellative Haar functions present. For this reason we did the splitting \eqref{E:eq104} with $\Sigma_{8}$,
and here we directly wrote the identity \eqref{E:eq105}.

After this we proceed similarly as we did with $\Sigma_{8}$.
The identity \eqref{E:eq105} corresponds to \eqref{E:eq10}, except that here all the cases where $n^m=0$ for some $m=1,2,3$ are allowed.
The parts of the sum where $n^m=0$ for some $m$ are handled similarly as in \eqref{E:eq102} and \eqref{E:eq103}. For example,
we have that
\begin{equation*}
\E_{\sigma}\sum_{I \in \calD_{Z} }  \sum_{n^3 \in \Z}c_{I,(0,0,n^3)}
=\E_{\sigma}\sum_{k^3=2}^\infty  \sum_{\substack{n^3 \in \Z \\ |n^3| \in (2^{k^3-3}, 2^{k^3-2}]}} \sum_{I \in \calD_{Z}(0,0,k^3) } 
c_{I,(0,0,n^3)}.
\end{equation*}
After organizing this correspondingly as in \eqref{E:eq101} it only remains to bound the shift coefficients, which is done in Section \ref{E:sec2}.

All the remaining terms $\Sigma_{j}$, $j \in \{1, \dots, 9\}\setminus\{1,8\}$ can be dealt with by
adapting the steps related to
 $\E_\sigma\Sigma_{1}$ and $\E_\sigma \Sigma_{8}$.
 Taking into account the estimates of the shift coefficients from section \ref{E:sec2}, this concludes the proof of
 Theorem \ref{E:thm1}.
 
 \end{proof}

\section{Estimates of the shift coefficients}\label{E:sec2}
This section is devoted to bounding the shift coefficients which appeared in the proof of Theorem \ref{E:thm1}.
We estimate coefficients of the form
\begin{equation}\label{E:eq80}
\langle T(h_{I^1 \dot +n^1}^0 \otimes h_{I^{2,3}}), h_{I^1 } \otimes h_{I^{2,3} \dotplus n^{2,3} }^0\rangle
\end{equation}
with all different values of $n \in \Z^3$. The coefficients coming from $\Sigma_{8}$ are of this form,
see \eqref{E:eq85}, but the coefficients related to the other main terms $\Sigma_{j}$ may have a different set of Haar functions.
For example, the term $\Sigma_{1}$ leads to coefficients of the form
$
\langle Th_{I,Z} , h_{I \dot + n,Z}\rangle,
$
see \eqref{E:eq106}.
However, these are all treated symmetrically. The point is that
related to the rectangles $I$ we have cancellative Haar functions. Related to $I^1\dot+n^1$ and
$I^{2,3}\dot+n^{2,3}$ the Haar function is either cancellative or non-cancellative, but this makes no difference.

One comment to avoid confusion. 
We'll estimate \eqref{E:eq80} also in the case that $n^1=0$ or $n^{2,3}=0$. As one sees for example from the 
treatment of $\Sigma_{8}$ above (see \eqref{E:eq10}), 
this kind of coefficients do not actually arise in the representation theorem. If $n^1=0$
then the Haar function related to $I^1\dot + n^1=I^1$ must be the cancellative Haar function $h_{I^1}$, and the same applies to $n^{2,3}$.
But since, as we already discussed, the form of these Haar functions does not affect the estimate, we simply estimate
coefficients of the form \eqref{E:eq80}.

Let us agree on the following notation. Let $n \in \Z$. If $n \not=0$, we define $k(n) \in \{2,3,4, \dots\}$ be the number such that
$|n| \in (2^{k-3},2^{k-2}]$. If $n=0$, define $k(n)=0$. If $n \in \Z^3$ we define $k(n)=(k(n^1),k(n^2),k(n^3))$. 
If $k=(k^1, k^2, k^3)$ recall that
\begin{equation}\label{E:eq116}
\varphi(k):= 
 2^{-k^1\alpha_1-k^2\min\{\alpha_{23},\theta\}-\max\{k^3-k^1-k^2, 0\}\theta},
\end{equation}
where $\theta$, $\alpha_1$ and $\alpha_{23}$  were defined in Section \ref{E:subsec6} and Section \ref{E:subsec5}.
All the coefficient estimates are included in the next proposition.

\begin{prop}\label{E:prop4}
Let $n \in \Z^3$. Then, with $k=k(n)=(k^1,k^2,k^3)$ there holds that
\begin{equation}\label{E:eq119}
|\langle T(h_{I^1\dot+ n^1}^0\otimes h_{I^{2,3}}), h_{I^1}\otimes h_{I^{2,3}\dot+n^{2,3}}^0\rangle|  
\lesssim (|k|+1)\varphi(k) \frac{|I|}{|K|},
\end{equation}
where $|K|=I^{(k)}$.
\end{prop}

We split the proof of Proposition \ref{E:prop4} into nine lemmas according to different values of $n$.

\begin{lem}[Separated/Separated]
The estimate \eqref{E:eq119} holds in the case that $|n^1|\ge 2$ and $\max\{|n^2|, |n^3|\}\ge 2$. 
\end{lem}

\begin{proof}
Suppose first that $\min\{|n^2|,|n^3|\} \ge 2$.
Using the zero integrals
$$
\int_{\R} h_{I^1} = 0 = \int_{\R^2} h_{I^{2,3}}
$$
we have that
\begin{equation*}
\begin{split} 
\langle &T(h_{I^1 \dot +n^1}^0  \otimes h_{I^{2,3}}), h_{I^1 } \otimes h_{I^{2,3} \dotplus n^{2,3} }^0\rangle 
=\iint K^*_{2,3}(x,y) h_{I\dot+n}^0(y) h_{I,Z}(x)  \ud x \ud y\\
&= \iint [K^*_{2,3}(x,y)-K^*_{2,3}((c_{I^1},x_{2,3}),y)-K^*_{2,3}((x_1,c_{I^{2,3}}),y) 
+K^*_{2,3}(c_I,y)] 
\\ & \hspace{9cm} h_{I\dot+n}^0(y) h_{I,Z}(x)  \ud x \ud y.
\end{split}
\end{equation*}
The H\"older estimate of the kernel \eqref{E:eq29} gives that the absolute value of the above integral is dominated by
\begin{align*}
 &\iint  2^{-k^1\alpha_1}(2^{-k^2}+  2^{-k^3})^{\alpha_{23}}  \frac {(2^{k^1+k^2-k^3}+2^{k^3-k^2-k^1})^{-\theta}}{2^{k^1+k^2+k^3}|I^3|^2}
h_{I\dot+n}^0(y) h^0_{I}(x)  \ud x \ud y\\
&= 2^{-k^1\alpha_1}(2^{-k^2}+2^{-k^3})^{\alpha_{23}}  \frac {(2^{k^1+k^2-k^3}+2^{k^3-k^2-k^1})^{-\theta}}{2^{k^1+k^2+k^3}}.
\end{align*}
Notice that $2^{-(k^1+k^2+k^3)}=|I|/|K|$. 
If $k^3 \ge k^1+k^2$, then also $k^3 \ge k^2$, and the above term is comparable with
$$
2^{-k^1\alpha_1}2^{-k^2\alpha_{23}} 2^{-(k^3-k^2-k^1)\theta}\frac{|I|}{|K|}
\le \varphi(k) \frac{|I|}{|K|}.
$$
Suppose that $k^3 < k^1+k^2$. If also $k^3 < k^2$, then the term is comparable with
$$
2^{-k^1\alpha_1}2^{-k^3\alpha_{23}} 2^{-(k^1+k^2-k^3)\theta}\frac{|I|}{|K|}
\le 2^{-k^1\alpha_1} 2^{-(k^1+k^2)\min \{ \alpha_{23},\theta\} }\frac{|I|}{|K|}
\le \varphi(k)\frac{|I|}{|K|}.
$$
Finally, if $k^3 < k^1+k^2$ and $k^3 \ge k^2$  we get
$$
2^{-k^1\alpha_1}2^{-k^2\alpha_{23}}  2^{-(k^1+k^2-k^3)\theta}\frac{|I|}{|K|}
\le 2^{-k^1\alpha_1}2^{-k^2\alpha_{23}} \frac{|I|}{|K|}
\le \varphi(k)\frac{|I|}{|K|}.
$$

We turn to the case $\min\{|n^2|,|n^3|\} < 2$.
Suppose first that $|n^2| < 2 \le |n^3|$.
We write 
\begin{align*}
\langle T(h_{I^1 \dot +n^1}^0 & \otimes h_{I^{2,3}}), h_{I^1 } \otimes h_{I^{2,3} \dotplus n^{2,3} }^0\rangle \\
&=\iint [K^*_{2,3}(x,y)-K^*_{2,3}((c_{I^1},x_{2,3}),y)] h_{I\dot+n}^0(y) h_{I,Z}(x)  \ud x \ud y.
\end{align*}
The mixed size and H\"older estimate \eqref{E:eq25} leads to
\begin{align*}
\iint & 2^{-k^1\alpha_1} 
\frac{\Big(\frac{2^{k^1-k^3}|x_2-y_2|}{|I^2|}+\frac{|I^2|}{2^{k^1-k^3}|x_2-y_2|}\Big)^{-\theta}}{2^{k^1}|I^1| |x_2-y_2| 2^{k^3}|I^3|}
h_{I\dot+n}^0(y) h^0_{I}(x)  \ud x \ud y \\
&= \frac{2^{-k^1(\alpha_1+1)-k^3}}{ |I^2|}
\iint \frac{\Big(\frac{2^{k^1-k^3}|x_2-y_2|}{|I^2|}+\frac{|I^2|}{2^{k^1-k^3}|x_2-y_2|}\Big)^{-\theta}}{|x_2-y_2| }
1_{I^2\dot+n^2}(y_2) 1_{I^2}(x_2)  \ud x_2 \ud y_2.
\end{align*}
Let $t=|I^2|/2^{k^1-k^3}$. By a change of variables the previous term equals
$$
2^{-k^1(\alpha_1+1)-k^3}2^{k^3-k^1}
\iint \frac{\Big(|x_2-y_2|+|x_2-y_2|^{-1}\Big)^{-\theta}}{|x_2-y_2| }
1_{t^{-1}(I^2\dot+n^2)}(y_2) 1_{t^{-1}I^2}(x_2)  \ud x_2 \ud y_2.
$$
Notice that $t^{-1} |I^2| = 2^{k^1-k^3}$. If $k^1 \ge k^3$, so that $t^{-1} |I^2| \ge 1$, we can estimate the integral by 
$t^{-1} |I^2|=2^{k^1-k^3}$. This gives
$$
2^{-k^1(\alpha_1+1)-k^3}2^{k^3-k^1}2^{k^1-k^3}
= 2^{-k^1\alpha_1} 2^{-k^1-k^3}
\lesssim \varphi (k) \frac{|I|}{|K|}.
$$
If $k^1 < k^3$, so that $t^{-1} |I^2| < 1$, we can estimate the integral by 
$(t^{-1}|I^2|)^{1+\theta}=2^{(k^1-k^3)(1+\theta)}$. This gives
$$
2^{-k^1(\alpha_1+1)-k^3}2^{k^3-k^1} 2^{(k^1-k^3)(1+\theta)}
=2^{-k^1\alpha_1}2^{-(k^3-k^1)\theta} 2^{-k^1-k^3}
\lesssim \varphi(k) \frac{|I|}{|K|}. 
$$

The final case $|n^2| \ge 2 > |n^3|$ is handled similarly. Let $t=2^{k^1+k^2}|I^3|$.
The mixed size and H\"older estimate \eqref{E:eq25} gives
\begin{equation*}
\begin{split}
&\iint  2^{-k^1\alpha_1} \frac{\Big(\frac{2^{k^1+k^2}|I^3|}{|x_3-y_3|} + \frac{|x_3-y_3|}{2^{k^1+k^2}|I^3|} \Big)^{-\theta}}
{2^{k^1+k^2}|I^3||x_3-y_3|} h^0_{I\dot+n}(y) h^0_{I}(x) \ud y \ud x \\
&=\frac{2^{-k^1\alpha_1-k^1-k^2}}{|I^3|} 2^{k^1+k^2}|I^3|
\iint \frac{(|x_3-y_3|^{-1} + |x_3-y_3| )^{-\theta}}
{|x_3-y_3|} 1_{t^{-1}(I^3\dot+n^3)}(y_3) 1_{t^{-1}I^3}(x_3) \ud y_3 \ud x_3.
\end{split}
\end{equation*} 
Since $t^{-1}|I^3| =2^{-k^1-k^2}<1$, the integral is dominated by $2^{-(k^1+k^2)(1+\theta)}$, and we end up with
\begin{equation*}
\frac{2^{-k^1\alpha_1-k^1-k^2}}{|I^3|} 2^{k^1+k^2}|I^3|
2^{-(k^1+k^2)(1+\theta)}
=2^{-k^1(\alpha_1+\theta)} 2^{-k^2\theta} 2^{-k^1-k^2}
\lesssim \varphi(k) \frac{|I|}{|K|}.\qedhere
\end{equation*}
\end{proof}

\begin{lem}[Separated/Adjacent]\label{lem:sa}
The estimate \eqref{E:eq119} holds in the case that $|n^1|\ge 2$ and $\max\{|n^2|, |n^3|\}= 1$. 
\end{lem}

\begin{proof}
We first record the estimate
\begin{equation}\label{E:eq117}
\int_{\R} \frac{\Big(\frac {t}{|u|} + \frac {|u|}{t}\Big)^{-\theta}}{t|u|} |f(u)| \ud u
\lesssim t^{-1} Mf(0), \quad t>0,
\end{equation}
where the implicit constant is independent of $t$.
Write
\begin{equation}\label{E:eq107}
\begin{split}
|\langle T(h_{I^1 \dot +n^1}^0  &\otimes h_{I^{2,3}}), h_{I^1 } \otimes h_{I^{2,3} \dotplus n^{2,3} }^0\rangle |\\
&=\Big|\iint [K^*_{2,3}(x,y)-K^*_{2,3}((c_{I^1},x_{2,3}),y)] h_{I\dot+n}^0(y) h_{I,Z}(x)  \ud x \ud y\Big|.
\end{split}
\end{equation}
The mixed size and H\"older estimate \eqref{E:eq25} gives  
\begin{equation}\label{E:eq118}
2^{-k^1\alpha_1}|I^1| \iint \frac{\Big(\frac{2^{k^1}|I^1||x_2-y_2|}{|x_3-y_3|} +\frac{|x_3-y_3|}{2^{k^1}|I^1||x_2-y_2|} \Big)^{-\theta}}
{2^{k^1}|I^1||x_2-y_2||x_3-y_3|}h^0_{I^{2,3} \dot+ n^{2,3}}(y_{2,3}) h^0_{I^{2,3}}(x_{2,3}) \ud x_{2,3} \ud y_{2,3}.
\end{equation}

Suppose that $|n^2|=1$ and $|n^3|\le 1$. Then we integrate first with respect to $y_3$ and
use \eqref{E:eq117} to have that \eqref{E:eq118} is dominated by
\begin{equation*}
\begin{split}
&2^{-k^1\alpha_1}|I^1|
\iint \frac{h^0_{I^2 \dot +n^2}(y_2) h^0_{I^2}(x_2)}{2^{k^1}|I^1||x_2-y_2|} \ud x_2 \ud y_2
\int M(h^0_{I^3 \dot + n^3})(x_3) h^0_{I^3}(x_3) \ud x_3 \\
&\lesssim 2^{-k^1\alpha_1}|I^1| \frac{1}{2^{k^1}|I^1|}
\lesssim \varphi(k) \frac{|I|}{|K|}.
\end{split}
\end{equation*}
If $(|n^2|, |n^3|)=(0,1)$, then we first integrate with respect to $y_2$ to have that \eqref{E:eq118}
is dominated by
\begin{equation*}
\begin{split}
&2^{-k^1\alpha_1}|I^1|
\iint \frac{h^0_{I^3 \dot +n^3}(y_3) h^0_{I^3}(x_3)}{2^{k^1}|I^1||x_3-y_3|} \ud x_3 \ud y_3
\int M(h^0_{I^2 })(x_2) h^0_{I^2}(x_2) \ud x_2 \\
& \lesssim 2^{-k^1\alpha_1}|I^1| \frac{1}{2^{k^1}|I^1|}
\lesssim \varphi(k) \frac{|I|}{|K|}.\qedhere
\end{split}
\end{equation*}
\end{proof}

\begin{lem}[Separated/Identical]
The estimate \eqref{E:eq119} holds in the case that $|n^1|\ge 2$ and $n^2=n^3=0$. 
\end{lem}

\begin{proof}
Define
$$
\ch(I^{2,3})= \{ (I^2)' \times (I^3)' \colon (I^m)' \in \ch(I^m)\}.
$$
We split the pairing $\langle T(h_{I^1 \dot +n^1}^0  \otimes h_{I^{2,3}}), h_{I^1 } \otimes h_{I^{2,3}  }^0\rangle$ as 
\[
\sum_{(I^{2,3})',(I^{2,3})'' \in \ch(I^{2,3}) }
\langle T(h_{I^1 \dotplus n^1 }^0 \otimes h_{I^{2,3} }1_{(I^{2,3})'}), h_{I^1 } \otimes h_{I^{2,3}}^0 1_{(I^{2,3})''}\rangle.
\]
Every term in the sum is estimated separately.
If $(I^{2,3})'\neq (I^{2,3})''$, the situation is essentially as in Lemma \ref{lem:sa}. 
So we only need to consider the case $(I^{2,3})'= (I^{2,3})''$. 
Then the partial kernel representation \eqref{E:eq68} gives 
that
\begin{equation*}
\begin{split}
\langle T&(h_{I^1 \dotplus n^1 }^0 \otimes h_{I^{2,3} }1_{(I^{2,3})'}), h_{I^1 } \otimes h_{I^{2,3}}^0 1_{(I^{2,3})'}\rangle \\
&=\pm |I^{2,3}|^{-1} \iint K_{(I^{2,3})'}(y_1,x_1) h_{I^1 \dotplus n^1 }^0(y_1) h_{I^1 }(x_1) \ud y_1 \ud x_1.
\end{split}
\end{equation*}
The kernel $K_{(I^{2,3})'}$ is assumed to satisfy $\| K_{(I^{2,3})'}\|_{CZ_{\alpha_1}(\R)} \lesssim |(I^{2,3})'| \le  |I^{2,3}|$, 
see \eqref{E:eq110}. Thus, by the zero average of $ h_{I^1 }$ and the H\"older estimate of $K_{(I^{2,3})'}$ (see \eqref{E:eq109}), 
we get that the last term can be estimated by
\begin{equation*}
|I^{2,3}|^{-1} \iint |I^{2,3}| \frac{1}{2^{k^1\alpha_1}} \frac{1}{2^{k^1}|I^1|}  h_{I^1 \dotplus n^1 }^0(y_1) h^0_{I^1 }(x_1) \ud y_1 \ud x_1
= \frac{1}{2^{k^1(1+\alpha_1)}} 
= \varphi(k)\frac{|I|}{|K|}.\qedhere
\end{equation*}
\end{proof}

\begin{lem}[Adjacent/Separated]\label{lem:as}
The estimate \eqref{E:eq119} holds in the case that $|n^1|=1$ and $\max\{|n^2|, |n^3|\}\ge 2$. 
\end{lem}

\begin{proof}
We first consider  the case $|n^2| \ge 2 >|n^3|$. By the size estimate of the kernel \eqref{E:eq30} we have 
\begin{equation}\label{E:eq114}
\begin{split}
&|\langle  T(h_{I^1 \dot +n^1}^0  \otimes h_{I^{2,3}}), h_{I^1 } \otimes h_{I^{2,3} \dotplus n^{2,3} }^0\rangle |\\
& \lesssim \frac{|I^2|}{|I^{1,3}|} 
\iint \frac{\Big(\frac{|x_1-y_1|2^{k^2}|I^2|}{|x_3-y_3|}+\frac{|x_3-y_3|}{|x_1-y_1|2^{k^2}|I^2|}\Big)^{-\theta}}{|x_1-y_1|2^{k^2}|I^2||x_3-y_3|} 
1_{I^{1,3}\dot+ n^{1,3}}(y_{1,3})1_{I^{1,3}}(x_{1,3})\ud x_{1,3} \ud y_{1,3}. \\
\end{split}
\end{equation}

Suppose $L \subset \R$ is an interval with $|L| \ge 1$. 
Since $|n^1|=1$ the intervals $L^1$ and $L^1 \dot + n^1$ are adjacent. 
There holds that
\begin{equation}\label{E:eq113}
\begin{split}
\iint \frac{(|x_1-y_1|+|x_1-y_1|^{-1})^{-\theta}}{|x_1-y_1|} 
&1_{L^1\dot+ n^1}(y_1) 1_{L^1}(x_1) \ud y_1 \ud x_1 \\
&\lesssim \begin{cases} 1+ \log(|L|), \quad &\theta=1,\\
|L|^{1-\theta}, \quad &\theta\in (0,1).
\end{cases}
\end{split}
\end{equation}

Now fix $x_3 \in I^3$ and $y_3 \in I^3 \dot + n^3$ and consider the remaining integral with respect to $x_1$ and $y_1$ in 
\eqref{E:eq114}. 
Let $r=2^{-k^2}|I^2|^{-1} |x_3-y_3|$ and notice that
$r^{-1} |I^1| \ge 1 $. A change of variables and \eqref{E:eq113} give that
\begin{equation}\label{E:eq111}
\begin{split}
& \iint \frac{\Big(\frac{|x_1-y_1|2^{k^2}|I^2|}{|x_3-y_3|}+\frac{|x_3-y_3|}{|x_1-y_1|2^{k^2}|I^2|}\Big)^{-\theta}}{|x_1-y_1|2^{k^2}|I^2||x_3-y_3|} 
1_{I^{1}\dot+ n^{1}}(y_{1})1_{I^{1}}(x_{1})\ud x_1 \ud y_1\\
&= \frac{r}{2^{k^2}|I^2||x_3-y_3|}
\iint \frac{\Big(|x_1-y_1|+\frac{1}{|x_1-y_1|}\Big)^{-\theta}}{|x_1-y_1|} 
1_{r^{-1}(I^{1}\dot+ n^{1})}(y_{1})1_{r^{-1}I^{1}}(x_{1})\ud x_1 \ud y_1 \\
&\lesssim \frac {1+\log (r^{-1}|I^1|)+(r^{-1}|I^1|)^{1-\theta}}{2^{2k^2}|I^2|^2}
=\frac {1+k^2\log 2+\log\Big(\frac{|I^3|}{|x_3-y_3|}\Big)+\Big(\frac{2^{k^2}|I^3|}{|x_3-y_3|}\Big)^{1-\theta}}{2^{2k^2}|I^2|^2}.
\end{split}
\end{equation}
A substitution of \eqref{E:eq111} into \eqref{E:eq114} leads to
\begin{align*}
\frac{|I^2|}{|I^{1,3}|}\iint \frac {k^2+\Big|\log \Big(\frac{|I^3|}{|x_3-y_3|}\Big)\Big|
+\Big(\frac{2^{k^2}|I^3|}{|x_3-y_3|}\Big)^{1-\theta}}{2^{2k^2}|I^2|^2}& 1_{I^3 \dot + n^3}(y_3) 1_{I^3}(x_3) \ud y_{3} \ud x_{3}\\&
\lesssim \frac{k^2}{2^{k^2(1+\theta)}} 
\lesssim k^2 \varphi(k)\frac{|I|}{|K|}.
\end{align*}

Assume then that $|n^2| \ge |n^3| \ge 2$. We can use the mixed size and H\"older estimate \eqref{E:eq28} which leads to
\begin{align*}
& 2^{-k^3\alpha_{23}}\frac{|I^{2,3}|}{|I^1|} 
\iint \frac{\Big(\frac{|x_1-y_1|2^{k^2}|I^2|}{2^{k^3}|I^3|} + \frac{2^{k^3}|I^3|}{|x_1-y_1|2^{k^2}|I^2|}\Big)^{-\theta}}
{|x_1-y_1|2^{k^2}|I^2|2^{k^3}|I^3|}1_{I^1 \dot + n^1}(y_1) 1_{I^1}(x_1)\ud x_1 \ud y_1 \\
&=2^{-2k^2-k^3\alpha_{23}}
\iint \frac{\Big(|x_1-y_1| + \frac{1}{|x_1-y_1|}\Big)^{-\theta}}
{|x_1-y_1|}1_{r^{-1}(I^1 \dot + n^1)}(y_1) 1_{r^{-1}I^1}(x_1)\ud x_1 \ud y_1,
\end{align*}
where $r=2^{k^3}|I^3|/(2^{k^2}|I^2|)$. 
Since $r^{-1}|I^1|=2^{k^2-k^3} \ge 1$, we have by 
\eqref{E:eq113} that this is dominated by
\begin{align*}
2^{-2k^2}2^{-k^3\alpha_{23}}(1+k^2-k^3+ 2^{(k^2-k^3)(1-\theta)})
&\lesssim k^2 2^{-k^2(1+\theta)-k^3(\alpha_{23}+1-\theta)}\\
&\lesssim k^2 2^{-k^2\theta-k^3(\alpha_{23}-\theta)}\frac{|I|}{|K|}.
\end{align*}
If $\alpha_{23} \ge \theta$, then 
$
2^{-k^2\theta-k^3(\alpha_{23}-\theta)} 
\le 2^{-k^2\theta} \lesssim \varphi(k). 
$
If $\alpha_{23} < \theta$, then
$$
2^{-k^2\theta-k^3(\alpha_{23}-\theta)} 
=2^{-k^2\alpha_{23}-(k^2-k^3)(\theta-\alpha_{23})} 
\le 2^{-k^2\alpha_{23}} \lesssim \varphi(k).
$$
In any case, this finishes the estimate under the assumption $|n^2| \ge |n^3| \ge 2$.

Assume then that $|n^2|< 2 \le  |n^3|$. The size estimate of the kernel \eqref{E:eq30} gives 
\begin{equation}
\begin{split}
 \frac{|I^3|}{|I^{1,2}|} & \iint \frac{1_{I^{1,2}\dot+ n^{1,2}}(y_{1,2})1_{I^{1,2}}(x_{1,2})}{ (2^{k^3}|I^3|)^{1+\theta} (|x_1-y_1||x_2-y_2|)^{1-\theta}}
  \ud y_{1,2} \ud x_{1,2}\\
&\lesssim 2^{-k^3(1+\theta)}
\sim 2^{-k^3\theta}\frac{|I|}{|K|}
\sim \varphi(k)\frac{|I|}{|K|}.
\end{split}
\end{equation}

Finally, assume that
$2 \le |n^2| < |n^3|$. The mixed size and H\"older estimate \eqref{E:eq28} gives
\begin{align*}
2^{-k^2\alpha_{23}} & \frac{|I^{2,3}|}{|I^1|} 
\iint \frac{ 1_{I^1\dot+n^1}(y_1) 1_{I^1}(x_1)   }{(2^{k^2}|I^2|)^{1-\theta}(2^{k^3}|I^3|)^{1+\theta}|x_1-y_1|^{1-\theta}}
 \ud y_1 \ud x_1 \\
& \lesssim \frac{2^{-k^2\alpha_{23}}}{2^{k^2(1-\theta)}2^{k^3(1+\theta)}}
\lesssim 2^{-k^2\alpha_{23}-(k^3-k^2)\theta} \frac{|I|}{|K|} 
\lesssim \varphi(k)\frac{|I|}{|K|}.\qedhere
\end{align*}
\end{proof}

\begin{lem}[Adjacent/Adjacent]\label{lem:aa}
The estimate \eqref{E:eq119} holds in the case that $|n^1|=1$ and $\max\{|n^2|, |n^3|\}=1$. 
\end{lem}

\begin{proof}
We first consider the case $|n^2|=1$ and $|n^3|\le 1$. The size estimate of the kernel
\eqref{E:eq30} leads to 
\begin{equation}\label{eq:E1204}
\begin{split}
\iint \frac{\Big(\frac{|x_1-y_1||x_2-y_2|}{|x_3-y_3|} + \frac{|x_3-y_3|}{|x_1-y_1||x_2-y_2|}\Big)^{-\theta} }{|x_1-y_1||x_2-y_2||x_3-y_3|}
h_{I\dot+n}^0(y) h_I^0(x) \ud x \ud y.
\end{split}
\end{equation}
We first integrate with respect to $y_3$ and use \eqref{E:eq117}. This gives
\begin{equation*}
\prod_{i=1}^2 \iint \frac {h_{I^i\dot+ n^i}^0(y_i) h_{I^i}^0(x_i)}{|x_i-y_i|} \ud y_i \ud x_i
\int M(h^0_{I^3\dot + n^3})(x_3) h^0_I(x_3) \ud x_3
\lesssim 1.
\end{equation*}

The remaining case $n^2=0$ and $|n^3|=1$ is done in the same way. One first uses \eqref{E:eq117} with respect to $y_2$ and 
the proceeds as above.
\end{proof}

\begin{lem}[Adjacent/Identical]\label{E:lem2}
The estimate \eqref{E:eq119} holds in the case that $|n^1|=1$ and $n^2=n^3=0$. 
\end{lem}

\begin{proof}
We first split the paring $\langle T(h_{I^1 \dot +n^1}^0  \otimes h_{I^{2,3}}), h_{I^1 } \otimes h_{I^{2,3}  }^0\rangle$ as 
\[
\sum_{(I^{2,3})',(I^{2,3})'' \in \ch(I^{2,3})}
\langle T(h_{I^1 \dotplus n^1 }^0 \otimes h_{I^{2,3} }1_{(I^{2,3})'}), h_{I^1 } \otimes h_{I^{2,3}}^0 1_{(I^{2,3})''}\rangle.
\]
Every term in the sum is estimated separately.
If $(I^{2,3})'\neq (I^{2,3})''$, then the situation is essentially as in Lemma \ref{lem:aa}. 
So we only need to consider the case $(I^{2,3})'= (I^{2,3})''$. 
Then the partial kernel representation \eqref{E:eq68} gives that the pairing in question can be written as
$$
\pm |I^{2,3}|^{-1} \iint K_{(I^{2,3})'}(y_1,x_1) h_{I^1 \dotplus n^1 }^0(y_1) h_{I^1 }(x_1) \ud y_1 \ud x_1.
$$
The kernel $K_{(I^{2,3})'}$ is assumed to satisfy $\| K_{(I^{2,3})'}\|_{CZ_{\alpha_1}(\R)} \lesssim |(I^{2,3})'| \le  |I^{2,3}|$, 
see \eqref{E:eq110}.
By the size estimate of $K_{(I^{2,3})'}$ (see \eqref{E:eq115}) the above is controlled by 
\begin{equation*}
 |I^{2,3}|^{-1} \iint \frac{|I^{2,3}|}{|x_1-y_1|} h_{I^1 \dotplus n^1 }^0(y_1) h_{I^1 }(x_1) \ud y_1 \ud x_1\lesssim 1.\qedhere
\end{equation*}
\end{proof}

\begin{lem}[Identical/Separated]
The estimate \eqref{E:eq119} holds in the case that $n^1=0$ and $\max\{|n^2|, |n^3|\}\ge 2$. 
\end{lem}

\begin{proof}
We may split the pairing as 
\[
\sum_{(I^1)', (I^1)''\in {\rm ch}(I^1)}\langle T(h_{I^1}^01_{(I^1)'}\otimes h_{I^{2,3}}), h_{I^1}1_{(I^1)''}\otimes h_{I^{2,3}\dot+n^{2,3}}^0\rangle.
\]
If $(I^1)'\neq (I^1)''$, this can essentially be reduced to Lemma \ref{lem:as}. So we focus on the case $(I^1)'= (I^1)''$. 
By the partial kernel representation we can write the pairing as 
\begin{equation}\label{eq:E12041}
\pm |I^1|^{-1}\iint K_{(I^1)'}(x_{2,3}, y_{2,3}) h_{I^{2,3}}(y_{2,3}) h_{I^{2,3}\dot+ n^{2,3}}^0(x_{2,3})\ud y_{2,3}\ud x_{2,3}.
\end{equation}

Suppose that $|n^3| < 2 \le |n^2|.$ By the size estimate of $K_{(I^1)'}$, see \eqref{E:eq62}, 
there holds that
\begin{equation*}
\begin{split}
|K_{(I^1)'}(x_{2,3}, y_{2,3})|
&\lesssim \frac{|I^1|}{2^{k^2}|I^2||x_3-y_3|} 
\Big(\frac{|x_3-y_3|}{|I^1|2^{k^2}|I^2|}
+\frac{|I^1|2^{k^2}|I^2|}{|x_3-y_3|}\Big)^{-\theta}\\
& \le \frac{|I^1|}{|I^1|^{\theta}(2^{k^2}|I^2|)^{1+\theta}}\cdot \frac 1{|x_3-y_3|^{1-\theta}}.
\end{split}
\end{equation*}
Therefore, the term in \eqref{eq:E12041} is controlled by 
\begin{equation*}
\begin{split}
\frac 1{|I^1|^{\theta}(2^{k^2}|I^2|)^{1+\theta}} & \frac{|I^2|}{|I^3|}
\iint \frac {1_{I^3}(y_3) 1_{I^3\dot + n^3}(x_3)}{|x_3-y_3|^{1-\theta}}\ud y_3 \ud x_3 \\
&\lesssim  2^{-k^2(1+\theta)}
\sim 2^{-k^2\theta}\frac{|I|}{|K|}
\le   \varphi(k)\frac{|I|}{|K|}.
\end{split}
\end{equation*}
Similarly, when $|n^2|<2 \le |n^3|$ the coefficient is controlled by 
$2^{-k^3\theta}|I|/|K| \sim \varphi(k)|I|/|K|$. 

Assume then that $\min \{|n^2|, |n^3|\}\ge 2$. Now we can use the cancellation of $h_{I^{2,3}}$ and the H\"older estimate of the kernel $ K_{(I^1)'}$, see \eqref{E:eq63}. We get that \eqref{eq:E12041} is controlled by 
\begin{align*}
&|I^1|^{-1}|I^{2,3}|(2^{-k^2}+2^{-k^3})^{\alpha_{23}}
\frac{|I^1|}{2^{k^2+k^3}|I^2||I^3|}2^{-|k^3-k^2|\theta}\\
&\hspace{1cm}= (2^{-k^2}+2^{-k^3})^{\alpha_{23}}2^{-|k^3-k^2|\theta} \frac 1{2^{k^2+k^3}}
\le \varphi(k)\frac{|I|}{|K|}.\qedhere
\end{align*}
\end{proof}

\begin{lem}[Identical/Adjacent]\label{E:lem4}
The estimate \eqref{E:eq119} holds in the case that $n^1=0$ and $\max\{|n^2|, |n^3|\}=1$. 
\end{lem}

\begin{proof}
As in the previous case we can reduce the estimate to bounding the integral \eqref{eq:E12041}. 
The size estimate \eqref{E:eq62} of $K_{(I^1)'}$ gives that the integral dominated by
\begin{align*}
\iint \frac{1}{|x_2-y_2||x_3-y_3|} 
\Big(\frac{|x_3-y_3|}{|I^1||x_2-y_2|}+\frac{|I^1||x_2-y_2|}{|x_3-y_3|}\Big)^{-\theta} 
 h_{I^{2,3}}^0(y_{2,3}) h_{I^{2,3}\dot+ n^{2,3}}^0(x_{2,3})\ud y_{2,3}\ud x_{2,3}.
\end{align*}
Suppose that $|n^2|=1$. We use \eqref{E:eq117} with respect to $y_3$ which gives
\begin{equation*}
\iint \frac{h_{I^2}^0(y_2) h_{I^2\dot+ n^2}^0(x_2)}{|x_2-y_2|} \ud x_2 \ud y_2
\int M(h_{I^3}^0) (x_3) h_{I^3\dot+ n^3}^0(x_3) \ud x_3
\lesssim 1.
\end{equation*}
The case $|n^3|=1$ is estimated similarly.
\end{proof}

\begin{lem}[Identical/Identical]
The estimate \eqref{E:eq119} holds in the case that $n=0$.
\end{lem}

\begin{proof}
We split the pairing into terms of the form
$$
\langle T(h_{I^1  }^0 1_{(I^1)'} \otimes h_{I^{2,3} }1_{(I^{2,3})'}), 
h_{I^1 }1_{(I^1)''} \otimes h_{I^{2,3}}^0 1_{(I^{2,3})'' }\rangle,
$$
where $(I^1)', (I^1)'' \in \ch (I^1)$ and $(I^{2,3})',(I^{2,3})'' \in \ch(I^{2,3})$. 
The case $(I^1)' \not =(I^1)''$ and $(I^{2,3})' \not=(I^{2,3})''$ is 
essentially contained in Lemma \ref{lem:aa}, the case $(I^1)' \not =(I^1)''$ and $(I^{2,3})' =(I^{2,3})''$ 
in Lemma \ref{E:lem2} and the case $(I^1)' =(I^1)''$ and $(I^{2,3})' \not =(I^{2,3})''$ in Lemma \ref{E:lem4}.
Finally, the estimate in the case $(I^1)' =(I^1)''$ and $(I^{2,3})'  =(I^{2,3})''$ follows immediately from 
the weak boundedness property 
$$
|\langle T(1_{(I^1)' \times (I^{2,3})'}), 1_{(I^1)' \times (I^{2,3})'} \rangle | 
\lesssim |(I^1)'| |(I^{2,3})'|,
$$
see \eqref{E:eq74}.
\end{proof}

\end{document}